\documentclass{amsart}
\usepackage[utf8]{inputenc}
\usepackage[left=3cm,right=3cm,top=3cm,bottom=3cm]{geometry}

\usepackage{enumitem}
\usepackage{longtable,booktabs}
\usepackage{graphicx}
\usepackage{wrapfig}
\usepackage{subcaption}
\usepackage{listings}
\usepackage{xcolor}
\usepackage[foot]{amsaddr}
\definecolor{mygreen}{RGB}{28,132,80} 
\definecolor{mylilas}{RGB}{170,55,241}
\definecolor{myblue}{RGB}{20,30,171}
\definecolor{myred}{RGB}{200,10,30}
\usepackage{microtype}
\usepackage{multirow}
\usepackage{tikz}
\usepackage{array}
\newcounter{rowcount}
 \usepackage{doi}
 \usepackage{upquote} 
 
\usepackage{amsmath}
\usepackage{amsfonts}
\usepackage{amssymb}
\usepackage{amscd}
\usepackage{amstext}
\usepackage{mathrsfs}
\usepackage{amsbsy}
\usepackage{amsopn}
\usepackage{amsthm}
\usepackage{thmtools}
\usepackage{amsxtra}
\usepackage{float}
\usepackage{bbm}
\usepackage{url}
\usepackage{fancyvrb} 

\usepackage[numbers,square]{natbib}

\usepackage{hyperref}
\hypersetup{
hypertexnames=false,
colorlinks=true,
linkcolor=myred,
citecolor=blue}

\newtheoremstyle{mytheoremstyle} 
  {12pt}                    
  {\topsep}                    
  {\slshape}                   
  {}                           
  {\bfseries}                   
  {.}                          
  {.5em}                       
  {}  

\newtheoremstyle{mydefstyle} 
  {12pt}                    
  {12pt}                    
  {}                   
  {}                           
  {\bfseries}                   
  {.}                          
  {.5em}                       
  {}  

  \newtheoremstyle{plainsl}%
  	{}
  	{\topsep}
  	{\slshape} 
  	{}
  	{\normalfont\bfseries}
  	{.}
  	{ }
  	{}

\theoremstyle{plainsl}
\newtheorem{thm}{Theorem}[section]
\newtheorem*{thm*}{Theorem}
\newtheorem{prop}[thm]{Proposition}

\newtheorem{cor}[thm]{Corollary}
\newtheorem{lem}[thm]{Lemma}
\newtheorem*{lem*}{Lemma}

\theoremstyle{remark}
\newtheorem{rem}[thm]{Remark}
\newtheorem{prob}[thm]{Problem}

\theoremstyle{mydefstyle}

\newtheorem{ex}[thm]{Example}
\newtheorem*{note*}{Remark}

\newcommand{\Z}{\mathbb{Z}}

\newcommand{\Q}{\mathbb{Q}}
\newcommand{\FF}{\mathbb{F}}

\newcommand{\eps}{\varepsilon}

\renewcommand{\phi}{\varphi}

\newcommand{\sm}{\backslash}

  \DeclareMathOperator\SL{SL}
  \DeclareMathOperator\GL{GL}
  \DeclareMathOperator\PSL{PSL}
  \DeclareMathOperator\GU{GU}
  
  \DeclareMathOperator\PGL{PGL}

\renewcommand\qed{%
	\ifmmode\eqno\sqr53
	\else\nolinebreak\ \hfill\sqr53\medbreak\fi}
\renewcommand\proof{\noindent\textsl{Proof. }}
\newcommand\sqr[2]{{\vbox{\hrule height.#2pt
    \hbox{\vrule width.#2pt height#1pt \kern#1pt
        \vrule width.#2pt}\hrule height.#2pt}}}

\DeclareMathOperator\Sym{Sym}
\DeclareMathOperator\Aut{Aut}
\DeclareMathOperator\Inn{Inn}
\DeclareMathOperator\Out{Out}
\DeclareMathOperator\Cay{Cay}

\DeclareMathOperator\soc{soc}

\definecolor{myorange}{RGB}{204,102,0}

\title{ Perfect State Transfer and Cayley Presentations}

  \author{Arnbj\"org Soff\'ia \'Arnad\'ottir}
  \address{Arnbj\"org Soff\'ia \'Arnad\'ottir: Division of Mathematics,
  Science Institute, University of Iceland, Dunhaga 3, 107 Reykjav\'ik,
  Iceland}
  \email{asoffia@hi.is}

  \author{Krystal Guo}
  \address{Krystal Guo: Korteweg--de Vries Institute for Mathematics,
  University of Amsterdam, and QuSoft, Amsterdam, The Netherlands}
  \email{k.guo@uva.nl}

\begin{document}

\renewcommand{\itshape}{\slshape}
\VerbatimFootnotes

  \begin{abstract}
    We study perfect state transfer on Cayley graphs from the point of view that state transfer is a property of a graph and not of a group.  This paper is a bridge between the classical question about isomorphic Cayley
    graphs of non-isomorphic groups and quantum walks on graphs.
  
    We show that a Cayley graph of a  group with an abelian subgroup of index two is a Cayley graph of an    abelian group under any one of three hypotheses, two drawn from the theory of isomorphic Cayley graphs.  A statement of the same kind
    holds for extraspecial groups: every Cayley graph of an extraspecial $p$-group of order $p^{2n+1}$ with a conjugacy-closed connection set   is a Cayley graph of $\Z_p^{2n+1}$.  From these results we deduce that
    every explicit construction of perfect state transfer in the six papers we survey, on dihedral, dicyclic, generalized dihedral,
    $V_{8n}$ and extraspecial $2$-groups, is a non-abelian presentation of an  abelian Cayley graph.  Moreover, we show that a non-abelian group with an abelian
    subgroup of index two admits a connected Cayley graph with perfect
    state transfer if and only if its order is divisible by four.
  
    Genuinely non-abelian examples do exist.  We prove that, for every odd
    prime power $q\ge5$, the $\SL(2,q)$ graph of Pantangi and Sin, which
    they showed to admit perfect state transfer, is a Cayley graph of no
    abelian group; to our knowledge, this is the first infinite family of
    Cayley graphs with perfect state transfer provably admitting no
    abelian Cayley presentation.  The claim rests on a theorem that never
    mentions state transfer: for $q\ge5$ the $\SL(2,q)$ graph is a
    lexicographic double of a Cayley graph of $\PSL(2,q)$ whose
    automorphism group, by Praeger's inclusion theorem for primitive
    groups of simple diagonal type, has socle $\PSL(2,q)\times\PSL(2,q)$
    and contains no abelian regular subgroup.  We also construct an
    infinite family of Cayley graphs with peak state transfer, each with
    amount $16/225$, determine all regular subgroups of the automorphism
    group of every member, and show that each group over which a member is
    a Cayley graph surjects onto the Frobenius group of order $20$; in
    particular, none of them is abelian.          
  
    An appendix records a census of the connected vertex-transitive graphs
    with perfect state transfer on at most $30$ vertices, together with
    all of their Cayley groups; the smallest examples with no abelian
    Cayley presentation appear at $24$ vertices, the smallest non-Cayley
    example at $30$, and a partial scan at $32$ yields the first known
    non-abelian graphical regular representations with perfect state
    transfer.






\end{abstract}

\maketitle

\section{Introduction}
One graph can be a Cayley graph of several non-isomorphic groups.  By Sabidussi's theorem \cite{Sab1958}, a graph $X$ is a Cayley graph of $G$ if and only if $\Aut(X)$ contains a regular subgroup isomorphic to $G$.  The same automorphism group may contain non-isomorphic regular subgroups, so the same graph may be a Cayley graph of several non-isomorphic groups; for example both graphs in Figure~\ref{fig:intro-universal} are Cayley graphs of every group of that order.  In particular, a non-abelian Cayley presentation need not be an intrinsically non-abelian feature of the graph: its automorphism group may also admit an abelian regular subgroup.  This paper takes that distinction as the starting point for the study of state transfer on Cayley graphs.

Let $A$ be the adjacency matrix of a graph $X$, and let $U(t)=\exp(itA)$ be the transition matrix of the continuous-time quantum walk on $X$.  The graph admits \textsl{perfect state transfer} from vertex $u$ to vertex $v$ at time $\tau$ if $\lvert U(\tau)_{v,u}\rvert=1$.  Perfect state transfer was introduced by Bose \cite{bose2003} and developed further by Christandl et al.\ \cite{christandl2004} and has since generated a substantial literature in algebraic graph theory.  We also consider \textsl{peak state transfer}, introduced in \cite{CouGuoSch2025}, in which the transfer probability between a pair of vertices attains its maximum possible value at some time.  Both notions depend only on $A$, and hence only on the graph $X$, not on a chosen presentation $X=\Cay(G,S)$.

\begin{figure}[t]
\centering
\begin{tikzpicture}[
  edge/.style={draw=black!40, line width=0.5pt},
  pst/.style={draw=myorange, line width=1.1pt},
  vx/.style={circle, draw=black, fill=white, inner sep=1.5pt},
  lb/.style={font=\small}]
  \begin{scope}[xshift=0cm]
    \foreach \k in {0,...,7}{ \coordinate (a\k) at ({90-45*\k}:1.9); }
    \foreach \k in {0,...,7}{ \foreach \d in {1,3}{
      \pgfmathtruncatemacro{\l}{mod(\k+\d,8)} \draw[edge] (a\k)--(a\l); } }
    \foreach \k in {0,...,3}{ \pgfmathtruncatemacro{\l}{\k+4} \draw[pst] (a\k)--(a\l); }
    \foreach \k in {0,...,7}{ \node[vx] at (a\k) {}; }
    \foreach \k in {0,...,7}{ \node[lb] at ({90-45*\k}:2.28) {$\k$}; }
    \node at (0,-2.85) {$X_{8c}$};
  \end{scope}
  \begin{scope}[xshift=6.6cm]
    \foreach \k in {0,...,11}{ \coordinate (b\k) at ({90-30*\k}:2.4); }
    \foreach \k in {0,...,11}{ \foreach \d in {1,2,4,5}{
      \pgfmathtruncatemacro{\l}{mod(\k+\d,12)} \draw[edge] (b\k)--(b\l); } }
    \foreach \k in {0,...,5}{ \pgfmathtruncatemacro{\l}{\k+6} \draw[pst] (b\k)--(b\l); }
    \foreach \k in {0,...,11}{ \node[vx] at (b\k) {}; }
    \foreach \k in {0,...,11}{ \node[lb] at ({90-30*\k}:2.78) {$\k$}; }
    \node at (0,-3.35) {$X_{12a}$};
  \end{scope}
\end{tikzpicture}
\caption{Two graphs with perfect state transfer, each of which is a Cayley graph of every group of its order. They are rows $X_{8c}$ (\textit{left}) and $X_{12a}$ (\textit{right}) of the census in Appendix~\ref{app:census}.  They are drawn in their circulant presentations, $X_{8c} = \Cay(\Z_8,\{\pm1,\pm3,4\})$ and $X_{12a} = \Cay(\Z_{12},\{\pm1,\pm2,\pm4,\pm5,6\})$, and in both the perfect state transfer at time $\pi/2$ is between the  vertices $x$ and $x+n/2$ connected by orange edges.}
\label{fig:intro-universal}
\end{figure}
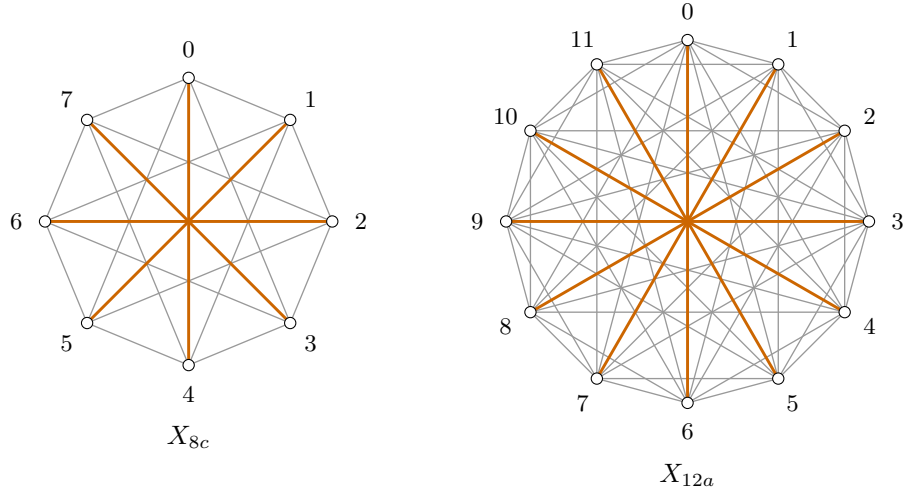 

Cayley graphs are nevertheless a natural setting to study quantum walks; they are vertex-transitive, the group structure gives algebraic tools for studying the graph, and for abelian groups and for conjugacy-closed connection sets their spectra are especially accessible.  This has led to important results on cubelike graphs \cite{bernasconi,cheung,chan2013complex}, circulants \cite{basic2013char}, and other families of Cayley graphs.  Independently, the theory of isomorphic Cayley graphs asks when one graph can be represented over different groups.  Beginning with work of Joseph \cite{Joseph1995} and Morris \cite{Morris1999}, this subject has developed tools for detecting and constructing different regular subgroups of a graph's automorphism group. Results particularly relevant here concern cyclic, dihedral, and generalized dihedral groups \cite{morris2021two,MorrisSkelton}.  
We bring these two lines of work together, using the theory of isomorphic Cayley graphs to organize and compare the many constructions of perfect state transfer on Cayley graphs, to identify when ostensibly different non-abelian constructions produce graphs already arising in the abelian setting, and to distinguish genuinely non-abelian examples from abelian Cayley graphs in non-abelian presentations.

Two familiar families show why this question is necessary.  The graph obtained from $K_{4n}$ by deleting a perfect matching admits perfect state transfer \cite{TsoPleVeg2008}, but it is a Cayley graph of every group of order $4n$.  Similarly, the hypercube $Q_d$, usually presented as a Cayley graph of the elementary abelian group $C_2^d$, is also a Cayley graph of a non-abelian group for every $d\geq3$.  Thus the existence of a non-abelian Cayley presentation, even for an infinite family, does not by itself reveal a non-abelian mechanism for state transfer.  A related issue is the distinction between the left and right actions of $G$ on $\Cay(G,S)$: with our convention the right action always consists of automorphisms, whereas the left action does so only when its elements normalize $S$.  In Section~\ref{sec:lr-actions} we use this distinction to give counterexamples to \cite[Lemma~7.1]{godsil2012state} as stated and to explain why the result becomes valid when the connection set is closed under conjugation.  Section~\ref{sec:nonexcl-pst} develops the two families above in detail.
Two small examples are shown in Figure~\ref{fig:intro-universal}: each admits perfect state transfer and is a Cayley graph of every group of its order.

Our first main results give conditions under which a non-abelian Cayley graph can be realized as an abelian Cayley graph.  Suppose that $G$ has an abelian subgroup $H$ of index two.  In Section~\ref{sec:abind2} we construct an abelian regular subgroup of $\Aut(\Cay(G,S))$ under each of three hypotheses, each of which is a way of certifying that some element of $G\setminus H$ normalizes $S$ (Lemma~\ref{lem:abelian-index-two}).  First, if $S$ is closed under conjugation, then $\Cay(G,S)$ is a Cayley graph of an abelian group; this statement does not require perfect state transfer (Lemma~\ref{lem:nonabel-normal-2}).  Second, if the graph admits perfect state transfer and its perfect-state-transfer involution is left multiplication by an element of $G\setminus H$, then the graph is a Cayley graph of $H\times C_2$ (Lemma~\ref{lem:AbCayTninH}).    Third, if $G$ is generalized dihedral over $H$ and $S=S_0\cup bS_1$,
  where $b$ is a reflection, $S_0,S_1\subseteq H$ and $S_1$ is
  inverse-closed, then $\Cay(G,S)$ is a Cayley graph of $H\times C_2$;
  again, no perfect state transfer is required (Lemma~\ref{lem:gendihedral}).

These reductions yield a sharp existence result.  If $G$ has an abelian subgroup of index two and its centre contains no involution, then no connected Cayley graph of $G$ admits perfect state transfer and without the connectedness assumption, the only possibility is a perfect matching (Corollary~\ref{cor:odd-centre}).  For a non-abelian group with an abelian subgroup of index two, the absence of a central involution is equivalent to $\lvert G\rvert\equiv2\pmod4$.  Conversely, when $4\mid\lvert G\rvert$, the complete graph on $\lvert G\rvert$ vertices with a perfect matching deleted is a connected Cayley graph of $G$ with perfect state transfer.  Consequently, such a group admits a connected Cayley graph with perfect state transfer if and only if its order is divisible by four (Theorem~\ref{thm:mod4-dichotomy}).

In the same vein with different methods, we prove that if $G$ is an extraspecial $p$-group and $S$ is closed under conjugation, then $\Cay(G,S)$ is a Cayley graph of the elementary abelian group $\Z_p^{\,2n+1}$ of the same order (Lemma~\ref{lem:extraspecial}).    
Together with the index-two results, this shows that the explicit perfect-state-transfer constructions in \cite{AreShaGho2022,CaoFen2021,WanCao2024,KalBha2024, WanFen2023-bicayley}, over dihedral, dicyclic, generalized dihedral, and $V_{8n}$ groups, as well as the extraspecial-group constructions of Sin and Sorci in \cite{sin2020}, all admit abelian Cayley presentations.
In several cases, this means that the underlying graphs had already appeared earlier in the literature as Cayley graphs for some abelian groups but this connection was not made. Alternative group presentations may themselves be useful, but distinguishing a new presentation from a new graph is essential when the property under study is invariant under graph isomorphism. 

There are, on the other hand, genuinely non-abelian examples.  By  this we mean Cayley graphs with perfect state transfer whose automorphism groups contain no abelian  regular subgroup.  Pantangi and Sin \cite{PanSin2026} construct perfect  state transfer on Cayley graphs of $\mathrm{SL}(2,q)$,  $\mathrm{GL}(2,q)$ and $\mathrm{GU}(2,q)$ for odd prime powers  $q$.  In Section~\ref{sec:pansin} we prove that, for every odd prime
  power $q\geq5$, their $\mathrm{SL}(2,q)$ graph is not a Cayley graph
  of any abelian group (Theorem~\ref{thm:pansin-sl}).  The proof
  passes to the quotient by the twin pairs, a Cayley graph of
  $\mathrm{PSL}(2,q)$, where the two-sided translations and inversion
  generate a primitive group of simple diagonal type. The inclusion
  theorem for such groups \cite[Proposition~8.1]{Praeger1990} and a bound
  on the abelian subgroups of $\mathrm{PSL}(2,q)$ finish the argument.
  For the $\mathrm{GL}$ and $\mathrm{GU}$ families the quotient group
  is not simple and a similar argument does not apply; we verify by
  exhaustive computation in SageMath \cite{sage} that the four smallest
  members are not Cayley graphs of abelian groups
  (Proposition~\ref{prop:pansin-glgu}, Appendix~\ref{app:pansin-glgu})
  and leave the general case open
  (Problem~\ref{prob:pansin-stability}).  

The smallest examples of this phenomenon are found by computation in Section~\ref{sec:counterexamples}.
We determine every connected vertex-transitive graph on at most $30$ vertices that admits perfect state transfer and, for each one, every group over which it is a Cayley graph. This census is available in \texttt{.csv} format as ancillary files.  Every example on at most $22$ vertices has an abelian Cayley presentation.  On $24$ vertices there are $14$ graphs whose Cayley groups are exclusively non-abelian; these are the smallest Cayley graphs with perfect state transfer and no abelian Cayley presentation.  On $30$ vertices, the graph $T(6)[K_2]$ is the smallest vertex-transitive non-Cayley graph with perfect state transfer.  In a partial scan at order $32$ and valency up to $12$ (complete only up to valency $11$), we find $38$ graphical regular representations of non-abelian groups with perfect state transfer.  Since a graphical regular representation has only one Cayley presentation,
these are the strongest possible counterexamples to an abelian reduction and, to our knowledge, the first examples of their kind.  We also give a $48$-vertex Cayley graph of both a dihedral and a dicyclic group that has perfect state transfer but no abelian Cayley presentation, showing that the symmetry hypothesis in our generalized-dihedral result is necessary. 

Finally, we use peak state transfer to illustrate how a new state-transfer family can be studied together with all of its Cayley presentations.
In Section~\ref{sec:peak-ref} we establish a Cartesian-product principle
for peak state transfer: if $X$ admits peak state transfer from $u$ to $v$ at
time $\tau$, and $Y$ is periodic at $y$ at the same time, then $X\square Y$
admits peak state transfer from $(u,y)$ to $(v,y)$ at time $\tau$, with the
same amount (Lemma~\ref{lem:peak-product}).  Applying this principle to the
$20$-vertex graph $X_{14}$ and the Hamming graphs $H(d,6) = K_6^{\,\square d}$,
which are periodic at $\pi/3$, gives the family 
\[
   \Gamma_d=X_{14}\mathbin{\square}H(d,6),\qquad d\geq0,
\]
on $20\cdot6^d$ vertices.  Each $\Gamma_d$ admits peak state transfer at time $\pi/3$ with amount $16/225$, independent of $d$.
We determine all regular subgroups of $\Aut(\Gamma_d)$; this reduces the Cayley groups of $\Gamma_d$ to the regular subgroups of $\Aut(H(d,6))$, and shows that every one of them has the non-abelian Frobenius group $F_{20}$ as a quotient, so that no member of the family admits an abelian Cayley presentation.

Taken together, these results give a two-way connection between quantum walks and Cayley graph theory.  Results about regular subgroups and isomorphic Cayley graphs determine when a state-transfer construction genuinely depends on non-abelian structure; in the other direction, state transfer leads to new examples and classification problems for regular subgroups, including non-abelian graphical regular representations and families for which all Cayley groups can be described.  After recalling the necessary quantum-walk preliminaries in Section~\ref{sec:prelim}, we develop the left--right action issue and the basic examples, prove the abelian-reduction theorems, treat the genuinely non-abelian examples and the census, and construct the peak-state-transfer family.  We end with open problems suggested by this interaction.  The appendix contains the complete census through $30$ vertices, together with the spectrum and all Cayley groups of each graph; the accompanying data also include the partial scan at $32$ vertices.

\section{Quantum walks preliminaries}\label{sec:prelim}

Let $X$ be a simple graph with adjacency matrix $A$. The \textsl{transition matrix} of $X$ at time $t$ is $U(t) = \exp(itA)$. It is a symmetric, unitary matrix. For vertices $u$ and $v$, the transfer probability at time $t$ is
$|U(t)_{v,u}|^2$. We say that $X$ admits \textsl{perfect state transfer (PST)} from $u$ to $v$ at time $\tau$ if 
\[|U(\tau)_{v,u}| = 1.\]
If $A=\sum_\theta \theta E_\theta$ is the spectral decomposition of $A$,
then
\[
    U(t)_{v,u} = \sum_\theta e^{i\theta t}\,(E_\theta)_{v,u},
    \qquad\text{so}\qquad
    |U(t)_{v,u}| \le \sum_\theta \,|(E_\theta)_{v,u}|
\]
for all $t$, and \textsl{peak state transfer} occurs exactly when this upper bound
is attained at some time $t=\tau$. 
If peak state transfer occurs at time $\tau$, the transfer probability
$\alpha = |U(\tau)_{v,u}|^2 = \left(\sum_\theta \,|(E_\theta)_{v,u}| \right)^2$
at that time is the \textsl{amount} of the transfer.  Perfect state transfer is
peak state transfer with amount one.  We also call
$\{\theta : (E_\theta)_{v,u}\neq 0\}$ the \textsl{eigenvalue support} of the
pair $(u,v)$; it is the set of eigenvalues that contribute to the sum above.
This notion was introduced in \cite{CouGuoSch2025}, along with a spectral
characterization of the pairs of vertices that admit it.  Unlike perfect
or pretty good state transfer, peak state transfer carries a quantitative
parameter, and it can occur with very small amount; we take the view that
an example of peak state transfer should always be reported together with
its amount.

Finally, $X$ is \textsl{periodic} at a vertex $u$ at time $\tau$ if $|U(\tau)_{u,u}| = 1$.  If $X$ admits perfect state transfer from $u$ to $v$ at time $\tau$, then $X$ is periodic at $u$ at time $2\tau$. If $X$ is periodic at every vertex at time $\tau$, equivalently if $U(\tau)$ is a scalar matrix, then we say that $X$ is \textsl{periodic} at time $\tau$.

In this paper, our focus is on Cayley graphs which are vertex-transitive. If $X$ is a vertex-transitive graph admitting PST between vertices $u$ and $v$, at time $\tau$, then every vertex $u'$ in $X$ is involved in PST with some vertex $v'$ at time $\tau$ and this induces a partition of the vertices of $X$ into sets of size two (Theorem~\ref{thm:vxtr-pst}). Consequently, $X$ must have an even number of vertices and it is periodic at time $2\tau$. Periodic graphs have a particularly nice spectrum as summarized in a theorem by Godsil \cite{godsil2011periodic}.

\begin{thm}[{\cite[Corollary 3.3]{godsil2011periodic}}]
    A graph $X$ is periodic if and only if either:
    \begin{enumerate}[label=(\alph*)]
        \item The eigenvalues of $X$ are integers, or
        \item The eigenvalues of $X$ are rational multiples of $\sqrt\Delta$, for some square-free integer $\Delta.$
    \end{enumerate}
\end{thm}

Since vertex-transitive graphs are regular, their degree --- an integer --- is an eigenvalue and since PST in a vertex-transitive graph implies periodicity, it follows from the theorem that vertex-transitive graphs admitting PST have only integer eigenvalues.

\section{Left vs Right Actions}\label{sec:lr-actions}

Given a group $G$ with identity $1$ and an inverse-closed subset $S\subseteq G\sm \{1\}$, we will define the \textsl{Cayley graph} $X = \Cay(G,S)$ as the graph with vertex set $G$ and $g\sim h$ whenever $hg^{-1}\in S$. In the literature, this is sometimes called the \textsl{right action Cayley graph}. With this convention, the action of $a\in G$ by right multiplication $R_a : g \mapsto ga$ is an automorphism of $X$, since $(ha)(ga)^{-1} = hg^{-1}\in S$. Therefore, defining $R(G) = \{R_a:a\in G\}$, we have $R(G)\cong G$ and $R(G) \le \Aut(X)$. We can interpret this as $G$ being a subgroup of the automorphism group, but in this paper it will be important to distinguish between the right and the left action.   Two classes of Cayley graphs recur in this paper.  If $G$ is cyclic, then $\Cay(G,S)$
  is a \textsl{circulant}; if $G$ is the elementary abelian group $C_2^d$, then
  $\Cay(G,S)$ is a \textsl{cubelike graph}. Note that we use both $\Z_2^d$ and $C_2^d$; we have chosen to use multiplicative notation throughout but $\Z_2^d$ is more common in the literature.

The action of $a\in G$ by left multiplication $L_a : g \mapsto ag$ is an automorphism of $X$
if and only if  $aS a^{-1} = S$; that is whenever $a$ normalizes $S$. Thus, $L(G)$ (defined analogously to $R(G)$) is a subgroup of the symmetric group, $\Sym(G)$, but it is not in general contained in $\Aut(X)$.

Denote by $Z(G)$ the centre of $G$. Then we have 

\begin{enumerate}
    \item $R(G)$ and $L(G)$ are both regular subgroups of $\Sym(G)$;
    \item every element of one commutes with every element of the other;
    \item $R(G) \cap L(G) = \{R_z : z \in Z(G)\}$.
\end{enumerate}
These are easy consequences of the next lemma which follows from \cite[Theorem~I.6.5]{huppert}, but we include a short proof for completeness. 

\begin{lem}\label{lem:centralizer-magic}
    The centralizer of $R(G)$ in $\Sym(G)$ is $L(G)$. The centralizer of $L(G)$ in $\Sym(G)$ is $R(G)$. For $X =\Cay(G,S)$, the centralizer of $R(G)$ in $\Aut(X)$ is $L(G) \cap \Aut(X)$.
\end{lem}
\proof First we show the statement for $\Sym(G)$. Let $g,h\in G$ and consider $L_g$ and $R_h$ acting on $x \in G$. We see that 
\[
x^{L_gR_h} = gxh = x^{R_hL_g},
\]
so $L_g, R_h$ commute. 

Now suppose $\sigma \in \Sym(G)$ commutes with all of the elements of $R(G)$; that is $x^{\sigma R_h} = x^{R_h\sigma}$ for all $x,h\in G$. In particular, this holds for $x=1_G$. Let $a = 1^{\sigma}$. We apply the left hand side to $1_G$ and obtain:
\[
1_G^{\sigma R_h}  = a^{R_h} = ah. 
\]
We apply the right hand side to $1_G$ and obtain:
\[
1_G^{R_h\sigma} = h^{\sigma}. 
\]
Since $1_G^{\sigma R_h} = 1_G^{R_h\sigma}$, we see that 
\[
h^{\sigma} = ah
\]
so $\sigma = L_a\in L(G)$, as required.

The second statement follows similarly. 
The vertex set of $X$ is $G$ and thus $\Aut(X) \leq \Sym(G)$. The centralizer of $R(G)$ in $\Aut(X)$ is then contained in $L(G)$ and the last statement follows. \qed

We will now address an incorrect statement on PST in Cayley graphs \cite[Lemma 7.1]{godsil2012state} which arises from the confusion of the left and the right regular action.
We start with a correct and more general version of the theorem in question due to Godsil \cite{godsil2011periodic}.

\begin{thm}[{\cite[Theorem 4.1]{godsil2011periodic}}]\label{thm:vxtr-pst}
    Let $X$ be a vertex-transitive graph with PST occurring at time $\tau$. Then $U(\tau)$ is a scalar multiple of a permutation matrix with order two and no fixed points that lies in the centre of $\Aut(X)$.
\end{thm}

Letting $T$ denote this permutation matrix, \cite[Lemma 7.1]{godsil2012state} states:

\noindent\textsl{``If $X$ is a Cayley graph for a group $G$ and perfect state transfer occurs at time $\tau$, then $T$ lies in the centre of $G$.''}

\noindent This is however only valid when both regular actions sit inside $\Aut(X)$: that is, this holds when
\[
    R(G) \le \Aut(X) \quad\text{and}\quad L(G) \le \Aut(X),
\]
but is not necessarily true in general. 
We observe that \cite[Lemma 7.1]{godsil2012state} is true with the added hypothesis that the connection set of $X$ is closed under conjugation; we call such a connection set, and the Cayley graph it defines, \textsl{conjugacy-closed}.  Some of the literature calls these \textsl{normal} Cayley graphs.  We avoid that term, because in permutation group theory a Cayley graph is called normal when $R(G)$ is a normal subgroup of its automorphism group, and the two conditions are different: the graphical regular representations of Section~\ref{sec:grrs} are normal in the permutation group sense, since their automorphism group is $R(G)$ itself, but their connection sets are not conjugacy-closed. This issue with  \cite[Lemma 7.1]{godsil2012state} has already been noticed and addressed by the field, without explicitly describing the error, as follows:
\begin{itemize}
    \item \cite[Theorem 1]{CaoWanFen2020} reproves this lemma with the added assumption that $S$ is conjugacy-closed. 
\item  \cite[Lemma 4.1]{arnadottir2023strongly} reproves this lemma for conjugacy-closed connection sets, without \cite{godsil2012state} attribution.
\item \cite[Lemma 3.2]{WanCao2022} and  \cite[Lemma 3.1]{WanCao2024} are restatements, attributed to \cite{godsil2012state}, with conjugacy-closed assumption added (but without proof).
\end{itemize}

For counterexamples to \cite[Lemma 7.1]{godsil2012state}, we have for instance many Cayley graphs for $S_4$, a group with trivial centre, that admit PST. This can be seen in the table in the appendix.

Concretely, let
\[S = \{(1\,2),(1\,3),(1\,4),(2\,3),(2\,4), (1\,3\,4),(1\,4\,3),(2\,3\,4),(2\,4\,3)\}.\]
We see that $S$ is inverse-closed and $X = \Cay(S_4,S)$\footnote{Graph6 string of $X$: \Verb"WX@zs]g?ioINoS@HysCaa@q@OxD?jXBTfWCKb_fPOup@B^C"}  is a connected $9$-regular graph on $24$ vertices with spectrum
\[
    \{9,\ 3^{(3)},\ 1^{(9)},\ (-1)^{(6)},\ (-3)^{(2)},\ (-5)^{(3)}\}.
\]
We verify in SageMath, exactly over $\Q(i)$, that $U(\pi/2) = i\,T$ where $T$ is the permutation matrix of $L_{(1\,2)}$; thus $X$ admits perfect state transfer from $x$ to $(1\,2)x$ at time $\pi/2$.  Since $T$ commutes with $R(S_4)$ and $S_4$ has trivial centre, we see that $T$ is not an element of $R(S_4)$, so the conclusion of \cite[Lemma~7.1]{godsil2012state} fails.  Note that $S$ is not closed under conjugation, but it is normalized by $(1\,2)$, and so $L_{(1\,2)}\in\Aut(X)$ as in Lemma~\ref{lem:centralizer-magic}.  Moreover, $|\Aut(X)| = 3072$ and the only regular subgroups of $\Aut(X)$ are $S_4$ and $C_2\times A_4$; in particular, $X$ is not a Cayley graph of any abelian group. This graph is shown in Figure~\ref{fig:s4-counterexample} with the transfer matching highlighted. 

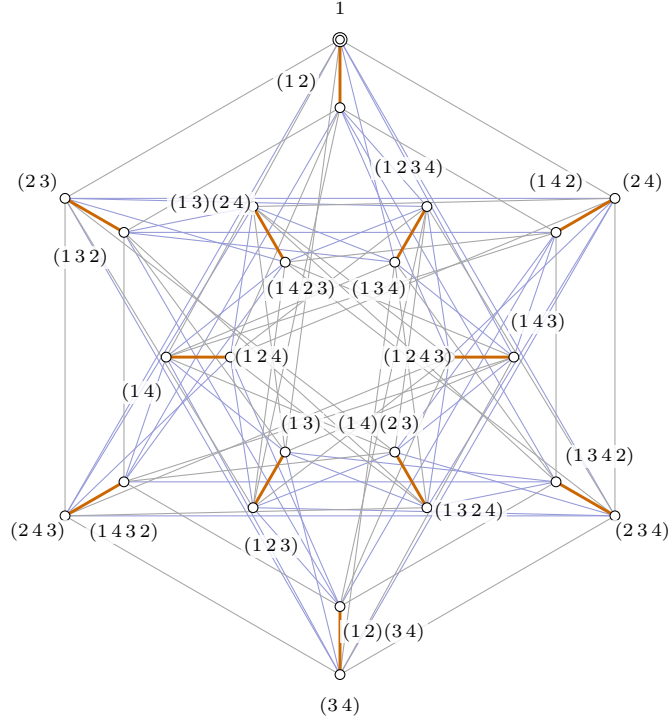
\begin{figure}[t]
\centering
\begin{tikzpicture}[
  tr/.style={draw=black!35, line width=0.4pt},
  cyc/.style={draw=myblue!45, line width=0.4pt},
  pst/.style={draw=myorange, line width=1.1pt},
  vx/.style={circle, draw=black, fill=white, inner sep=1.3pt},
  idv/.style={circle, draw=black, double, fill=white, inner sep=1.5pt},
  lb/.style={inner sep=0.8pt, fill=white, fill opacity=0.8, text opacity=1, font=\scriptsize}]
  \coordinate (v0) at (90:4.20000000000000);
  \coordinate (v1) at (120:2.30000000000000);
  \coordinate (v2) at (300:1.45000000000000);
  \coordinate (v3) at (270:3.30000000000000);
  \coordinate (v4) at (330:4.20000000000000);
  \coordinate (v5) at (150:3.30000000000000);
  \coordinate (v6) at (0:2.30000000000000);
  \coordinate (v7) at (180:1.45000000000000);
  \coordinate (v8) at (210:4.20000000000000);
  \coordinate (v9) at (60:1.45000000000000);
  \coordinate (v10) at (30:3.30000000000000);
  \coordinate (v11) at (240:2.30000000000000);
  \coordinate (v12) at (270:4.20000000000000);
  \coordinate (v13) at (300:2.30000000000000);
  \coordinate (v14) at (120:1.45000000000000);
  \coordinate (v15) at (90:3.30000000000000);
  \coordinate (v16) at (150:4.20000000000000);
  \coordinate (v17) at (330:3.30000000000000);
  \coordinate (v18) at (180:2.30000000000000);
  \coordinate (v19) at (0:1.45000000000000);
  \coordinate (v20) at (30:4.20000000000000);
  \coordinate (v21) at (240:1.45000000000000);
  \coordinate (v22) at (210:3.30000000000000);
  \coordinate (v23) at (60:2.30000000000000);
  \draw[cyc] (v21) -- (v13);
  \draw[cyc] (v15) -- (v17);
  \draw[cyc] (v14) -- (v18);
  \draw[cyc] (v13) -- (v20);
  \draw[cyc] (v2) -- (v4);
  \draw[cyc] (v0) -- (v4);
  \draw[cyc] (v12) -- (v16);
  \draw[cyc] (v14) -- (v23);
  \draw[cyc] (v6) -- (v10);
  \draw[cyc] (v3) -- (v5);
  \draw[cyc] (v10) -- (v2);
  \draw[cyc] (v14) -- (v16);
  \draw[cyc] (v3) -- (v10);
  \draw[cyc] (v22) -- (v18);
  \draw[cyc] (v12) -- (v20);
  \draw[cyc] (v6) -- (v9);
  \draw[cyc] (v8) -- (v0);
  \draw[cyc] (v21) -- (v17);
  \draw[cyc] (v13) -- (v17);
  \draw[cyc] (v11) -- (v2);
  \draw[cyc] (v20) -- (v16);
  \draw[cyc] (v1) -- (v7);
  \draw[cyc] (v13) -- (v19);
  \draw[cyc] (v22) -- (v17);
  \draw[cyc] (v11) -- (v4);
  \draw[cyc] (v5) -- (v9);
  \draw[cyc] (v5) -- (v10);
  \draw[cyc] (v12) -- (v18);
  \draw[cyc] (v19) -- (v20);
  \draw[cyc] (v1) -- (v8);
  \draw[cyc] (v3) -- (v7);
  \draw[cyc] (v19) -- (v23);
  \draw[cyc] (v6) -- (v2);
  \draw[cyc] (v8) -- (v7);
  \draw[cyc] (v1) -- (v5);
  \draw[cyc] (v23) -- (v16);
  \draw[cyc] (v15) -- (v19);
  \draw[cyc] (v11) -- (v7);
  \draw[cyc] (v6) -- (v0);
  \draw[cyc] (v22) -- (v15);
  \draw[cyc] (v21) -- (v12);
  \draw[cyc] (v21) -- (v18);
  \draw[cyc] (v9) -- (v0);
  \draw[cyc] (v15) -- (v23);
  \draw[cyc] (v1) -- (v9);
  \draw[cyc] (v3) -- (v11);
  \draw[cyc] (v8) -- (v4);
  \draw[cyc] (v22) -- (v14);
  \draw[tr] (v6) -- (v12);
  \draw[tr] (v3) -- (v17);
  \draw[tr] (v1) -- (v19);
  \draw[tr] (v2) -- (v18);
  \draw[tr] (v9) -- (v12);
  \draw[tr] (v2) -- (v16);
  \draw[tr] (v0) -- (v16);
  \draw[tr] (v8) -- (v12);
  \draw[tr] (v21) -- (v0);
  \draw[tr] (v15) -- (v11);
  \draw[tr] (v14) -- (v4);
  \draw[tr] (v5) -- (v13);
  \draw[tr] (v0) -- (v18);
  \draw[tr] (v21) -- (v5);
  \draw[tr] (v21) -- (v6);
  \draw[tr] (v15) -- (v10);
  \draw[tr] (v11) -- (v16);
  \draw[tr] (v10) -- (v17);
  \draw[tr] (v8) -- (v16);
  \draw[tr] (v20) -- (v4);
  \draw[tr] (v22) -- (v5);
  \draw[tr] (v11) -- (v19);
  \draw[tr] (v20) -- (v0);
  \draw[tr] (v14) -- (v6);
  \draw[tr] (v12) -- (v4);
  \draw[tr] (v9) -- (v17);
  \draw[tr] (v1) -- (v20);
  \draw[tr] (v10) -- (v18);
  \draw[tr] (v23) -- (v4);
  \draw[tr] (v1) -- (v21);
  \draw[tr] (v8) -- (v19);
  \draw[tr] (v22) -- (v6);
  \draw[tr] (v3) -- (v19);
  \draw[tr] (v22) -- (v2);
  \draw[tr] (v20) -- (v7);
  \draw[tr] (v2) -- (v23);
  \draw[tr] (v9) -- (v18);
  \draw[tr] (v3) -- (v22);
  \draw[tr] (v1) -- (v17);
  \draw[tr] (v14) -- (v11);
  \draw[tr] (v13) -- (v7);
  \draw[tr] (v9) -- (v13);
  \draw[tr] (v8) -- (v13);
  \draw[tr] (v3) -- (v23);
  \draw[tr] (v15) -- (v5);
  \draw[tr] (v23) -- (v7);
  \draw[tr] (v15) -- (v7);
  \draw[tr] (v14) -- (v10);
  \draw[pst] (v18) -- (v7);
  \draw[pst] (v3) -- (v12);
  \draw[pst] (v9) -- (v23);
  \draw[pst] (v15) -- (v0);
  \draw[pst] (v21) -- (v11);
  \draw[pst] (v22) -- (v8);
  \draw[pst] (v13) -- (v2);
  \draw[pst] (v5) -- (v16);
  \draw[pst] (v10) -- (v20);
  \draw[pst] (v17) -- (v4);
  \draw[pst] (v6) -- (v19);
  \draw[pst] (v1) -- (v14);
  \node[idv] at (90:4.20000000000000) {};
  \node[vx] at (120:2.30000000000000) {};
  \node[vx] at (300:1.45000000000000) {};
  \node[vx] at (270:3.30000000000000) {};
  \node[vx] at (330:4.20000000000000) {};
  \node[vx] at (150:3.30000000000000) {};
  \node[vx] at (0:2.30000000000000) {};
  \node[vx] at (180:1.45000000000000) {};
  \node[vx] at (210:4.20000000000000) {};
  \node[vx] at (60:1.45000000000000) {};
  \node[vx] at (30:3.30000000000000) {};
  \node[vx] at (240:2.30000000000000) {};
  \node[vx] at (270:4.20000000000000) {};
  \node[vx] at (300:2.30000000000000) {};
  \node[vx] at (120:1.45000000000000) {};
  \node[vx] at (90:3.30000000000000) {};
  \node[vx] at (150:4.20000000000000) {};
  \node[vx] at (330:3.30000000000000) {};
  \node[vx] at (180:2.30000000000000) {};
  \node[vx] at (0:1.45000000000000) {};
  \node[vx] at (30:4.20000000000000) {};
  \node[vx] at (240:1.45000000000000) {};
  \node[vx] at (210:3.30000000000000) {};
  \node[vx] at (60:2.30000000000000) {};
  \node[lb] at (90:4.62) {$1$};
  \node[lb] at (130:2.66) {$(1\,3)(2\,4)$};
  \node[lb] at (300:1.03) {$(1\,4)(2\,3)$};
  \node[lb] at (279:3.68) {$(1\,2)(3\,4)$};
  \node[lb] at (330:4.62) {$(2\,3\,4)$};
  \node[lb] at (159:3.68) {$(1\,3\,2)$};
  \node[lb] at (10:2.66) {$(1\,4\,3)$};
  \node[lb] at (180:1.03) {$(1\,2\,4)$};
  \node[lb] at (210:4.62) {$(2\,4\,3)$};
  \node[lb] at (60:1.03) {$(1\,3\,4)$};
  \node[lb] at (39:3.68) {$(1\,4\,2)$};
  \node[lb] at (250:2.66) {$(1\,2\,3)$};
  \node[lb] at (270:4.62) {$(3\,4)$};
  \node[lb] at (310:2.66) {$(1\,3\,2\,4)$};
  \node[lb] at (120:1.03) {$(1\,4\,2\,3)$};
  \node[lb] at (99:3.68) {$(1\,2)$};
  \node[lb] at (150:4.62) {$(2\,3)$};
  \node[lb] at (339:3.68) {$(1\,3\,4\,2)$};
  \node[lb] at (190:2.66) {$(1\,4)$};
  \node[lb] at (0:1.03) {$(1\,2\,4\,3)$};
  \node[lb] at (30:4.62) {$(2\,4)$};
  \node[lb] at (240:1.03) {$(1\,3)$};
  \node[lb] at (219:3.68) {$(1\,4\,3\,2)$};
  \node[lb] at (70:2.66) {$(1\,2\,3\,4)$};
\end{tikzpicture}
\caption{The $24$-vertex counterexample $X = \Cay(S_4,S)$.
Edges are coloured by generator: the transposition $(1\,2)$ gives the orange perfect matching, the other four transpositions are grey, and the four $3$-cycles are blue.  Perfect state transfer at time $\pi/2$ occurs between vertices connected by the orange perfect matching. 
The graph is row $X_{24b}$ of the census (Appendix~\ref{app:census}).}
\label{fig:s4-counterexample}
\end{figure}

Counterexamples occur already on $8$ vertices.  Let $s$ be a reflection in $D_8$ (note that the index denotes the order, not the degree).  Then $\Cay(D_8,\,D_8\setminus\{1,s\})$ is $K_8$ minus a perfect matching, which has perfect state transfer at time $\pi/2$ by \cite{TsoPleVeg2008}, and the deleted matching pairs $x$ with $sx$, so the PST involution is $T = L_s$; since $s$ is not central, $T\notin R(D_8)$.  (Verified exactly in SageMath: $U(\pi/2) = -T$.)  This is precisely the presentation of Section~\ref{sec:kminuspm}: the connection set of $H_{2n}$ as a Cayley graph of $G$ fails to be conjugacy-closed exactly when the deleted involution is non-central, and it is then that the conclusion of the lemma fails. The graph, with its deleted matching shown in dashed lines, and a second, conjugacy-closed Cayley labelling of the same graph over $Q_8$ are drawn in Figure~\ref{fig:d8-counterexample}. 

\begin{figure}[t]
\centering
\begin{tikzpicture}[
  edge/.style={draw=black!40, line width=0.5pt},
  gone/.style={draw=myorange, dashed, line width=1.0pt},
  vx/.style={circle, draw=black, fill=white, inner sep=1.5pt},
  lb/.style={font=\small},
  qlb/.style={font=\small, text=mygreen}]
  \foreach \k in {0,...,7}{ \coordinate (w\k) at ({90-45*\k}:2.2); }
  \foreach \k in {0,...,7}{ \foreach \l in {0,...,7}{
    \pgfmathparse{ifthenelse(\l>\k && \l-\k!=4, 1, 0)}
    \ifnum\pgfmathresult>0 \draw[edge] (w\k) -- (w\l); \fi } }
  \foreach \k in {0,...,3}{ \pgfmathtruncatemacro{\l}{\k+4} \draw[gone] (w\k) -- (w\l); }
  \foreach \k in {0,...,7}{ \node[vx] at (w\k) {}; }
  \node[lb] at (90:2.55)  {$1$};
  \node[lb] at (45:2.55)  {$r$};
  \node[lb] at (0:2.6)  {$r^{2}$};
  \node[lb] at (-45:2.55) {$r^{3}$};
  \node[lb] at (-90:2.55) {$s$};
  \node[lb] at (-135:2.6){$r^{3}s$};
  \node[lb] at (180:2.6){$r^{2}s$};
  \node[lb] at (135:2.55) {$rs$};
  \node[qlb] at (90:3.0)  {$1$};
  \node[qlb] at (45:3.0)  {$i$};
  \node[qlb] at (0:3.1)  {$j$};
  \node[qlb] at (-45:3.0) {$k$};
  \node[qlb] at (-90:3.0) {$-1$};
  \node[qlb] at (-135:3.1){$-i$};
  \node[qlb] at (180:3.15){$-j$};
  \node[qlb] at (135:3.0) {$-k$};
\end{tikzpicture}
\caption{The graph $H_8$, that is, $K_8$ minus a perfect matching (dashed orange lines), with two Cayley labellings over $D_8$ (\textit{in black}) and $Q_8$ (\textit{in green}).  
In fact $H_8$ is a Cayley graph of all five groups of order eight, see Lemma~\ref{lem:H2n}, (row $X_{8d}$ of Appendix~\ref{app:census}).}
\label{fig:d8-counterexample}
\end{figure}
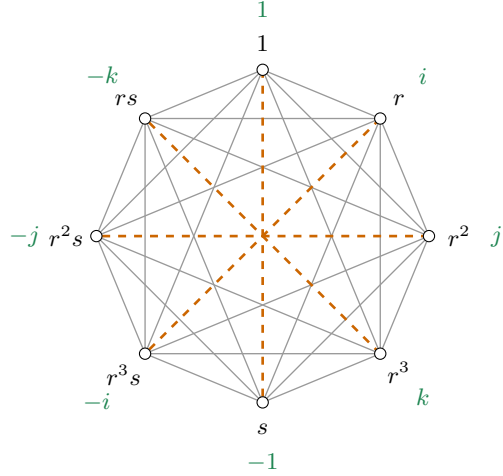  


\section{Examples of ``non-exclusive'' non-abelian PST}\label{sec:nonexcl-pst}

Here we collect examples which were known  in the literature to have PST, which also happen to be Cayley graphs of non-abelian groups. 

\subsection{Complete graphs with a perfect matching deleted}\label{sec:kminuspm}
Let $H_{2n}$ be $K_{2n}$ with a perfect matching deleted. Figure~\ref{fig:d8-counterexample} shows $H_8$ with vertex labels for two of its Cayley presentations.

\begin{lem}\label{lem:H2n}
    The graph $H_{2n}$, obtained from $K_{2n}$ by deleting a perfect matching, is Cayley for every group on $2n$ elements. Further, $H_{2n}$ admits a non-conjugacy-closed Cayley presentation over $G$ if and only if $G$ has a non-central involution, and   the presentation $\Cay(G,G\setminus\{1,h\})$ associated with an involution $h$ is conjugacy-closed exactly when $h$ is central.
\end{lem}

\proof
    Let $G$ be a group on $2n$ elements. Since $2 | 2n$, there is an element of order $2$, say $h\in G$. Then $\Cay(G, G\setminus \{h, 1_G\})$.

 Note that $S = G\setminus \{h, 1_G\}$ is closed under conjugation in $G$ if and only if its complement $\{1_G,h\}$ is normal in $G$,
which is equivalent to $ghg^{-1}=h$ for every $g\in G$ and the claim follows.\qed

Perfect state transfer in these graphs was first found in \cite{TsoPleVeg2008} where the graph is called the $(N/2)$-cross-polytope and the existence of perfect state transfer is proved for $N \equiv 0 \pmod 4$. Since then, it has appeared in a number of general, framework-building papers, including \cite{AngNorOpp2010,basic2013char}, which treat $\Cay(C_{4n}, C_{4n}\setminus\{0,2n\})$ as an integral circulant,  \cite{CouGodGuo2015} which characterizes perfect state transfer in strongly regular graphs (of which this is an example), and in \cite{TanFenCao2019} where it appears in the characterization of abelian Cayley graphs with perfect state transfer. In particular, the characterization result of  \cite{basic2013char} (and later \cite{CouGodGuo2015}) implies that PST in $H_N$ occurs if and only if $N \equiv 0 \pmod 4$. 

The graph $H_{4n}$ has been rediscovered several times, in papers which take
  a non-abelian group $G$ of order $4n$ with central involution $z \in G$ and
  establish perfect state transfer on $\Cay(G,\, G \setminus \{1, z\})$.  In
  each case the resulting graph is $H_{4n}$, so perfect state transfer already
  follows from \cite{TsoPleVeg2008}; since the graph is described only through
  a group presentation, the identification appears to have gone unnoticed.  These include:
\begin{itemize}  
      \item \cite{CaoFen2021}, Example~5.5: $\Cay(D_{4p}, D_{4p}\setminus\{1,a^p\})                      
      = H_{4p}$ for $p$ odd prime.                                 
      \item \cite{AreShaGho2022}, Example~4.1: $\Cay(T_{4n}, T_{4n}\setminus\{1,a^n\})                   
      =H_{4n}$ for all $n > 1$.                                      
      \item \cite{WanCao2024}, Example~5.1: $\Cay(T_{4n}, T_{4n}\setminus\{1,a^n\})                      
      =H_{4n}$ for $n$ odd.                                         
  \end{itemize}  

\subsection{Hypercubes}
Another rich example is the hypercubes, which can be realized as Cayley graphs of more than one group. The obvious groups are of course the elementary abelian 2-groups, $C_2^d$, but we now show that they are not the only ones. Figure~\ref{fig:cube-two-groups} shows $Q_3$ in its Cayley presentations over  $C_2^3$ and over $D_8$, which is the smallest case of the following lemma.  

\begin{lem}
    For $d\geq 3$, the hypercube $Q_d$ is a Cayley graph, with a connection set that is not
conjugacy-closed, of the non-abelian group $D_8\times C_2^{\,d-3}$.
\end{lem}

\proof
For the cube $Q_3$ with $d=3$, we write
\[
    D_8=\langle r,s\mid r^4=s^2=1,\ srs=r^{-1}\rangle
\]
for the dihedral group of order eight and we can see that
\[
    Q_3 \cong \Cay \bigl(D_8,\{r,r^{-1},s\}\bigr).
\]
This immediately gives examples in every higher dimension. For $d\ge 3$,
put
\[
    G_d=D_8\times C_2^{\,d-3}.
\]
Let $e_1,\ldots,e_{d-3}$ be the standard basis of $C_2^{\,d-3}$, and let
\[
    T_d=
    \{(r,1),(r^{-1},1),(s,1)\}
    \cup
    \{(1,e_i):1\le i\le d-3\}.
\]
Then
\[
    \Cay(G_d,T_d)
    \cong
    Q_3\square Q_{d-3}
    \cong
    Q_d.
\]
To see that $T_d$ is not closed under conjugation in $G_d$, we first look at $d=3$ and observe that  $rsr^{-1}=r^2s$ and so $T_3$ is not conjugacy-closed. For $d>3$, the conjugation in the first coordinate still sends $(s,1)$ to $(r^2s,1) \notin  T_d$.\qed

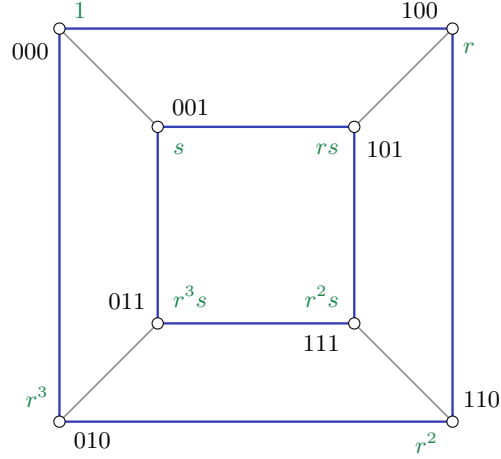
\begin{figure}[t]
\centering
\begin{tikzpicture}[scale=1.0,
  rsq/.style={draw=myblue!85, line width=0.9pt},
  smatch/.style={draw=black!45, line width=0.6pt},
  vx/.style={circle, draw=black, fill=white, inner sep=1.5pt},
  lb/.style={font=\small, inner sep=1pt},
  glb/.style={font=\small, text=mygreen, inner sep=1pt}]
  \coordinate (o0) at (-2.6, 2.6);
  \coordinate (o1) at ( 2.6, 2.6);
  \coordinate (o2) at ( 2.6,-2.6);
  \coordinate (o3) at (-2.6,-2.6);
  \coordinate (i0) at (-1.3, 1.3);
  \coordinate (i1) at ( 1.3, 1.3);
  \coordinate (i2) at ( 1.3,-1.3);
  \coordinate (i3) at (-1.3,-1.3);
  \draw[rsq] (o0)--(o1)--(o2)--(o3)--cycle;
  \draw[rsq] (i0)--(i1)--(i2)--(i3)--cycle;
  \foreach \k in {0,...,3}{ \draw[smatch] (o\k)--(i\k); }
  \foreach \p in {o0,o1,o2,o3,i0,i1,i2,i3}{ \node[vx] at (\p) {}; }
  \node[lb,  anchor=north east] at (-2.7, 2.45) {$000$};
  \node[glb, anchor=south west] at (-2.45, 2.7) {$1$};
  \node[lb,  anchor=south east] at ( 2.45, 2.7) {$100$};
  \node[glb, anchor=north west] at ( 2.7, 2.45) {$r$};
  \node[lb,  anchor=south west] at ( 2.7,-2.45) {$110$};
  \node[glb, anchor=north east] at ( 2.45,-2.7) {$r^{2}$};
  \node[lb,  anchor=north west] at (-2.45,-2.7) {$010$};
  \node[glb, anchor=south east] at (-2.7,-2.45) {$r^{3}$};
  \node[lb,  anchor=south west] at (-1.15, 1.42) {$001$};
  \node[glb, anchor=north west] at (-1.14, 1.12) {$s$};
  \node[lb,  anchor=north west] at ( 1.42, 1.15) {$101$};
  \node[glb, anchor=north east] at ( 1.14, 1.12) {$rs$};
  \node[lb,  anchor=north east] at ( 1.15,-1.42) {$111$};
  \node[glb, anchor=south east] at ( 1.14,-1.12) {$r^{2}s$};
  \node[lb,  anchor=south east] at (-1.42,-1.15) {$011$};
  \node[glb, anchor=south west] at (-1.14,-1.12) {$r^3 s$};
\end{tikzpicture}
\caption{The cube $Q_3$ with two Cayley labellings: the elementary abelian $\Z_2^{3}$ in black, with connection set the standard basis, and the dihedral $D_8=\langle r,s\mid r^{4}=s^{2}=1,\ srs=r^{-1}\rangle$ in green, with connection set $\{r,r^{-1},s\}$: the two $r$-cycles are the blue squares and the $s$-edges are the grey matching.  The dihedral presentation is not conjugacy-closed, since $rsr^{-1}=r^{2}s$ is not in the connection set.  As a graph with perfect state transfer, $Q_3$ is row $X_{8a}$ of the census (Appendix~\ref{app:census}).}
\label{fig:cube-two-groups}
\end{figure}

For small $d$ we can be completely explicit. Table~\ref{tab:hypercubes} lists, for $d\le 5$, every group over which $Q_d$ is a Cayley graph, obtained by enumerating the conjugacy classes of regular subgroups of $\Aut(Q_d) = C_2^d\rtimes S_d$ in GAP \cite{gap}. 
The cube $Q_3$ is a Cayley graph of exactly one non-abelian group, $D_8$, realizing the lemma above; $Q_4$ is a Cayley graph of five non-abelian groups; and $Q_5$ of exactly thirteen.

\begin{table}
\begin{center}
\renewcommand{\arraystretch}{1.25}
\begin{tabular}{@{}lr rr l@{}}
\toprule
Graph & $|V|$ & $|\Aut|$ & \#\,non-ab.\ & Non-abelian regular subgroups of $\Aut(Q_d)$\\
\midrule
$Q_1$ & $2$   & $2$    & $0$ & --- \\
$Q_2$ & $4$   & $8$    & $0$ & --- \\
$Q_3$ & $8$   & $48$   & $1$ & $D_8$ \\
$Q_4$ & $16$  & $384$  & $5$ & $(C_4{\times}C_2){:}C_2,\ C_4{:}C_4,\ C_8{:}C_2,\ QD_{16},\ C_2{\times}D_8$ \\
$Q_5$ & $32$  & $3840$ & $13$ &
$(C_2^3){:}C_4$, $(C_8{:}C_2){:}C_2$, $(C_4^2){:}C_2$,
$C_2{\times}((C_4{\times}C_2){:}C_2)$, \\
& & & & $C_2{\times}(C_4{:}C_4)$, $C_4{\times}D_8$, $(C_2^4){:}C_2$,
$(C_4{\times}C_2^2){:}C_2$, \\
& & & & $(C_4^2){:}C_2$, $C_2{\times}(C_8{:}C_2)$,
$C_2{\times}QD_{16}$, $C_8{:}(C_2{\times}C_2)$, \\
& & & & $C_2^2{\times}D_8$  \\
\bottomrule
\end{tabular}
\end{center}
\caption{Regular subgroups of small hypercubes. The group names are as returned by \textsf{StructureDescription} in GAP, and $A{:}B$ denotes a split extension $A\rtimes B$; this notation records neither the action nor, in some cases, the group itself.  For instance, the two entries $(C_4^2){:}C_2$ in the $Q_5$ row are non-isomorphic groups, namely $\mathrm{SmallGroup}(32,11)$ and $\mathrm{SmallGroup}(32,34)$.  The GAP ID $(o,i)$ of every group appearing in the census is given in the appendix. }
\label{tab:hypercubes}
\end{table}

\section{Constructing abelian regular subgroups in non-abelian Cayley graphs}\label{sec:abind2}

Recall that by Sabidussi's Theorem~\cite{Sab1958}, a graph $X$ is a Cayley graph of a group $G$ if and only if $\Aut(X)$ contains a regular subgroup isomorphic to $G$ (see also \cite[Section~3.7]{godsilAGT}). 
This naturally implies that a graph can be Cayley for many different groups, since the automorphism group can contain many (non-isomorphic) regular subgroups, and we saw evidence of this in the previous section. In this section, we give some conditions for when Cayley graphs for non-abelian groups are also Cayley graphs for abelian groups, focusing on groups having an abelian subgroup of index two on the one hand, and extraspecial $p$-groups on the other hand.  We will also use freely the standard fact that a transitive abelian
  permutation group is regular: if an element $a$ of such a group $A$
  fixes a point $\omega$, then it fixes $\omega^g$ for every $g\in A$,
  since $\omega^{ga} = \omega^{ag} = \omega^{g}$, and by transitivity
  $a$ fixes every point, so $a = 1$.

All three of our index-two reductions come from a single observation: if some
element outside the abelian index 2 subgroup, $H$, normalizes the connection set, then it can be adjoined to
the right regular action of $H$ to produce an abelian regular subgroup.  We
isolate that observation first, and then record three ways of certifying its
hypothesis.

\begin{lem}\label{lem:abelian-index-two}
    Let $G$ be a finite group with an abelian subgroup $H$ of index two, and
    let $X = \Cay(G,S)$.  If $gSg^{-1} = S$ for some $g\in G\setminus H$, then
    $A = \langle R(H), L_g\rangle$ is an abelian regular subgroup of $\Aut(X)$
    of order $|G|$; in particular, $X$ is a Cayley graph of an abelian group.
    If moreover $g^2 = 1_G$, then $A\cong H\times C_2$.
\end{lem}

\proof  Vertices $x$ and $y$ of $X$ are adjacent if and only if $yx^{-1}\in S$,
  and $(gy)(gx)^{-1} = g(yx^{-1})g^{-1}$, so $L_g$ is an automorphism of $X$ if
and only if $gSg^{-1} = S$; by hypothesis, $L_g\in\Aut(X)$.  By
Lemma~\ref{lem:centralizer-magic}, $L_g$ centralizes $R(G)$ and in particular
$R(H)$, which is abelian because $H$ is; hence $A$ is abelian.  The orbits of
$R(H)$ on $G$ are the cosets $xH$, of which there are exactly two, namely $H$
and $gH$, and $L_g$ maps $H$ to $gH$.  So $A$ is transitive, and a transitive
abelian permutation group is regular; thus $|A| = |G|$, and $X$ is a Cayley
graph of the abelian group $A$ by Sabidussi's theorem. 

Suppose finally that $g^2 = 1_G$.  Every element of $R(H)$ fixes $H$ setwise
while $L_g$ does not, so $L_g\notin R(H)$; since $L_g^2 = L_{g^2}$ is the
identity and $L_g$ centralizes $R(H)$, we get
$A = R(H)\times\langle L_g\rangle\cong H\times C_2$. \qed

The next lemma is an easy consequence of Lemma~\ref{lem:abelian-index-two}, but it also follows from \cite[Theorem~2.3]{morris2021two} together with \cite[Theorem~2.6]{MorrisSkelton} (see Lemma~\ref{lem:gendihedral}). Recall that a Cayley graph is \textsl{conjugacy-closed} if its connection set is closed under conjugation. 

\begin{lem}\label{lem:nonabel-normal-2} Suppose $G$ is a non-abelian group with an abelian subgroup $H$ of index $2$.
    If $\Cay(G,S)$ is conjugacy-closed, then it is a Cayley graph for an abelian group. 
\end{lem}

\proof Since $H$ has index two, we may pick $g\in G\setminus H$.  As $S$ is
conjugacy-closed, $gSg^{-1} = S$, and Lemma~\ref{lem:abelian-index-two}
applies. \qed

\begin{lem}\label{lem:centre-in-H}
    Let $G$ be a non-abelian group with an abelian subgroup $H$ of index two.  Then $Z(G)\le H$.  Consequently, if $s\in G\setminus H$ then $L_s\notin R(G)$.
\end{lem}

\proof Suppose $z\in Z(G)\setminus H$.  Since $H$ has index two, we see that $G = \langle H, z\rangle$, and since $H$ is abelian and $z$ is central, the group $G$ is abelian, a contradiction.  For the second statement, suppose $L_s = R_t$ for some $t\in G$.  Applying both permutations to $1_G$ gives $s = t$, and then $sx = xs$ for all $x\in G$, so $s\in Z(G)\le H$. \qed 

We saw in Section~\ref{sec:lr-actions} that if $X$ is a vertex-transitive graph with PST, then there is a permutation matrix $T$ with order two and no fixed points lying in the centre of $\Aut(X)$ (Theorem~\ref{thm:vxtr-pst}). We will call $T$ the \textsl{PST involution} of $X$.
If we suppose further that $X$ is a Cayley graph, that is $X = \Cay(G,S)$ for some $S\subseteq G$, then there are two cases: 
\begin{enumerate}[label=(\alph*)]
       \item $T$ is an element of $R(G)$ (we equate the elements of $G$ with the permutation matrices corresponding to the right action); or 
       \item $T$ is not an element of $R(G)$.
   \end{enumerate}
We treat these two cases separately. 

First we recall from Lemma~\ref{lem:centralizer-magic} that the centralizer of $R(G)$ in $\Aut(X)$ is $L(G) \cap \Aut(X)$. 
Moreover, $L_g \in \Aut(X)$ if and only if $gSg^{-1} = S$,  thus we see that 
\[
L(G) \cap \Aut(X) = \{L_g \mid gSg^{-1} = S \}. 
\]
Thus the PST involution is $T= L_s$ for some $s \in G$ such that $sSs^{-1} = S $ and so we can always consider $T$ as a left regular action of $G$.

\begin{lem}\label{lem:AbCayTninH}
    Let $G$ be a finite group with an abelian subgroup $H$ of index $2$. Suppose $X = \Cay(G,S)$ admits PST and the PST involution is $T$. If $T = L_s$ for some $s\notin H$, then $\langle R(H), L_s \rangle \leq \Aut(X)$ is an abelian regular subgroup of order $|G|$. In particular, $X$ is a Cayley graph for the abelian group $H\times C_2$. 
\end{lem}

\proof The PST involution $T = L_s$ lies in $\Aut(X)$, so $sSs^{-1} = S$,
and $s\notin H$ by hypothesis.  Since $T$ is an involution with no fixed
points, $s^2 = 1_G$ and $s\neq 1_G$.  Lemma~\ref{lem:abelian-index-two},
applied with $g = s$, gives that $\langle R(H),L_s\rangle$ is an abelian
regular subgroup of $\Aut(X)$ of order $|G|$, isomorphic to $H\times C_2$. \qed

We get the following as a  corollary.

\begin{cor}\label{cor:odd-centre}
    Let $G$ be a finite group with an abelian subgroup $H$ of index $2$, and suppose that the centre of $G$ contains no involution. If the Cayley graph $X = \Cay(G,S)$ admits PST, then $S$ consists of a single involution and $X$ is a perfect matching. In particular, no connected Cayley graph $X = \Cay(G,S)$ admits PST.
\end{cor}

\proof First observe that $G$ is non-abelian: the order of $G$ is even, so if $G$ were abelian, then $Z(G) = G$ would contain an involution by Cauchy's theorem. 

Suppose first that $X$ is connected and, for a contradiction, that it admits PST.  By Lemma~\ref{lem:centralizer-magic}, the PST involution is $T = L_s$ for some $s\in G$ with $s^2 = 1$ and $s \neq 1_G$.  To apply Lemma~\ref{lem:AbCayTninH} we must check $s\notin H$, and it suffices to show $H$ contains no involution.

Pick $g\in G\setminus H$ and let $\iota:H\rightarrow H$ be given by $\iota(h) = ghg^{-1}$ for $h\in H$; this is well defined since index $2$ subgroups are normal, and $\iota^2 = \mathrm{id}$ since $g^2\in H$ and $H$ is abelian.  An element of $H$ fixed by $\iota$ commutes with $g$ and also with the elements of $H$, hence is a central element; conversely, every central element lies in $H$ and is fixed by $\iota$.  Thus the fixed points of $\iota$ are precisely the elements of $Z(G)$.  Now let $\varphi: H\rightarrow H$ be given by $\varphi(h) = h^{-1}\iota(h)$.  Since $H$ is abelian, $\varphi$ is a homomorphism, and its kernel is the set of fixed points of $\iota$, which is $Z(G)$. 
Consider the image, $\varphi(H)$ of $\varphi$. For $h\in H$ we have
\[
\iota\bigl(\varphi(h)\bigr) = \iota(h)^{-1}\,\iota^2(h) = \iota(h)^{-1} h = \varphi(h)^{-1},
\]
so $\iota$ acts on $\varphi(H)$ as inversion.  If $x\in Z(G)\cap \varphi(H)$, then $x = \iota(x) = x^{-1}$, so $x^2 = 1_G$, and since $Z(G)$ contains no involution, $x = 1_G$.  By the first isomorphism theorem $|H| = |Z(G)|\,|\varphi(H)|$, and we conclude that $H = Z(G)\times \varphi(H)$.

Now suppose $h\in H$ satisfies $h^2 = 1_G$, and write $h = zx$ with $z\in Z(G)$ and $x\in \varphi(H)$.  Then $z^2 = 1_G$ and $x^2 = 1_G$.  Since $Z(G)$ contains no involution, $z = 1_G$.  Since $\iota$ acts on $\varphi(H)$ as inversion, $\iota(x) = x^{-1} = x$, so $x$ is a fixed point of $\iota$ and hence $x\in Z(G)\cap \varphi(H) =\{1_G\}$.  Thus $h = 1_G$ and $H$ contains no involution.

Therefore, $s\notin H$, and Lemma~\ref{lem:AbCayTninH} applies and $X$ is a Cayley graph of $H\times C_2$.  Since $H$ has no involution, it has odd order.  By the classification of perfect state transfer on circulants and abelian Cayley graphs \cite{basic2013char,TanFenCao2019}, a connected abelian Cayley graph on $N>2$ vertices admitting perfect state transfer satisfies $4\mid N$.  But since $|H|$ is odd, this gives $|G| = 2|H| \equiv 2\pmod 4$, a contradiction. 

If $X$ is not connected, the result follows from \cite[Theorem 7.1]{arnadottir2022pst}. \qed

Together with the graphs of Section~\ref{sec:kminuspm}, the corollary gives
a complete existence criterion for the groups in question.

\begin{thm}\label{thm:mod4-dichotomy}
Let $G$ be a non-abelian group with an abelian subgroup of index two.  Then $G$
has a connected Cayley graph admitting perfect state transfer if and only if
$4$ divides $|G|$.
\end{thm}

\proof Let $H$ be the abelian subgroup of index two, so that $|G| = 2|H|$ is
even.  Suppose first that $4\nmid |G|$; then $|H|$ is odd, and since
$Z(G)\le H$ by Lemma~\ref{lem:centre-in-H}, the centre of $G$ has odd order and
therefore contains no involution.  By Corollary~\ref{cor:odd-centre}, no
connected Cayley graph of $G$ admits perfect state transfer.

For the converse, suppose that $4\mid |G|$, and write $|G| = 2n$ where $n$ is even.
By Lemma~\ref{lem:H2n} the graph $H_{2n}$ obtained from $K_{2n}$ by deleting a
perfect matching is a Cayley graph of $G$; it is connected, and it admits
perfect state transfer because $n$ is even \cite{TsoPleVeg2008}. \qed

\begin{rem}\label{rem:mod4-dichotomy}
For a non-abelian group $G$ with an abelian subgroup $H$ of index two, the hypothesis of Corollary~\ref{cor:odd-centre} is a condition on the order of the group: the centre of $G$ contains no involution if and only if $|G|\equiv 2\pmod 4$.  Indeed, the proof of the corollary shows that if $Z(G)$ contains no involution then $|H|$ is odd, and conversely, if $|H|$ is odd then $Z(G)\le H$ has odd order.  The proof also describes the groups in question.  We have $H = Z(G)\times \phi(H)$, where $\iota$ acts on $\phi(H)$ as inversion.  Since $g^2$ is fixed by $\iota$, we see that $g^2\in Z(G)$, and since squaring is a bijection on the odd-order group $Z(G)$, we may replace $g$ by $gz$ for a suitable $z\in Z(G)$ so that $g^2 = 1_G$.  Then $G \cong Z(G)\times (\phi(H)\rtimes C_2)$, the direct product of its centre with the generalized dihedral group over $\phi(H)$.  The smallest of these groups that is not itself generalized dihedral is $C_3\times D_6$, of order $18$.
Theorem~\ref{thm:mod4-dichotomy} is consistent with the census: among the
  graphs of Appendix~\ref{app:census} on more than two vertices, exactly one has
  order congruent to $2$ modulo $4$, namely the $30$-vertex graph $T(6)[K_2]$,
  and it is not a  Cayley graph of any group.
\end{rem}

\begin{rem}\label{rem:Heawood}
The PST hypothesis in Lemma~\ref{lem:AbCayTninH} cannot be dropped: a
connected Cayley graph of a group with trivial centre and an abelian
subgroup of index two need not be a Cayley graph of an abelian group at
all. The Heawood graph is a Cayley graph over $D_{14}$  and its
automorphism group, $\PGL(2,7)$, contains no element of order $14$,
hence no regular subgroup isomorphic to $C_{14}$, the unique abelian
group of order $14$.
\end{rem}

If $T = L_s$ with $s\in H$, no
  general reduction is possible: the graphical regular representations of  Section~\ref{sec:grrs} all have $T = R_z = L_z$ for a central involution $z\in H$, and no abelian Cayley presentation. For generalized dihedral groups, however, a theorem of Morris and Smol\v{c}i\'c forces an abelian presentation from a condition on the connection set alone, without requiring perfect state transfer; we state it in our notation and include a short proof.

\begin{lem}[{cf.\ Morris--Smol\v{c}i\'c \cite[Theorem~2.3]{morris2021two}}]\label{lem:gendihedral}
Let $G\cong H\rtimes C_2$ be a generalized dihedral group, where $H$ is abelian and $C_2 = \langle b\rangle$, and let $X=\Cay(G,S)$ with $S = S_0\cup bS_1$, where $S_0,S_1\subseteq H$.  If $S_0$ and $S_1$ are both inverse-closed, then $\langle R(H),L_b\rangle\cong H\times C_2$ is a regular subgroup of $\Aut(X)$; in particular, $X$ is a Cayley graph of the abelian group $H\times C_2$.
\end{lem}

\proof Recall that $b$ inverts $H$.  For $s_0\in S_0$ we have
$bs_0b^{-1} = s_0^{-1}$, and for $s_1\in S_1$ we have
$b(bs_1)b^{-1} = b\,(bs_1b^{-1}) = bs_1^{-1}$, so
\[
    bSb^{-1} = S_0^{-1}\cup bS_1^{-1}.
\]
The two pieces lie in $H$ and in $bH$ respectively, so $bSb^{-1} = S$ if and
only if $S_0$ and $S_1$ are both inverse-closed, which holds by hypothesis.
Since $b\in G\setminus H$ and $b^2 = 1_G$, Lemma~\ref{lem:abelian-index-two}
applied with $g = b$ gives that $\langle R(H),L_b\rangle$ is a regular subgroup
of $\Aut(X)$ isomorphic to $H\times C_2$. \qed

 \begin{rem}
  The hypothesis on $S_0$ follows from $S$ being inverse-closed, but the
  hypothesis on $S_1$ does not, and Section~\ref{sec:dih48} shows that the
  condition on $S_1$ cannot be dropped.
  \end{rem}

This lemma shows that many known examples of non-abelian Cayley graphs admitting PST are in fact statements about abelian Cayley graphs, and thus an instance of the classification in \cite{TanFenCao2019}, which predates \cite{WanFen2023-bicayley}. 
We use the lemma to show that the graphs of the dihedral PST family of Wang and Feng, whose connection sets are not conjugacy-closed, are in fact abelian Cayley graphs  in the next example.

\begin{ex}\label{ex:wangfeng}
Fix $m\ge1$ and let $G$ be the dihedral group of order $16m$, with cyclic subgroup $H=\langle a\rangle$ of order $8m$ and a reflection $b$. Wang and Feng \cite[Example~5.1]{WanFen2023-bicayley} show that the Cayley graph $X = \Cay(G,S)$ with
\[
    S \;=\; \{\,a^{2j-1} : 1\le j\le 4m\,\}\;\cup\;\{\,ba^{2m},\ ba^{6m}\,\}
\]
is a connected $(4m+2)$-regular graph on $16m$ vertices and has perfect state transfer from $x$ to $xa^{4m}$ at time $\pi/2$.  The set $S$ is not closed under conjugation: conjugating $ba^{2m}$ by $a$ gives $ba^{2m-2}\notin S$.  So $X$ is not a conjugacy-closed Cayley graph of $G$, and Lemma~\ref{lem:nonabel-normal-2} does not apply.

We can apply Lemma~\ref{lem:gendihedral}.  In its notation, $S_0 = \{a^{2j-1} : 1\le j\le 4m\}$ and $S_1 = \{a^{2m},\,a^{6m}\}$, and both are inverse-closed: the inverse of an odd power of $a$ is an odd power, and $(a^{2m})^{-1} = a^{6m}$.  Hence $X$ is a Cayley graph of the abelian group $\langle R(H),L_b\rangle\cong C_{8m}\times C_2$.  Tracing the proof of Lemma~\ref{lem:abelian-index-two} makes the abelian presentation explicit: write $C_{8m}\times C_2 = \langle c\rangle\times\langle z\rangle$ and identify $R(a^t)L_b^{\eps}$ with $c^tz^{\eps}$.  The element $c^tz^{\eps}$ moves the vertex $1_G$ to $b^{\eps}a^t$, so the connection set is
\[
    S' \;=\; \{\,c^t : t \text{ odd}\,\}\;\cup\;\{\,c^{2m}z,\ c^{6m}z\,\},
\]
and the transfer $x\mapsto xa^{4m}$ becomes translation by the involution $c^{4m}$.  We verified in SageMath for $m=1,2$, exactly over $\Q(i)$, that $\Cay(G,S)\cong\Cay(C_{8m}\times C_2,\,S')$ and that $U(\pi/2) = -T$ where $T$ is that translation.  (For $m=1$ this is a $6$-regular graph on $16$ vertices with spectrum $\{6,2,0,-2,-6\}$.)  So the perfect state transfer of \cite[Example~5.1]{WanFen2023-bicayley} is a statement about an abelian Cayley graph, that is, an instance of the classification in \cite{TanFenCao2019}.
\end{ex}

  \begin{rem}\label{rem:some-families}
The same treatment disposes of every explicit dihedral, dicyclic and
  generalized dihedral PST family in \cite{CaoFen2021,WanCao2024,WanFen2023-bicayley}.
  The papers \cite{CaoFen2021,WanCao2024} work exclusively with conjugacy-closed  connection sets. Their characterizations (\cite[Theorems~3.1
  and~3.2]{CaoFen2021}, \cite[Theorems~3.2 and~3.3]{WanCao2024}) all have
  the hypothesis $gSg^{-1}=S$, and so, by Lemma~\ref{lem:nonabel-normal-2},
every graph satisfying this hypothesis is a 
  Cayley graph of an abelian group, regardless of whether or not it has PST.  
 The generalized dihedral family \cite[Example~5.2]{WanFen2023-bicayley} over
  $C_{8m}\times C_2 = \langle a\rangle\times\langle h\rangle$ with
  connection set $\{a^{2j-1}\}\cup\{a^{4m}h,\,b\}$ and stated there
  without proof, has $S_0 = \{a^{2j-1}\}\cup\{a^{4m}h\}$ and
  $S_1 = \{1\}$, both inverse-closed, so Lemma~\ref{lem:gendihedral}
  applies exactly as in Example~\ref{ex:wangfeng}. 
  \end{rem}

\begin{rem}
The recent characterization of PST on Cayley graphs over $V_{8n} = \langle a,b \mid a^{2n} = b^4 = 1,\ ba = a^{-1}b^{-1},\ b^{-1}a = a^{-1}b\rangle$ in \cite{KalBha2024} is also stated exclusively for conjugacy-closed connection sets. 
The element $b^2$ is central and $H = \langle a\rangle\times\langle b^2\rangle \cong C_{2n}\times C_2$ is abelian of index $2$, so by Lemma~\ref{lem:nonabel-normal-2} every graph considered by \cite{KalBha2024} is in fact a Cayley graph of the abelian group $C_{2n}\times C_2 \times C_2$.  \end{rem}

\begin{rem}
  Cao and Feng ask, crediting their referee, to ``determine whether a Cayley
  graph over a dihedral group is isomorphic to a Cayley graph over a cyclic
  group'' \cite[Open question~1]{CaoFen2021}. 
  In general it need not be, as the Heawood graph shows (Remark~\ref{rem:Heawood}). Isomorphic Cayley graphs over different groups, and dihedral groups in particular, are an active line of research; see \cite{Joseph1995,Morris1999,morris2021two,MorrisSkelton}.
  Within the scope of \cite{CaoFen2021}, where
  the connection set is closed under conjugation,
  Lemma~\ref{lem:nonabel-normal-2} shows that every such graph is a Cayley
  graph of $C_n\times C_2$, which is cyclic when $n$ is odd.  Moreover, every
  explicit example constructed in \cite{CaoFen2021,WanCao2024} has connection
  set containing the full reflection coset $b\langle a\rangle$, and any such
  Cayley graph of a dihedral or dicyclic group is a circulant. 
  In particular, the two isomorphisms verified by computer in
  \cite[Section~5]{CaoFen2021} are instances of this, and every perfect state
  transfer example of \cite{CaoFen2021,WanCao2024} lies within the circulant
  classification of \cite{basic2013char}.  Assuming perfect state transfer,
  the example of Section~\ref{sec:dih48} shows that from $48$ vertices
  onwards, a dihedral Cayley graph with perfect state transfer need not be a
  Cayley graph of any abelian group.  Finally, abelian cannot be improved to
  cyclic: the graphs of Example~\ref{ex:wangfeng} are Cayley graphs of
  $C_{8m}\times C_2$ whose automorphism groups contain no element of order
  $16m$ (verified for $m=1,2$), so they are not circulants.
  \end{rem}

We now turn our attention to extraspecial groups. A $p$-group $G$ is called \textsl{extraspecial} if its centre, $Z:=Z(G)$, has order $p$ and moreover,
\[Z=G'=\Phi(G)\]
where $G'=[G,G]$ is the commutator subgroup and $\Phi(G)$ is the Frattini subgroup. In this case, $G/Z(G)$ is elementary abelian and $|G|=p^{2n+1}$ for some $n\geq 1$.
  
The following lemma is the extraspecial case of the wreath-product
description of the group association scheme of a class-two Camina
$p$-group; see, for example, \cite{Bag2019} and
\cite[Corollary~5.9]{BasHum2025}. We include a direct proof,
which also identifies the resulting elementary abelian connection set.

\begin{lem}\label{lem:extraspecial}
    Let $G$ be an extraspecial $p$-group of order $p^{2n+1}$ and let $S$ be a conjugacy-closed connection set.  Then $X = \Cay(G,S)$ is a Cayley graph of the elementary abelian $p$-group, $\Z_p^{2n+1}$.
  \end{lem}

\proof Let $Z = Z(G)$; then $|Z| = p$ and $G/Z$ is elementary abelian of order $p^{2n}$, and let $\pi: G\to G/Z$ be the quotient map. We first show that every non-central conjugacy class of $G$ is a full fibre of $\pi$.  Let $g\in G\sm Z$ and define $f: G\to Z$ by $f(x) = xgx^{-1}g^{-1}$.  The values of $f$ lie in $G' = Z$, and $f$ is a homomorphism because $G$ has nilpotency class two, that is, because $G' = Z$ is central: for $x,y\in G$,
\[
    f(xy) \;=\; x\,ygy^{-1}\,x^{-1}g^{-1} \;=\; x\,f(y)g\,x^{-1}g^{-1} \;=\; f(x)\,f(y),
\]
where the second equality substitutes $ygy^{-1} = f(y)g$ and the third holds because $f(y)$ is central.  The image of $f$ is then a subgroup of $Z\cong\Z_p$, and it is nontrivial since $g$ is not central, so $f(G) = Z$.  Since $xgx^{-1} = f(x)g$, the conjugacy class of $g$ is $f(G)g = Zg$, as claimed.

Consequently, $S\sm Z$ is a union of fibres: $S\sm Z = \pi^{-1}(S_0)$, where $S_0 = \pi(S\sm Z)$; note that $1\notin S_0$, because the fibre of $1$ is $Z$.  Define $T = (S_0\times Z)\,\cup\,(\{1\}\times(S\cap Z))$.  Then $T$ is inverse-closed and avoids the identity, since $S$ is inverse-closed and $1\notin S$.  We will show that
\[ X \cong \Cay\bigl(G/Z\times Z,\ T\bigr) =: Y,\]
implying that $\Cay(G,S)$ is a Cayley graph for the abelian group $G/Z\times Z\cong \Z_p^{2n+1}$.

Choose coset representatives $1=g_0,\dots, g_k \in G$, so that every element of $G$ may be written uniquely as $cg_i$ for some $c\in Z$, and the map $cg_i\mapsto (Zg_i,c)$ is a bijection from $G$ to $G/Z\times Z$.  Since $Z$ is central, the vertices $ag_i$ and $bg_j$ are adjacent in $X$ if and only if $ag_i(bg_j)^{-1} = ab^{-1}g_ig_j^{-1}\in S$, and their images are adjacent in $Y$ if and only if
\[(Zg_i,a)(Zg_j,b)^{-1} \;=\; (Zg_ig_j^{-1},\ ab^{-1})\;\in\; T.\]

If $i=j$, then $ag_i(bg_j)^{-1} = ab^{-1}\in Z$ and $(Zg_ig_j^{-1}, ab^{-1}) = (1, ab^{-1})$.  Since $1\notin S_0$, both conditions say that $ab^{-1}\in S\cap Z$: within a fibre, adjacency is determined by $S\cap Z$.

If $i\neq j$, then $ab^{-1}g_ig_j^{-1}\notin Z$, so it lies in $S$ if and only if it lies in $S\sm Z = \pi^{-1}(S_0)$, that is, if and only if $Zg_ig_j^{-1}\in S_0$: between distinct fibres, membership in $S$ depends only on the image in $G/Z$.  On the other side, $Zg_ig_j^{-1}\neq 1$, so $(Zg_ig_j^{-1}, ab^{-1})\in T$ if and only if $Zg_ig_j^{-1}\in S_0$; the second coordinate in $S_0\times Z$ is unrestricted.  The two conditions agree.

In both cases $ag_i(bg_j)^{-1}\in S$ if and only if $(Zg_i,a)(Zg_j,b)^{-1}\in T$, so the map $cg_i\mapsto (Zg_i,c)$ is an isomorphism from $X$ to $Y$. \qed

Considering the case where $p=2$, we immediately obtain the following corollary.
\begin{cor}\label{cor:extrasp2}
    A conjugacy-closed Cayley graph of an extraspecial $2$-group is cubelike.\qed
\end{cor}

  \begin{rem}\label{rem:sinsorci}
  Sin and Sorci \cite{sin2020} construct perfect state transfer on Cayley
  graphs of extraspecial $2$-groups of order $2^{2n+1}$, with connection
  sets built from partial spreads of $G/Z(G)$.  Their connection sets are
  unions of conjugacy classes, so by Corollary~\ref{cor:extrasp2} every
  graph in this family is cubelike, and these examples fall within the older cubelike framework of \cite{bernasconi, cheung, chan2013complex}.
  We note that
  Corollary~\ref{cor:extrasp2} is not a case of the earlier results of
  this section: for $n\ge 2$ an extraspecial $2$-group has no abelian
  subgroup of index two.  
  \end{rem}

\section{Graphs with no abelian regular subgroup}\label{sec:counterexamples}

In Section~\ref{sec:abind2} we considered groups with an abelian subgroup of index two and saw how their Cayley graphs have an abelian structure under any of the three hypotheses: a conjugacy-closed connection set (Lemma~\ref{lem:nonabel-normal-2}); a PST involution $L_s$ with $s\notin H$ (Lemma~\ref{lem:AbCayTninH}, and Corollary~\ref{cor:odd-centre} where this is automatic); or the symmetry hypothesis on $S_1$ in the generalized dihedral case (Lemma~\ref{lem:gendihedral}).
In this section, we show that these hypotheses are necessary by providing examples of Cayley graphs for a group with an abelian subgroup of index two that are not Cayley for an abelian group. 

Our source is the census of connected vertex-transitive graphs on fewer than $48$ vertices \cite{RoyleHolt2020VTGraphs}, which we scanned for perfect state transfer for $10\le n\le 30$ and partially at $n=32$; for every hit we enumerated the conjugacy classes of regular subgroups of $\Aut(X)$ in GAP. For graphs on at most $9$ vertices, we enumerated directly using SageMath. Since the graphs are necessarily integral,
each PST claim is certified by an exact computation of $U(\tau)$ over $\Q(i)$, and each ``no abelian regular subgroup'' claim by a completed subgroup enumeration.

We start with some notable facts revealed by the scan of the census.

\begin{enumerate}
  \item Every connected vertex-transitive graph on at most $22$ vertices with perfect state transfer is a Cayley graph of an abelian group.
  \item On $24$ vertices, there are $14$ graphs with perfect state transfer that are Cayley graphs of only non-abelian groups: their regular subgroups are precisely $\{S_4,\ C_2\times A_4\}$ or  $\{\SL(2,3),\ S_4\}$, they have degrees $9\le k\le 14$, and they all have PST at time $\pi/2$.  
  None of these three groups has an abelian subgroup of index two, so these examples are consistent with the lemmas above.  \item On $30$ vertices, there is a vertex-transitive graph with perfect state transfer that is not a Cayley graph: the lexicographic product $T(6)[K_2]$, where $T(6)$ is the line graph of $K_6$.  Its automorphism group, of order $2^{15}\cdot 720$, contains no regular subgroup (all four groups of order $30$ were checked), and the transfer, at time $\pi/2$, is between the twin pairs.  Within the census this is the smallest vertex-transitive non-Cayley graph with perfect state transfer.  Note that $30\equiv 2\pmod 4$: vertex-transitive graphs escape the divisibility condition that governs the abelian case  \cite{basic2013char,TanFenCao2019,arnadottir2022pst}.
\end{enumerate}

These items underline the theme of this paper: perfect state transfer is not a Cayley phenomenon, let alone an abelian or non-abelian one.  The published constructions of PST graphs happen to be abelian Cayley, but there are genuinely non-abelian examples starting at $24$ vertices.

\subsection{Graphical regular representations}\label{sec:grrs}

We say that $X$ is a \textsl{graphical regular representation (GRR)}  of $G$
if $\Aut(X)$ is regular and isomorphic to $G$. Since the automorphism group of a Cayley graph for a group $G$ always contains $G$ as a subgroup, a GRR has the smallest possible automorphism group that a Cayley graph can have. Moreover, it is exclusively a Cayley graph for that group and no others.  A GRR of a non-abelian group which admits perfect
state transfer is thus the strongest possible counterexample to any hoped-for
reduction to the abelian case.

Such graphs do exist. There are at least $344$ vertex-transitive graphs on $32$ vertices with degree $k \le 12$  that admit
perfect state transfer; of these, there are $38$ that have non-abelian automorphism groups of order $32$.
The groups occurring are the following.

\begin{table}[h]
    \centering
\renewcommand{\arraystretch}{1.2}
\begin{tabular}{@{}rlll@{}}
\toprule
\#\,graphs & $\Aut(X)$ & GAP ID & abelian index-$2$ subgroup \\
\midrule
$28$ & $C_2\times C_2\times D_8$ & $(32,46)$ & $C_2^4$ \\
$7$  & $C_2\times QD_{16}$       & $(32,40)$ & $C_2\times C_8$ \\
$3$  & $C_2\times D_{16}$        & $(32,39)$ & $C_2\times C_8$ \\
\bottomrule
\end{tabular}
    \caption{Groups of order 32 admitting non-abelian GRRs with PST (resulting from a partial scan).}
    \label{tab:GRRgroups}
\end{table}

  The degrees of these graphs are $k\in\{9,10,11,12\}$, and in every case
  the transfer time is $\pi/2$.  Each of the three groups has an abelian
  subgroup of index two but is itself non-abelian.  To our knowledge these
  are the first recorded examples of non-abelian graphical regular
  representations admitting perfect state transfer.

For a representative, let $X$ be the graph on $32$ vertices whose graph6
string is given below.\footnote{In graph6 format,
\Verb[fontsize=\tiny]"_E_@eAaPOgoCGOk?G@O[s?IcbDAPKCobDL@G_?mKPWIGyc?Jh?oHCANFOCo?A^?_F_OcGoTKGI?PwKaJC@B_".}
It is $9$-regular with spectrum
\[
    \{9,\ 5,\ 3^{(6)},\ 1^{(7)},\ (-1)^{(8)},\ (-3)^{(7)},\ (-5)^{(2)}\}
\]
and $\Aut(X)\cong C_2\times QD_{16}$.  Since the spectrum is integral, the transition matrix at $\pi/2$ can be evaluated exactly over $\Q(i)$ from the spectral idempotents and we get $U(\pi/2) \;=\; i\,P$, where $P$ is the permutation matrix of a fixed-point-free central involution of $\Aut(X)$. Therefore, $X$ has perfect state transfer at $\pi/2$. 
We give the connection set explicitly: write
\[
    C_2\times QD_{16} \;=\; \langle z\rangle \times \langle r,s \mid r^8 = s^2 = 1,\ srs = r^3\rangle.
\] 
Then $X = \Cay(C_2\times QD_{16}, S)$ with
\[
    S \;=\; \{r,\ r^{-1}\} \;\cup\; \{z,\ zr^2,\ zr^{-2}\} \;\cup\;
            \{s,\ r^2s,\ zs,\ zr^4s\},
\]
which is inverse-closed of size nine.  The perfect state transfer involution is $T = zr^4$, and the transfer is from $x$ to $zr^4x$ at time $\pi/2$. 

 We now look at these examples in the context of Section~\ref{sec:abind2}.  Each of the three groups in Table~\ref{tab:GRRgroups} has an abelian subgroup  $H$ of index two, listed in the last column of the table, and each graph admits
  perfect state transfer but is a Cayley graph of no abelian group, since its  only regular subgroup is its own non-abelian automorphism group.  
  By  Lemma~\ref{lem:abelian-index-two}, no element of $G\setminus H$ normalizes the  connection set, so all three hypotheses of Section~\ref{sec:abind2} fail here; these graphs show that none of them can be dropped.  Concretely, the perfect  state transfer involution is $T = R_z$ for a central involution $z$, and  $Z(G)\le H$ by Lemma~\ref{lem:centre-in-H}, so $T$ lies in $H$ and  Lemma~\ref{lem:AbCayTninH} does not apply.  Moreover, since $T$ lies in the  centre of $\Aut(X)$ and here $\Aut(X)\cong G$, the group $G$ does not have  trivial centre.

\subsection{A dihedral non-abelian example}\label{sec:dih48}

Here we will present an example of a Cayley graph for a dihedral group which has PST and is not Cayley for an abelian group. This shows that in Lemma~\ref{lem:gendihedral}, the condition on the connection set is necessary. 

Consider the dihedral group of order 48,
\[D_{48} = \langle a,b: a^{24}=b^2 =1, bab = a^{-1}\rangle,\]
and define $S=S_0\cup bS_1$ where 
\begin{align*}
S_0 &= \{a^{\pm2},a^{\pm10}\}\quad\text{and}\\
S_1 &= \{1,a,a^2,a^4,a^7,a^{12},a^{13}, a^{14}, a^{16}, a^{19}\}.
\end{align*}
The Cayley graph $X=\Cay(D_{48},S)$ is 14-regular, with spectrum 
\[\{14,6^{(5)},2^{(5)},0^{(24)},-2^{(6)},-6^{(7)}\}\] 
and it has perfect state transfer from $x$ to $xa^{12}$ at time $\pi/2$. The automorphism group of $X$ has order $2^{28}\cdot 3$, and so verifying that it has no abelian regular subgroups is non-trivial. We will instead look at a certain quotient graph and make use of the following lemma.

\begin{lem}\label{lem:quotientCay}
    Let $G$ be a group, let $S\subseteq G\setminus\{1\}$ be inverse-closed and let $H\le G$.  The left cosets of $H$ form an equitable partition of $\Cay(G,S)$.  If moreover $H$ is normal in $G$, then the quotient graph for this partition is the Cayley graph $\Cay(G/H, \bar S)$, where $\bar S = \{sH : s\in S\}\setminus\{H\}$.
\end{lem}
\proof The cells $gH$ are exactly the orbits of $R(H)\le\Aut(\Cay(G,S))$, and an orbit partition of a group of automorphisms is equitable.  For the second statement, two cells $g_1H$ and $g_2H$ are joined by an edge in the quotient graph $Y$ if and only if $g_2h'(g_1h)^{-1}\in S$ for some $h,h'\in H$. 
If $H$ is normal, this is equivalent to $g_2g_1^{-1}H\in \bar S$, that is, $(g_2H)(g_1H)^{-1}\in\bar S$.  Finally $\bar S$ is inverse-closed because $S$ is, and $H\notin\bar S$ by construction, so $Y = \Cay(G/H,\bar S)$. \qed

Let $H = \langle a^{12}\rangle \leq D_{48}$. Since $a^{12}$ is central in $D_{48}$, $H$ is a normal subgroup. The quotient graph of $X$ over the cosets of $H$, which we call $Y$, is a graph on $24$ vertices and its graph6 string is given below\footnote{In graph6 format \Verb"WqUdAcKB?cCG_HOE_Bg@[?TcAh_Ib?TCdTCKiaBBp@AUaAO"
}. This graph has automorphism group of order 48, and it is easy to check that none of its regular subgroups are abelian.

  Suppose $X$ were a Cayley graph of an abelian group; that is, suppose
  $\Aut(X)$ contained an abelian regular subgroup $A$.  Identify the
  vertices of $X$ with the elements of $A$, so that $A$ acts as its right
  regular representation.  An abelian regular subgroup is its own
  centralizer in $\Sym(V(X))$: by Lemma~\ref{lem:centralizer-magic} the
  centralizer of the right regular action is the left regular action, and
  for an abelian group the two actions coincide.  The PST involution $T$
  lies in the centre of $\Aut(X)$ by Theorem~\ref{thm:vxtr-pst}, so $T$
  centralizes $A$, and therefore $T\in A$.  Since $A$ is regular, $T$ is
  translation by an involution $z\in A$, and the orbits of $T$ are the
  cosets of $\langle z\rangle$ in $A$.  On the other hand, the perfect
  state transfer at time $\pi/2$ is from $x$ to $xa^{12}$, so $T$ maps
  each vertex $x$ to $xa^{12}$ and its orbits are the cosets of $H$.  The
  two coset partitions are therefore the same partition of $V(X)$, and by
  Lemma~\ref{lem:quotientCay}, applied to $A$ and its normal subgroup
  $\langle z\rangle$, the quotient graph over this partition, which is
  $Y$, is a Cayley graph of the abelian group $A/\langle z\rangle$.  This
  contradicts the computation above, and so $X$ is not a Cayley graph of
  any abelian group.

\begin{rem}
    If we replace the dihedral group in our example with the dicyclic group of order 48,
    \[\langle a,b: a^{24}=1, b^2 =a^{12}, b^{-1}ab = a^{-1}\rangle,\]
    and take the same connection set, then we get the same Cayley graph, $X$.
\end{rem}

\subsection{General linear group examples of Pantangi and Sin}\label{sec:pansin}
In this section, we show that  the infinite family of Pantangi and Sin over $\SL(2,q)$, $q$ odd, are examples of Cayley graphs of exclusively non-abelian groups with PST.

Pantangi and Sin \cite{PanSin2026} construct perfect state transfer on
Cayley graphs of the groups $\SL(2,q)$, $\GL(2,q)$ and
$\GU(2,q)$ for odd prime powers $q$, with connection sets that
are unions of conjugacy classes; the state transfer is from $g$ to $-g$
\cite[Theorems~4.8, 4.15, 4.16]{PanSin2026}.  Here
$\GU(2,q)$ denotes the group preserving a non-degenerate
Hermitian form over $\FF_{q^2}$.  Throughout this section $q$ is odd,
$G$ is one of the three groups above, $t=-I$ is its central involution,
and $X$ denotes the corresponding graph.  For the family we focus on over
$\SL(2,q)$, the connection set is $\{t\}\cup S$, where $S$ is
the set of elements of order $p$ or $2p$, with
$p=\operatorname{char}\FF_q$; for the other two groups we write the
connection set as $S$ and refer to \cite{PanSin2026} for its
definition.  In all three cases $S$ is a union of conjugacy classes
and $S t=S$.

We call two vertices  $u,v$ \textsl{twins} if they have the same neighbours in $V\setminus \{u,v\}$. This definition has been chosen such that $u,v$ are twins if and only if the transposition $(uv)$ is an automorphism. Note that this terminology includes both adjacent and non-adjacent twins.  Since the automorphisms of a graph form a group and $(uv)(vw)(uv)=(uw)$, ``equal or twins'' is an equivalence relation; its classes are the \textsl{twin classes}, and each of them is a clique or an independent set, because the transpositions above take edges inside a class to edges.  For graphs $Y$ and $Z$, the \textsl{lexicographic product} $Y[Z]$ has  vertex set $V(Y)\times V(H)$, and $(y_1,z_1)$ is adjacent to $(y_2,z_2)$  if $y_1$ is adjacent to $y_2$ in $Y$, or $y_1=y_2$ and $z_1$ is adjacent  to $z_2$ in $Z$; equivalently, $A(Y[Z]) = A(Y)\otimes J + I\otimes A(Z)$. We mainly use $Y[K_2]$ and $Y[2K_1]$, which replace each vertex of $Y$ by an adjacent, respectively non-adjacent, pair of twins.
See Figure~\ref{fig:twin-quotient} for an illustration of this reduction.

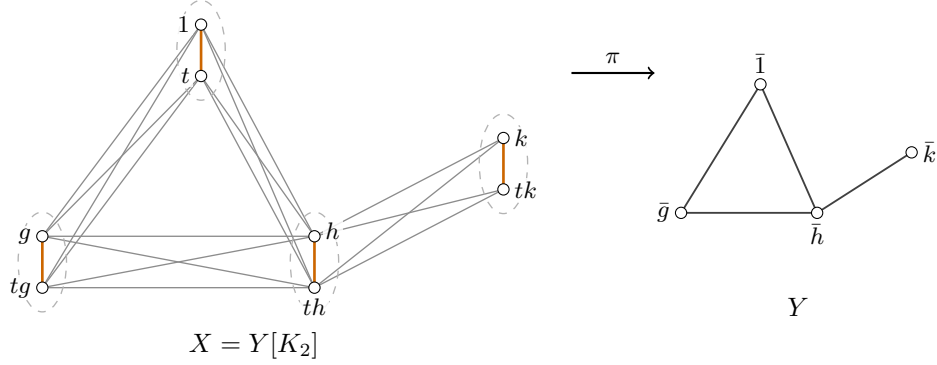
\begin{figure}[t]
\centering
\begin{tikzpicture}[
  cross/.style={draw=black!45, line width=0.5pt},
  pair/.style={draw=myorange, line width=1.0pt},
  yedge/.style={draw=black!75, line width=0.7pt},
  blob/.style={draw=black!30, dashed, line width=0.5pt},
  vx/.style={circle, draw=black, fill=white, inner sep=1.5pt},
  lb/.style={inner sep=1.0pt, fill=white, fill opacity=0.85, text opacity=1, font=\small}]
  \coordinate (P1) at (0, 1.8);
  \coordinate (Pg) at (-2.1, -1.0);
  \coordinate (Ph) at (1.5, -1.0);
  \coordinate (Pk) at (4.0, 0.3);
  \coordinate (a1) at (0, 2.14);   \coordinate (a2) at (0, 1.46);
  \coordinate (b1) at (-2.1,-0.66);\coordinate (b2) at (-2.1,-1.34);
  \coordinate (c1) at (1.5,-0.66); \coordinate (c2) at (1.5,-1.34);
  \coordinate (d1) at (4.0,0.64);  \coordinate (d2) at (4.0,-0.04);
  \foreach \u/\v in {a1/b1,a1/b2,a2/b1,a2/b2, a1/c1,a1/c2,a2/c1,a2/c2, b1/c1,b1/c2,b2/c1,b2/c2, c1/d1,c1/d2,c2/d1,c2/d2}{
    \draw[cross] (\u) -- (\v);}
  \draw[pair] (a1)--(a2); \draw[pair] (b1)--(b2); \draw[pair] (c1)--(c2); \draw[pair] (d1)--(d2);
  \draw[blob] (P1) ellipse [x radius=0.32, y radius=0.66];
  \draw[blob] (Pg) ellipse [x radius=0.32, y radius=0.66];
  \draw[blob] (Ph) ellipse [x radius=0.32, y radius=0.66];
  \draw[blob] (Pk) ellipse [x radius=0.32, y radius=0.66];
  \foreach \p in {a1,a2,b1,b2,c1,c2,d1,d2}{ \node[vx] at (\p) {}; }
  \node[lb, anchor=east, xshift=-3pt] at (a1) {$1$};
  \node[lb, anchor=east, xshift=-3pt] at (a2) {$t$};
  \node[lb, anchor=east, xshift=-3pt] at (b1) {$g$};
  \node[lb, anchor=east, xshift=-3pt] at (b2) {$tg$};
  \node[lb, anchor=west, xshift=3pt, yshift=2pt] at (c1) {$h$};
  \node[lb, anchor=north, yshift=-3pt] at (c2) {$th$};
  \node[lb, anchor=west, xshift=3pt] at (d1) {$k$};
  \node[lb, anchor=west, xshift=3pt] at (d2) {$tk$};
  \node at (0.7, -2.1) {$X = Y[K_2]$};
  \draw[->, thick] (4.9, 1.5) -- (6.0, 1.5) node[midway, above, font=\small] {$\pi$};
  \coordinate (Q1) at (7.4, 1.35);
  \coordinate (Qg) at (6.35, -0.35);
  \coordinate (Qh) at (8.15, -0.35);
  \coordinate (Qk) at (9.4, 0.45);
  \draw[yedge] (Q1)--(Qg); \draw[yedge] (Q1)--(Qh); \draw[yedge] (Qg)--(Qh); \draw[yedge] (Qh)--(Qk);
  \foreach \q in {Q1,Qg,Qh,Qk}{ \node[vx] at (\q) {}; }
  \node[lb, anchor=south, yshift=3pt] at (Q1) {$\bar 1$};
  \node[lb, anchor=east, xshift=-3pt] at (Qg) {$\bar g$};
  \node[lb, anchor=north, yshift=-3pt] at (Qh) {$\bar h$};
  \node[lb, anchor=west, xshift=3pt] at (Qk) {$\bar k$};
  \node at (7.9, -1.6) {$Y$};
\end{tikzpicture}
\caption{Illustration of the twin reduction.  If every twin class of $X$ has size two and distinct classes are joined by all four possible edges or by none, then $X$ is the lexicographic product $Y[K_2]$ (adjacent classes, as drawn) or $Y[2K_1]$ (non-adjacent classes) of the quotient graph $Y$ over the classes.  For the graphs of \cite{PanSin2026} the classes are the pairs $\{g,tg\}$ with $t = -I$ (Lemma~\ref{lem:pansin-structure}, Appendix~\ref{app:pansin-glgu}), perfect state transfer runs from $g$ to $tg$ inside each class (orange), and $\Aut(X) = C_2\wr\Aut(Y)$ by Lemma~\ref{lem:twin-reduction}.}
\label{fig:twin-quotient}
\end{figure}

Kirkland, Monterde and Plosker \cite{KirMonPlo2023} prove that a graph admits perfect state transfer between $u_1, u_2$, a twin class of size exactly two, at time $\tau$ if and only if $U(\tau)_{u_1,u_1} = 0$.  The pairs $\{u_1,u_2\}$ below are twins in $Y[K_2]$ and in $Y[2K_1]$, and writing the entry $U(\tau)_{u_1,u_1} = 0$ of these lexicographic products in terms of the walk on $Y$ turns their criterion into the statement below. The even case of the final sentence is also a special case of a lexicographic product construction of Ge, Greenberg, Perez and Tamon \cite{GeGrePerTam2011}, and the odd case with $Y = K_m$ is the cocktail party example worked out in \cite{KirMonPlo2023}.

\begin{prop}[{\cite{GeGrePerTam2011,KirMonPlo2023}}]\label{prop:doubling}
    Let $Y$ be a graph and $u$ a vertex of $Y$, and let $\{u_1,u_2\}$ be the corresponding twin pair in $Y[K_2]$, respectively in $Y[2K_1]$.  Write $\gamma = U_Y(2\tau)_{u,u}$.  Then $Y[K_2]$ admits perfect state transfer from $u_1$ to $u_2$ at time $\tau$ if and only if $\gamma = -e^{-2i\tau}$, and $Y[2K_1]$ does so if and only if $\gamma = -1$.  In particular, if every eigenvalue of $Y$ is an even integer then $Y[K_2]$ has perfect state transfer at time $\pi/2$ between every twin pair, and if every eigenvalue of $Y$ is an odd integer then $Y[2K_1]$ does.
\end{prop}

We note that perfect state transfer between twins occurs only when their twin class has size exactly two
\cite[Corollary~1]{KirMonPlo2023};  indeed, if $w$ were a third twin,
the transposition $(u_2\,w)$ would be an automorphism fixing $u_1$, so
$U(\tau)_{u_1,u_2}=U(\tau)_{u_1,w}$, and two entries of equal modulus
in a row of a unitary matrix have modulus at most $1/\sqrt2$.  The
proposition is nevertheless (correctly) stated for all $Y$: the twin class of $\{u_1,u_2\}$ is larger than the pair precisely when $u$
itself has an adjacent twin in $Y$, for $Y[K_2]$ (respectively, a
non-adjacent one, for $Y[2K_1]$); for such pairs the criterion on $\gamma$ then fails by the
equivalence established in the proof. 

The $30$-vertex example of Section~\ref{sec:counterexamples} is the case $Y = T(6)$, whose eigenvalues are $8, 2, -2$; the Shrikhande graph and $C_4$ are two more such $Y$, and $K_4[2K_1] = H_8$ recovers a graph from Section~\ref{sec:kminuspm}.

Every automorphism of a graph permutes the twin classes.  The next lemma turns this observation into a reduction tool for lexicographic products with $K_2$ and $2K_1$. The equality $\Aut(X)=C_2\wr\Aut(Y)$ is the case $Z\in\{K_2,\,2K_1\}$ of
  Sabidussi's theorem on automorphism groups of lexicographic products
  \cite{Sab1959}: the hypothesis that every twin class of $X$ has size two
  says precisely that no two vertices of $Y$ have the same closed
  (respectively open) neighbourhood.  We include a short proof to keep the
  treatment self-contained.

\begin{lem}\label{lem:twin-reduction}
Let $Y$ be a graph and let $X=Y[K_2]$ or $X=Y[2K_1]$.  If every twin
class of $X$ has size two, then $\Aut(X)=C_2\wr\Aut(Y)$, of order
$2^{|V(Y)|}\,|\Aut(Y)|$, and $X$ is a Cayley graph of an abelian group
if and only if $\Aut(Y)$ contains an abelian regular subgroup.
\end{lem}

 \proof Write $V(X)=V(Y)\times\{1,2\}$ and call the sets
$F_y=\{(y,1),(y,2)\}$ \textsl{fibres}.  The two vertices of a fibre are
twins, so every twin class is a union of fibres; under the hypothesis
the twin classes are therefore exactly the fibres.  Every automorphism
of $X$ maps twin classes to twin classes, hence permutes the fibres,
and the induced permutation of $V(Y)$ is an automorphism of $Y$,
because two distinct fibres are joined by all four possible edges or by
none, according as the underlying vertices are adjacent in $Y$ or not.
This gives a homomorphism $\varphi\colon\Aut(X)\to\Aut(Y)$.  It is
surjective, since $a\in\Aut(Y)$ lifts to the automorphism
$(y,j)\mapsto(a(y),j)$, and its kernel consists of the maps fixing
every fibre setwise. Since the two vertices of a fibre are twins, these
maps (swapping the vertices in an arbitrary set of fibres, while fixing the rest) are all automorphisms, so the kernel is $C_2^{V(Y)}$ and
$\Aut(X)=C_2^{V(Y)}\rtimes\Aut(Y)=C_2\wr\Aut(Y)$.

If $A\le\Aut(Y)$ is abelian and regular, let $\widetilde A$ be its lift
fixing the second coordinate, and let $\sigma$ be the simultaneous swap
of the two vertices in every fibre.  Then $\sigma$ commutes with
$\widetilde A$, and $\langle\widetilde A,\sigma\rangle\cong A\times
C_2$ is abelian of order $|V(X)|$ and transitive, hence regular; so $X$
is a Cayley graph of the abelian group $A\times C_2$.  Conversely, if
$B\le\Aut(X)$ is abelian and regular, then $\varphi(B)\le\Aut(Y)$ is
abelian, and it is transitive because $B$ is; a transitive abelian
permutation group is regular (as noted in Section~\ref{sec:abind2}).\qed

For the remainder of the section we specialize to
$G=\SL(2,q)$ with $q\ge5$, whose graphs provide the infinite
family of exclusively non-abelian Cayley graphs with PST; the $\GL$ and $\GU$ families reappear at the
end of the section.  So $X=\Cay(G,\{t\}\cup S)$, where $S$ is the set
of elements of order $p$ or $2p$.  The elements of order $p$ are
exactly the $q^2-1$ non-identity matrices with both eigenvalues equal
to $1$, and the elements of order $2p$ are their negatives, so
$|S|=2(q^2-1)$ and multiplication by $t$ interchanges the two kinds.
Write $H:=\mathrm{PSL}(2,q)=G/\langle t\rangle$ and let
$\pi\colon G\to H$ be the quotient map.  Since $\pi$ identifies $s$
with $ts$, the set $\bar S:=\pi(S)$ consists of the $q^2-1$ elements of
order $p$ in $H$; it is closed under conjugation and inversion, and we
put
\[
    Y:=\Cay(H,\bar S).
\]
For $q=3$ the group $H=\mathrm{PSL}(2,3)\cong A_4$ is not simple, and indeed, $\bar S$ then consists of the eight elements of order
  $3$ in $A_4$, so $Y$ is the complete tripartite graph $K_{4,4,4}$, a
  circulant, and $X=Y[K_2]$ is a circulant as well: writing
  $C_{24}=\langle c\rangle$,
  \[
      X=\Cay(C_{24},\{c^k:3\nmid k\}\cup\{c^{12}\}).
  \]
Thus the hypothesis $q\ge 5$ in the remainder of the section is necessary.

\begin{lem}\label{lem:pansin-structure}
Let $q\ge5$.  The twin classes of $X$ are exactly the pairs
$\{g,tg\}$, and $X=Y[K_2]$.
\end{lem}

\proof   First, note that $\SL(2,q)$ has precisely one proper, non-trivial normal subgroup for $q\geq 5$, namely its centre, $Z(G) = \langle t\rangle$. Indeed, if $M$ is a normal subgroup of $G$, then $\pi(M)$ is a normal subgroup of the simple group $H$  \cite[Section~7.7]{dixon1996}, so $\pi(M)$ is trivial or all of $H$. In the first case $M\le\ker\pi = \langle t\rangle$.   In the second  case $G = M\langle t\rangle$, so $M$ has index at most two in $G$; but  the elementary unipotent matrices generate $G$ and have odd order $p$,  so $G$ has no quotient of order two, and hence $M = G$.

Consider the set $T = S\cup\{1,t\}$. It is conjugacy-closed, and easy calculation shows that the set of elements that fix it, i.e.\ the set $N = \{g\in G: Tg = T\}$ is a normal subgroup of $G$. Further, $N$ is contained in $T$ because $1\in T$, and by definition,   $|T|=2q^2<q^3-q=|G|$ for $q\ge3$ so $G\sm N$ is non-empty and thus $N\neq G$. We see that $t\in N$, and now this implies that $N = \langle t\rangle.$ 

The neighbourhood of a vertex $g$ in $X$ is the set $(S\cup\{t\})g$. Therefore, two distinct vertices $g,h\in G$ are adjacent twins if and only if 
\[(S\cup\{t\})g \cup \{g\} = (S\cup\{t\})h \cup \{h\}.\]
Since $t$ is the only non-trivial element fixing $T$, this implies that $1$ has precisely one adjacent twin, namely $t$. It is clearly impossible for a vertex to have both an adjacent and a non-adjacent twin, and we have shown that $\{1,t\}$ is a twin class of $X$. By vertex-transitivity, this implies that $\{g,tg\}$ are the twin classes of $X$.

Since each pair $\{g,tg\}$ is a twin class, distinct pairs are joined by all four possible edges or by none, and $g$ is adjacent to $tg$. By Lemma~\ref{lem:quotientCay}, the quotient over the cosets of $\langle t\rangle$ is $Y$, and so $X=Y[K_2]$. \qed 

\begin{rem}\label{rem:pansin-parity}
Pantangi and Sin prove that the eigenvalues of $X$ belonging to the
characters $\chi$ of $G$ with $\chi(t)=\chi(1)$ (equivalently, the
characters of $H$) are congruent to $1$ modulo $4$
\cite[Section~4.3]{PanSin2026}.  Since $X=Y[K_2]$, such an eigenvalue
equals $2\eta+1$, where $\eta$ is the corresponding eigenvalue of $Y$.
Hence every eigenvalue of $Y$ is an even integer, and the perfect state
transfer in $X$ from $g$ to $tg$ at time $\pi/2$ is the even case of
Proposition~\ref{prop:doubling}.
\end{rem}

  The proof of Theorem~\ref{thm:pansin-sl} uses some notions from the theory of
  permutation groups, for which we refer to \cite[Sections~1.5 and~4.3]{dixon1996}:
  a transitive permutation group is \textsl{primitive} if it preserves no
  partition of the point set other than the two trivial ones; the \textsl{socle},
  $\soc(G)$, of a finite group $G$ is the subgroup generated by its minimal normal
  subgroups (normal subgroups that contain no non-trivial normal subgroup of $G$
  properly); and a primitive group $G\le\Sym(\Omega)$ is of \textsl{simple
  diagonal type} if $\soc(G)$ is a direct power $T^k$ of a non-abelian simple
  group with $k>1$ and the stabilizer in $T^k$ of a point of $\Omega$ is a
  diagonal subgroup of $T^k$.  We also use the fact that, for odd $q=p^f$, the
  outer automorphism group of $\mathrm{PSL}(2,q)$ has order $2f$, generated by
  the images of conjugation by a diagonal matrix with non-square determinant and
  of the map raising every matrix entry to the $p$-th power
  \cite[Sections~2.8 and~7.7]{dixon1996}.

\begin{lem}[{Dickson; see \cite[Theorem~II.8.27]{huppert}}]\label{lem:psl-abelian}
Let $q = p^f\ge5$ be odd and let $H = \mathrm{PSL}(2,q)$.  Every abelian subgroup of $H$ is an elementary abelian $p$-group of order dividing $q$, a cyclic group of order dividing $(q-1)/2$ or $(q+1)/2$, or a Klein four-group.  In particular, every abelian subgroup of $H$ has order at most $q$, and the orders of the abelian subgroups of $\mathrm{PSL}(2,5)$, $\mathrm{PSL}(2,7)$ and $\mathrm{PSL}(2,9)$ are $\{1,2,3,4,5\}$, $\{1,2,3,4,7\}$ and $\{1,2,3,4,5,9\}$, respectively.
\end{lem}

We also use two results about primitive groups of simple diagonal type,
  which we state in the forms in which they are applied.  The first is the case of
  Praeger's inclusion theorem in which the smaller group has this type; it rests
  on the classification of finite simple groups.

  \begin{prop}[{Praeger \cite[Proposition~8.1]{Praeger1990}; stated in this form
  as \cite[Lemma~2.3]{KajMor2023}}]\label{prop:praeger-sd}
Let $G\le G_1\le\Sym(\Omega)$ with $|\Omega| = n$, where $G$ and $G_1$ are both
primitive and $G$ is of simple diagonal type.  Then either $G_1$ is $A_n$ or
$S_n$, or $G_1$ is of simple diagonal type with
$\soc(G_1) = \soc(G)$.
\end{prop}

The second fact describes the largest group in which the argument takes place.
Here $T\times T$ is identified with the group of permutations
$x\mapsto axb^{-1}$ of $T$, whose point stabilizer at $1$ is the diagonal
subgroup, so that any primitive group between $T\times T$ and its normalizer is
of simple diagonal type.

\begin{prop}[{see \cite[3.4]{Praeger1990}; \cite[Section~4.5]{dixon1996}}]\label{prop:sd-normalizer}
Let $T$ be a non-abelian simple group and let $D = N_{\Sym(T)}(T\times T)$.
Then $D = \langle\iota\rangle\ltimes\operatorname{Hol}(T)$, where
$\operatorname{Hol}(T) = \Aut(T)\ltimes T$ and $\iota\colon x\mapsto x^{-1}$.
Consequently $T\times T$ is normal in $D$ with
$D/(T\times T)\cong\Out(T)\times C_2$, the $C_2$ being generated by the image
of $\iota$; every element of $D$ either preserves or interchanges the two
direct factors of $T\times T$; and an element interchanges them if and only if
it lies outside $\operatorname{Hol}(T)$, that is, if and only if its image in
$D/(T\times T)$ lies outside $\Out(T)$.
\end{prop}

\begin{thm}\label{thm:pansin-sl}
For every odd prime power $q\ge5$, the $\SL(2,q)$ graph of
\cite[Theorem~4.16]{PanSin2026} is not a Cayley graph of any abelian
group.
\end{thm}

\proof By Lemmas~\ref{lem:twin-reduction} and
\ref{lem:pansin-structure}, it suffices to show that $\Aut(Y)$
contains no abelian regular subgroup.  Write $m:=|H|=q(q^2-1)/2$.

Since $\bar S$ is closed under conjugation and inversion, all maps
$x\mapsto axb^{-1}$ with $a,b\in H$, as well as the inversion map
$\iota\colon x\mapsto x^{-1}$, are automorphisms of $Y$.  The two-sided
translations form a group $M_0\cong H\times H$, and the stabilizer in
$M_0$ of the vertex $1$ is the diagonal subgroup $\Delta H$.  Since $H$
is simple, $\Delta H$ is maximal in $H\times H$, so $M_0$ is primitive.
The inversion map normalizes $M_0$ and interchanges its two simple direct
factors, and so $H\times \{1\}$ and $\{1\}\times H$ are not normal in $M:=\langle M_0,\iota\rangle$. Since they are the only proper non-trivial normal subgroups of $M_0$, this implies that
$M_0\cong H\times H$ is a minimal normal subgroup of
$M$.  The group $M$ is primitive because it
contains the primitive group $M_0$, while this minimal normal subgroup is
not regular, since its stabilizer of $1$ is $\Delta H$.  It follows from
\cite[Theorem~4.3B]{dixon1996} that
$\soc(M)=M_0\cong H\times H$.  Thus $M$ is a 
  primitive group of simple diagonal type in the O'Nan--Scott classification (see 
  \cite[Section~4.5]{dixon1996}; \cite{Praeger1990}).

The graph $Y$ is neither empty nor complete, since $0<q^2-1<m-1$, so
$\Aut(Y)$ is not $2$-transitive and in particular does not contain
$A_m$. 
Moreover, $\Aut(Y)$ is primitive because it contains
$M$.  Proposition~\ref{prop:praeger-sd}, applied to $M\le\Aut(Y)$, says that
$\Aut(Y)$ is $A_m$ or $S_m$, or is itself of
simple diagonal type with the same socle as $M$.  The first two
alternatives are excluded by the preceding observation.  Consequently
$\soc(\Aut(Y))=H\times H$ and
\[
  \Aut(Y)\le D:=N_{\Sym(H)}(H\times H).
\]
By Proposition~\ref{prop:sd-normalizer} applied to $T = H$, the subgroup
$H\times H$ is normal in $D$, every element of $D$ either preserves or
interchanges the two direct factors, and
$D/(H\times H)\cong\Out(H)\times C_2$, where the $C_2$ is generated by
the image of $\iota$ and an element of $D$ maps onto it if and only if it
interchanges the factors.

Suppose that $A\le D$ is abelian and regular.  Let
$B_0:=A\cap(H\times H)$ and $c:=|A:B_0|$.  Since $A/B_0$ embeds into
$D/(H\times H)\cong\Out(H)\times C_2$, whose order is $4f$, the
index $c$ divides $4f$.  
Since $A$ is regular of degree $m$ we have $|A| = m$, so $|B_0| = m/c$.
Let $B_1$ and $B_2$ be the projections of $B_0$ to the two direct factors, so
that $B_0\le B_1\times B_2$ and hence $|B_0|\le|B_1|\,|B_2|$.
These projections are abelian, so each has order at most $q$ by
Lemma~\ref{lem:psl-abelian}, and hence
\[
  \frac{q(q^2-1)}{2c}=\frac{m}{c}\le|B_1|\,|B_2|\le q^2,
  \qquad\text{so}\qquad q^2-1\le2cq\le8fq.
\]
For $f=1$ this forces $q\le8$, for $f=2$ it forces $q\le16$, and for
$f\ge3$ it is impossible because $q\ge3^f>8f$.  Thus only
$q\in\{5,7,9\}$ remain, and for these the inequality $q^2-1\le2cq$
together with $c\mid4f$ determines $c$ and therefore $|B_0|$:
\[
\begin{array}{c|ccc}
q&5&7&9\\ \hline
c&4&4&8\\
|B_0|&15&42&45.
\end{array}
\]

For $q=7$, Lemma~\ref{lem:psl-abelian} gives that the possible
orders of abelian subgroups of $H=\mathrm{PSL}(2,7)$ are
$1,2,3,4,7$.  Since $B_0\le B_1\times B_2$,
the order $42$ of $B_0$ would have to divide $|B_1|\,|B_2|$, and no
product of two numbers in this list is divisible by $42$.

For $q=5$, we have $H\cong A_5$
\cite[Section~7.7]{dixon1996}, and the possible orders are $1,2,3,4,5$ by Lemma~\ref{lem:psl-abelian}. Then the only pair with product divisible by $15$ is $|B_1|=3$ and $|B_2|=5$ and in this case we see that $B_0=B_1\times B_2$.  For $q=9$, the
possible orders of abelian subgroups of $H\cong A_6$ are
$1,2,3,4,5,9$ (Lemma~\ref{lem:psl-abelian} again), and the same reasoning gives
$|B_1|=5$, $|B_2|=9$ and
$B_0=B_1\times B_2$.  
In either case $B_0=B_1\times B_2$ with $B_1,B_2$
  non-trivial, so $B_0$ contains an element $(b_1,1)$ with $b_1\neq1$. If $d\in D$  
  interchanges the two factors, then $d(b_1,1)d^{-1}\in\{1\}\times H$, so $d$ does not 
  centralize $B_0$. Since $A$ is abelian, $A\le C_D(B_0)$, so no element of $A$ 
  interchanges the factors;  hence the
embedding $A/B_0\hookrightarrow\Out(H)\times C_2$ has image in
the factor-preserving subgroup $\Out(H)$, and therefore $c$
divides $|\Out(H)| = 2f$.  This contradicts $c = 4$ for
$q = 5$, where $2f = 2$, and $c = 8$ for $q = 9$, where $2f = 4$.
Hence no abelian regular subgroup exists for any $q\ge5$.\qed

We conclude this section with a brief discussion of the other two
families of \cite{PanSin2026}: the graphs on $\GL(2,q)$ and on
$\GU(2,q)$ \cite[Theorems~4.8 and~4.15]{PanSin2026}.  For $q=3$ the linear family degenerates: the connection set of \cite[Theorem~4.8]{PanSin2026} is
all of $\GL(2,3)\sm\{\pm I\}$, so the graph is $K_{48}$ minus a
perfect matching, a Cayley graph of every group of order $48$
(Section~\ref{sec:kminuspm}). It is possible that this degenerate case is
the only member of the two families that is a Cayley graph of an
abelian group; the four smallest members beyond it are not, by the computation of Appendix~\ref{app:pansin-glgu}. 

\begin{prop}\label{prop:pansin-glgu}
The following graphs of \cite{PanSin2026} are not Cayley graphs of any
abelian group: the graph
$\Gamma'=\Cay(\GL(2,3),\mathfrak T)$ of
\cite[Section~4.1]{PanSin2026}, where $\mathfrak T$ is the set of
non-central elements of order $2$, $3$, $4$ or $6$ (so $\mathfrak T$
omits the elements of order $8$, and $\Delta$ is not the degenerate member
of the $\GL(2,q)$ family discussed above); the
$\GU(2,3)$ graph on $96$ vertices; the $\GL(2,5)$ graph
on $480$ vertices; and the $\GU(2,5)$ graph on $720$ vertices.
\end{prop}

\proof This is verified by an exhaustive computation in SageMath
\cite{sage}; the computation is described in
Appendix~\ref{app:pansin-glgu}.\qed

\begin{prob}\label{prob:pansin-stability}
Decide whether the $\GL(2,q)$ and $\GU(2,q)$ graphs of
\cite{PanSin2026} with $q\ge7$ are Cayley graphs of abelian groups.
\end{prob}

\section{Peak state transfer}\label{sec:peak-ref}

In keeping with the theme of this paper, when we find peak state transfer
on a Cayley graph we do not stop at one presentation: we determine every
group over which the graph is a Cayley graph.  By Sabidussi's theorem
\cite{Sab1958}, these are exactly the isomorphism types of regular
subgroups of the automorphism group.  In this section we carry this out
for an infinite family of graphs with peak state transfer at time $\pi/3$
whose amount does not decay as the family grows, and whose Cayley groups
are all non-abelian.

\subsection{A graphical regular representation with peak state transfer}

Let $F_{20}$ denote the Frobenius group of order $20$,
\[
    F_{20} = \langle a,b \mid a^5 = b^4 = 1,\ bab^{-1} = a^2 \rangle,
\]
which is $\mathrm{SmallGroup}(20,3)$ in GAP.  Writing $(i,j)$ for
$a^ib^j$, let
\[
    S = \{(0,1),(0,2),(0,3),(1,0),(1,1),(1,2),(1,3),
          (2,0),(2,3),(3,0),(3,1),(3,3),(4,0),(4,1)\}.
\]
Then $S$ is inverse-closed, and $X_{14} = \Cay(F_{20}, S)$ is a connected
$14$-regular graph on $20$ vertices.
\footnote{In graph6 format, $X_{14}$
  is \Verb"S|X~s|~ztib|qvZNz[~~pn~rq~my~fh^k".}

This graph was found in a census of the vertex-transitive graphs on fewer than $48$ vertices
\cite{RoyleHolt2020VTGraphs}.  A computation shows that
$|\Aut(X_{14})| = 20$; hence $\Aut(X_{14})$ is the regular group
$R(F_{20})$, and $X_{14}$ is a graphical regular representation of
$F_{20}$.  In particular, $X_{14}$ is a Cayley graph of $F_{20}$ and of
no other group.
Since the graph $X_{14}$ is dense, for visual simplicity Figure~\ref{fig:x14-complement} shows its complement, the $5$-regular Cayley graph of $F_{20}$ on the connection set $\{a^{2}b,\,a^{4}b^{3},\,a^{2}b^{2},\,a^{3}b^{2},\,a^{4}b^{2}\}$, which has the same automorphism group.

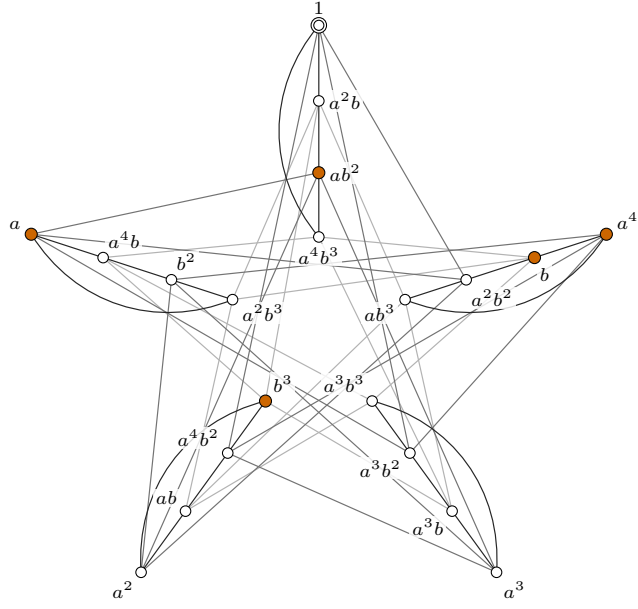
\begin{figure}[t]
\centering
\begin{tikzpicture}[scale=1.0,
  spoke/.style={draw=black!85, line width=0.45pt},
  petal/.style={draw=black!85, line width=0.45pt},
  chordA/.style={draw=black!55, line width=0.45pt},
  chordB/.style={draw=black!30, line width=0.45pt},
  vx/.style={circle, draw=black, fill=white, inner sep=1.4pt},
  pk/.style={circle, draw=black, fill=myorange, inner sep=1.6pt},
  idv/.style={circle, draw=black, double, fill=white, inner sep=1.7pt},
  lb/.style={label distance=0.5pt, inner sep=0.9pt, fill=white, fill opacity=0.82, text opacity=1, font=\scriptsize}]
  \coordinate (v00) at (90.0:4.0);
  \coordinate (v01) at (18.0:3.0);
  \coordinate (v02) at (162.0:2.05);
  \coordinate (v03) at (234.0:1.2);
  \coordinate (v10) at (162.0:4.0);
  \coordinate (v11) at (234.0:3.0);
  \coordinate (v12) at (90.0:2.05);
  \coordinate (v13) at (18.0:1.2);
  \coordinate (v20) at (234.0:4.0);
  \coordinate (v21) at (90.0:3.0);
  \coordinate (v22) at (18.0:2.05);
  \coordinate (v23) at (162.0:1.2);
  \coordinate (v30) at (306.0:4.0);
  \coordinate (v31) at (306.0:3.0);
  \coordinate (v32) at (306.0:2.05);
  \coordinate (v33) at (306.0:1.2);
  \coordinate (v40) at (18.0:4.0);
  \coordinate (v41) at (162.0:3.0);
  \coordinate (v42) at (234.0:2.05);
  \coordinate (v43) at (90.0:1.2);
  \draw[chordB] (v01) -- (v23);
  \draw[chordB] (v01) -- (v33);
  \draw[chordB] (v01) -- (v43);
  \draw[chordB] (v03) -- (v21);
  \draw[chordB] (v03) -- (v31);
  \draw[chordB] (v03) -- (v41);
  \draw[chordB] (v11) -- (v13);
  \draw[chordB] (v11) -- (v23);
  \draw[chordB] (v11) -- (v33);
  \draw[chordB] (v13) -- (v21);
  \draw[chordB] (v13) -- (v31);
  \draw[chordB] (v21) -- (v23);
  \draw[chordB] (v31) -- (v43);
  \draw[chordB] (v33) -- (v41);
  \draw[chordB] (v41) -- (v43);
  \draw[chordA] (v00) -- (v22);
  \draw[chordA] (v00) -- (v32);
  \draw[chordA] (v00) -- (v42);
  \draw[chordA] (v02) -- (v20);
  \draw[chordA] (v02) -- (v30);
  \draw[chordA] (v02) -- (v40);
  \draw[chordA] (v10) -- (v12);
  \draw[chordA] (v10) -- (v22);
  \draw[chordA] (v10) -- (v32);
  \draw[chordA] (v12) -- (v20);
  \draw[chordA] (v12) -- (v30);
  \draw[chordA] (v20) -- (v22);
  \draw[chordA] (v30) -- (v42);
  \draw[chordA] (v32) -- (v40);
  \draw[chordA] (v40) -- (v42);
  \draw[spoke] (v00) -- (v21);
  \draw[spoke] (v01) -- (v22);
  \draw[spoke] (v01) -- (v40);
  \draw[spoke] (v02) -- (v23);
  \draw[spoke] (v02) -- (v41);
  \draw[spoke] (v03) -- (v42);
  \draw[spoke] (v10) -- (v41);
  \draw[spoke] (v11) -- (v42);
  \draw[spoke] (v11) -- (v20);
  \draw[spoke] (v12) -- (v43);
  \draw[spoke] (v12) -- (v21);
  \draw[spoke] (v13) -- (v22);
  \draw[spoke] (v30) -- (v31);
  \draw[spoke] (v31) -- (v32);
  \draw[spoke] (v32) -- (v33);
  \draw[petal] (v00) to[bend right=40] (v43);
  \draw[petal] (v03) to[bend right=40] (v20);
  \draw[petal] (v10) to[bend right=40] (v23);
  \draw[petal] (v13) to[bend right=40] (v40);
  \draw[petal] (v30) to[bend right=40] (v33);
  \node[idv, label={[lb]90:$1$}] at (90.0:4.0) {};
  \node[pk, label={[lb]288:$b$}] at (18.0:3.0) {};
  \node[vx, label={[lb]72:$b^{2}$}] at (162.0:2.05) {};
  \node[pk, label={[lb]54:$b^{3}$}] at (234.0:1.2) {};
  \node[pk, label={[lb]162:$a$}] at (162.0:4.0) {};
  \node[vx, label={[lb]144:$ab$}] at (234.0:3.0) {};
  \node[pk, label={[lb]0:$ab^{2}$}] at (90.0:2.05) {};
  \node[vx, label={[lb]198:$ab^{3}$}] at (18.0:1.2) {};
  \node[vx, label={[lb]234:$a^{2}$}] at (234.0:4.0) {};
  \node[vx, label={[lb]0:$a^{2}b$}] at (90.0:3.0) {};
  \node[vx, label={[lb]288:$a^{2}b^{2}$}] at (18.0:2.05) {};
  \node[vx, label={[lb]342:$a^{2}b^{3}$}] at (162.0:1.2) {};
  \node[vx, label={[lb]306:$a^{3}$}] at (306.0:4.0) {};
  \node[vx, label={[lb]216:$a^{3}b$}] at (306.0:3.0) {};
  \node[vx, label={[lb]216:$a^{3}b^{2}$}] at (306.0:2.05) {};
  \node[vx, label={[lb]126:$a^{3}b^{3}$}] at (306.0:1.2) {};
  \node[pk, label={[lb]18:$a^{4}$}] at (18.0:4.0) {};
  \node[vx, label={[lb]72:$a^{4}b$}] at (162.0:3.0) {};
  \node[vx, label={[lb]144:$a^{4}b^{2}$}] at (234.0:2.05) {};
  \node[vx, label={[lb]270:$a^{4}b^{3}$}] at (90.0:1.2) {};
\end{tikzpicture}
\caption{The complement of $X_{14}$. 
The identity (vertex marked double circle) admits peak state transfer in $X_{14}$ at time $\pi/3$, with amount $16/225$, with each of the five filled vertices $a$, $a^{4}$, $b$, $b^{3}$, $ab^{2}$. }
\label{fig:x14-complement}
\end{figure}

The graph $X_{14}$ admits peak state transfer from the identity to each of
$a$, $a^{-1}$, $b$, $b^{-1}$ and $ab^{2}$.
Each of these five pairs $(u,v)$ has eigenvalue support $\{14,2,-1\}$, with idempotent
entries
\[
    (E_{14})_{v,u} = \frac{1}{20},\qquad
    (E_{2})_{v,u} = \frac{1}{12},\qquad
    (E_{-1})_{v,u} = -\frac{2}{15}.
\]
At $\tau = \pi/3$ we have
$e^{14i\tau} = e^{2i\tau} = -e^{-i\tau} = e^{2\pi i/3}$, so the three
terms of $U(\tau)_{v,u}$ align, and
\[
    |U(\tau)_{v,u}| = \frac{1}{20}+\frac{1}{12}+\frac{2}{15}
        = \frac{4}{15}.
\]
Since $|U(t)_{v,u}|$ never exceeds the sum of the absolute values of the
idempotent entries, this is peak state transfer at time $\pi/3$ with
amount $16/225$. 

\subsection{Products with periodic graphs}

Recall that a graph $Y$ is \textsl{periodic} at a vertex $y$ at time
$\tau$ if $|U_Y(\tau)_{y,y}| = 1$.  Peak state transfer survives Cartesian
products with periodic factors, and the amount is unchanged. The analogue of following lemma for perfect state transfer, which is the case $\alpha = 1$, is due to Coutinho and Godsil~\cite[Lemma~3.2]{Coutinho2018}.

\begin{lem}\label{lem:peak-product}
    Suppose $X$ admits peak state transfer from $u$ to $v$ at time $\tau$
    with amount $\alpha$, and $Y$ is periodic at the vertex $y$ at time
    $\tau$.  Then $X\square Y$ admits peak state transfer from $(u,y)$ to
    $(v,y)$ at time $\tau$ with amount $\alpha$.
\end{lem}

\proof The adjacency matrix of $X\square Y$ is
$A_X\otimes I + I\otimes A_Y$, and the two summands commute, so the
exponential factors:
\[
    U_{X\square Y}(t) = U_X(t)\otimes U_Y(t).
\]
Entries of a Kronecker product factor coordinate-wise, so
\[
    |U_{X\square Y}(\tau)_{(v,y),(u,y)}|
        = |U_X(\tau)_{v,u}|\,|U_Y(\tau)_{y,y}| = \sqrt{\alpha},
\]
using the periodicity of $Y$ at $y$ for the second factor.  It remains to check
that $\sqrt{\alpha}$ is the spectral bound for the pair $(u,y)$, $(v,y)$ in the
product.
In $X\square Y$ distinct sums of eigenvalues may coincide, so this does not follow immediately from the corresponding statement in $X$.
The eigenvalues of $X\square Y$ are the sums $\theta+\mu$ of an eigenvalue $\theta$
of $X$ and an eigenvalue $\mu$ of $Y$, and the spectral idempotents of the
product are
\[
    E^{X\square Y}_{\lambda}
        = \sum_{\theta+\mu=\lambda} E^X_\theta\otimes E^Y_\mu .
\]
Write $f_\mu = (E^Y_\mu)_{y,y}$.  Since $E^Y_\mu$ is a projection,
$f_\mu = \|E^Y_\mu e_y\|^2 \ge 0$, and $\sum_\mu f_\mu = 1$ because the
idempotents sum to the identity.  Hence
\[
    \sum_\lambda \bigl|(E^{X\square Y}_\lambda)_{(v,y),(u,y)}\bigr|
    \;\le\; \sum_\theta\sum_\mu \bigl|(E^X_\theta)_{v,u}\bigr| f_\mu
    \;=\; \sum_\theta \bigl|(E^X_\theta)_{v,u}\bigr| \;=\; \sqrt{\alpha},
\]
the last equality because $X$ has peak state transfer from $u$ to $v$ with
amount $\alpha$.  So the spectral bound for the product is at most
$\sqrt{\alpha}$, and it is attained at $t=\tau$; therefore it equals
$\sqrt{\alpha}$ and $X\square Y$ has peak state transfer from $(u,y)$ to
$(v,y)$ at time $\tau$ with amount $\alpha$. \qed

The complete graph $K_6$ has eigenvalues $5$ and $-1$, and
$e^{5\pi i/3} = e^{-\pi i/3}$, so $K_6$ is periodic at time $\pi/3$.
Periodicity at a fixed time is also preserved by Cartesian products,
since the transition matrix of a product is the tensor product of the
transition matrices of the factors.  Consequently the Hamming graph
$H(d,6) = K_6^{\,\square d}$ is periodic at $\pi/3$ at every vertex, and
Lemma~\ref{lem:peak-product} yields an infinite family,
\[
    \Gamma_d = X_{14}\,\square\, H(d,6), \qquad d\ge 0,
\]
of graphs with peak state transfer at time $\pi/3$ with amount $16/225$,
independent of $d$. Here $H(0,6)$ is the one-vertex graph, so $\Gamma_0 = X_{14}$ itself, and the case $d = 0$ of the transfer claim is the computation of the previous subsection.

\subsection{The Cayley groups of the family}

Each $\Gamma_d$ is a Cayley graph, being a Cartesian product of Cayley  graphs.
We now determine all regular subgroups of its automorphism group.

\begin{thm}\label{thm:x14-family}
    Let $d\ge 0$, let $A = \Aut(X_{14})\cong F_{20}$, and let
    $W_d = S_6\wr S_d$ act on $V(H(d,6)) = \{1,\ldots,6\}^d$ in product
    action, so that $W_0$ is the  group of order $1$ acting on the one-point
    set and $\Gamma_0 = X_{14}$.  Then
    \[
        \Aut(\Gamma_d) = A\times W_d ,
    \]
    and the regular subgroups of $\Aut(\Gamma_d)$ are precisely the groups
    \[
        R_\phi = \{(a,c)\in A\times N_{W_d}(H) : cH = \phi(a)\},
    \]
    where $H$ is a regular subgroup of $W_d$ and
    $\phi: A\to N_{W_d}(H)/H$ is a homomorphism.  Moreover
    $R_\phi\cong H\rtimes_\psi A$, where $\psi: A\to\Aut(H)$ is the composite
    of $\phi$ with the embedding of $N_{W_d}(H)/H$ into $\Aut(H)$ given by
    conjugation. Every $R_\phi$ has a
    quotient isomorphic to $F_{20}$; in particular, $\Gamma_d$ is not a
    Cayley graph of an abelian group.
\end{thm}

\proof We note that $H(0,6)=K_1$ and $\Gamma_0=X_{14}$.
Complete graphs are prime with respect to the Cartesian product, and $X_{14}$ is prime as well.  Suppose $X_{14} = Y\square Z$ with $1<|V(Y)|\le|V(Z)|$.  The factors of a connected graph are connected, and every prime factor of a connected vertex-transitive graph is vertex-transitive (see \cite{HamImrKla2011}), so $Y$ and $Z$ are connected vertex-transitive graphs on $2$ and $10$, or on $4$ and $5$, vertices; moreover $\Aut(Y)\times\Aut(Z)$, acting coordinate-wise, embeds in $\Aut(X_{14})$, which has order $20$.  If $|V(Y)| = 4$, then $Y\in\{C_4,K_4\}$ and $Z\in\{C_5,K_5\}$, so $|\Aut(Y)|\,|\Aut(Z)|\ge 8\cdot 10 > 20$, a contradiction.  If $Y = K_2$, then $|\Aut(Z)|\ge 10$, since $Z$ is vertex-transitive on ten vertices, and comparing orders gives $\Aut(X_{14}) = \Aut(K_2)\times\Aut(Z)$; this group has a central involution, while $\Aut(X_{14})\cong F_{20}$ is a Frobenius group with trivial centre, a contradiction.  So
\[
    \Gamma_d = X_{14}\,\square\,
        \underbrace{K_6\square\cdots\square K_6}_{d}
\]
is the prime factorization of the connected graph $\Gamma_d$.  By the
structure theorem for automorphisms of Cartesian products of connected
graphs \cite[Theorem~6.10]{HamImrKla2011}, every automorphism of
$\Gamma_d$ permutes the prime factors, carrying each factor to an
isomorphic one.  Since $X_{14}\not\cong K_6$, it follows that
$\Aut(\Gamma_d) = \Aut(X_{14})\times(\Aut(K_6)\wr S_d) = A\times W_d$,
with $W_d$ in product action on $\{1,\ldots,6\}^d$.

Let $R$ be a regular subgroup of $A\times W_d$, so
$|R| = 20\cdot 6^d$, and let $\pi_A$ and $\pi_W$ denote the projections
onto the two factors.  Since $(a,c)$ maps a vertex $(x,y)$ to
$(a(x),c(y))$, the group $\pi_A(R)$ is transitive on $V(X_{14})$; as $A$
is regular, $\pi_A(R) = A$.  Let $H = R\cap(\{1\}\times W_d)$, viewed as
a subgroup of $W_d$.  Then $H$ is the kernel of the restriction of
$\pi_A$ to $R$, so $H$ is normal in $R$ and
$|H| = |R|/|A| = 6^d$.  If $c\in H$ fixes a vertex $y$ of $H(d,6)$, then
$(1,c)$ fixes $(x,y)$ for every $x$, and the semiregularity of $R$ gives
$c=1$; hence $H$ is semiregular on $\{1,\ldots,6\}^d$ and, having order
$6^d$, regular.  Conjugating $\{1\}\times H$ by $(a,c)\in R$ shows that
$cHc^{-1} = H$, so $\pi_W(R)\le N_{W_d}(H)$.  For $a\in A$, the fibre
$\{c : (a,c)\in R\}$ is nonempty because $\pi_A(R)=A$, and any two of its
elements differ by an element of $H$; so it is a coset of $H$ in
$N_{W_d}(H)$, and the map $\phi$ sending $a$ to this coset is a
homomorphism from $A$ to $N_{W_d}(H)/H$ with $R = R_\phi$.

Conversely, let $H$ be a regular subgroup of $W_d$ and let
$\phi: A\to N_{W_d}(H)/H$ be a homomorphism.  Since $H$ is normal in
$N_{W_d}(H)$, the set $R_\phi$ is a subgroup of $A\times W_d$, and its
fibre over each $a\in A$ is a coset of $H$, so
$|R_\phi| = |A||H| = 20\cdot6^d$.  Given vertices $(x,y)$ and
$(x',y')$, the regularity of $A$ provides $a$ with $a(x)=x'$; choosing
$c\in\phi(a)$ gives an element $(a,c)$ of $R_\phi$ taking $(x,y)$ to
$(x',c(y))$, and the regularity of $H$ provides $h$ with $h(c(y)) = y'$,
where $(1,h)\in R_\phi$ because $\phi(1) = H$.  So $R_\phi$ is
transitive.  If $(a,c)\in R_\phi$ fixes $(x,y)$, then $a$ fixes $x$, so
$a=1$ by the regularity of $A$; then $c\in\phi(1)=H$ and $c$ fixes $y$,
so $c=1$ by the regularity of $H$.  As $|R_\phi| = |V(\Gamma_d)|$, the
group $R_\phi$ is regular.

Next we identify $R_\phi$ up to isomorphism.  Fix a base point $\omega_0$
of $\{1,\ldots,6\}^d$ and write $N = N_{W_d}(H)$.  Since $H\le N$ is regular,
$N$ is transitive, so $|N| = |H|\,|N_{\omega_0}|$, and $H\cap N_{\omega_0} = 1$
because a nonidentity element of $H$ fixes no point.  Hence $N = H\rtimes
N_{\omega_0}$, and $N_{\omega_0}$ maps isomorphically onto $N/H$.  Conjugation
embeds $N_{\omega_0}$ into $\Aut(H)$: if $\eta\in N_{\omega_0}$ centralizes
$H$, then $\eta(h\omega_0) = h\eta(\omega_0) = h\omega_0$ for every $h\in
H$, so $\eta = 1$.  Let $\sigma\colon N/H\to N_{\omega_0}$ be the inverse of
the isomorphism above and let $\psi\colon A\to\Aut(H)$ be the composite of
$\phi$ with $\sigma$ and this embedding.  Every element of $R_\phi$ is uniquely
of the form $(a, h\,\sigma(\phi(a)))$ with $a\in A$ and $h\in H$, and since
$\sigma(\phi(a))\,h'\,\sigma(\phi(a))^{-1} = \psi(a)(h')$, the map
$(h,a)\mapsto (a, h\,\sigma(\phi(a)))$ is an isomorphism $H\rtimes_\psi A\to
R_\phi$.

Finally, $\{1\}\times H$ is the kernel of the restriction of $\pi_A$ to
$R_\phi$, so $R_\phi/(\{1\}\times H)\cong A\cong F_{20}$.  Every
quotient of an abelian group is abelian, so $R_\phi$ is not abelian. \qed

For $d=1$ the theorem can be made completely explicit.

\begin{cor}\label{cor:x14-k6}
    The graph $X_{14}\square K_6$ is a Cayley graph of exactly three
    groups:
    \[
        S_3\times F_{20},\qquad
        C_6\times F_{20},\qquad
        C_2\times(C_{15}\rtimes C_4).
    \]
\end{cor}

\proof Here $W_1 = S_6$.  
A regular subgroup $H$ of $S_6$ is the regular representation of a group 
  of order six, and two regular subgroups of $S_6$ are conjugate if and only if they are 
  isomorphic; so up to conjugacy $H$ is the regular representation of $C_6$ or of $S_3$.
In both cases the normalizer of $H$ in $S_6$ is the holomorph of $H$, so $N_{S_6}(H)/H\cong\Aut(H)$ \cite{dixon1996}.  Conjugating the pair $(H,\phi)$ by an element of $W_1$ replaces $R_\phi$ by a conjugate subgroup, so for the isomorphism types of the $R_\phi$ it suffices to run over the two subgroups $H$ above and over homomorphisms $\phi\colon F_{20}\to\Aut(H)$ up to conjugation by $\Aut(H)$. The group $F_{20}$ has a unique subgroup of index two, so there are exactly two homomorphisms from $F_{20}$ to $\Aut(C_6)\cong C_2$.  Since every proper quotient of $F_{20}$ is cyclic, every homomorphism from $F_{20}$ to $\Aut(S_3)\cong S_3$ has image of order at most two; there are four of them, the trivial one and three with image of order two, and the latter three are conjugate under $\Aut(S_3)$, leaving two classes.
By Theorem~\ref{thm:x14-family} the corresponding regular subgroup is
$H\rtimes_\psi F_{20}$ with $\psi = \phi$, and we identify the four resulting
groups in turn.  The two trivial homomorphisms give the direct products
$C_6\times F_{20}$ and $S_3\times F_{20}$.  For the nontrivial
$\psi\colon F_{20}\to\Aut(C_6)\cong C_2$, the generator of the image inverts
$C_6 = C_2\times C_3$, hence acts trivially on the factor $C_2$, so
$C_6\rtimes_\psi F_{20}\cong C_2\times(C_3\rtimes_\psi F_{20})$; in the second
factor the kernel $\langle a,b^2\rangle$ of $\psi$ centralizes $C_3$ and $b$
inverts it, while $bab^{-1} = a^2$, so $C_3\rtimes_\psi F_{20}\cong
C_{15}\rtimes C_4$ and the group is $C_2\times(C_{15}\rtimes C_4)$.  For a
nontrivial $\psi\colon F_{20}\to\Aut(S_3)$ the image lies in
$\Aut(S_3) = \Inn(S_3)$, so $\psi$ lifts to a homomorphism
$\hat\psi\colon F_{20}\to S_3$ with $\psi(g)$ conjugation by $\hat\psi(g)$;
as $Z(S_3) = 1$, the map $g\mapsto(\hat\psi(g)^{-1},g)$ is an injective
homomorphism $F_{20}\to S_3\rtimes_\psi F_{20}$ whose image centralizes
$S_3\times\{1\}$ and meets it trivially, so
$S_3\rtimes_\psi F_{20}\cong S_3\times F_{20}$.  The three isomorphism types
listed are therefore the Cayley groups of $X_{14}\square K_6$, by
Sabidussi's theorem. \qed

\begin{rem}
    The graphs $\Gamma_d$ admit more and more Cayley presentations as $d$
    grows, but by Theorem~\ref{thm:x14-family} every one of their Cayley
    groups surjects onto $F_{20}$.  The non-abelianness of this family is
    therefore a property of the graphs and not an artifact of a chosen
    presentation.
\end{rem}

We now summarize this section in a theorem.

\begin{thm}\label{thm:peak-family}
    For every $d\ge 0$ the graph $\Gamma_d = X_{14}\,\square\,H(d,6)$, on $20\cdot 6^d$ vertices, admits peak state transfer at time $\pi/3$ with amount $16/225$, independent of $d$.  Every group over which $\Gamma_d$ is a Cayley graph surjects onto $F_{20}$; in particular no $\Gamma_d$ is a Cayley graph of an abelian group. 
\end{thm}

  \proof Lemma~\ref{lem:peak-product}, applied to $X_{14}$ and the periodic
  graph $H(d,6)$, gives the transfer (for $d=0$ note $H(0,6)=K_1$ and
  $\Gamma_0 = X_{14}$); and every Cayley group of $\Gamma_d$ surjects onto
  $F_{20}$ by Theorem~\ref{thm:x14-family}.
  The result follows.\qed

We note that the amount $16/225$ is modest, and we regard this family as structural evidence rather than as a strong transfer example; finding an infinite family of non-abelian examples with amount bounded away from zero by a meaningful constant remains open.

\section{Conclusions and open problems}

We end with a comment on method and future directions.  Finding
perfect state transfer on Cayley graphs of genuinely non-abelian groups
is an interesting problem: Section~\ref{sec:counterexamples} shows that
exclusively non-abelian PST graphs exist, on $24$ vertices over $S_4$ and
$\SL(2,3)$ and on $32$ vertices as graphical regular
representations, and Theorem~\ref{thm:pansin-sl} shows that they
exist in infinite families.  Very little else is known about them, and
the record of the literature suggests why: they are hard to find.  A
common approach in recent years
\cite{CaoFen2021,AreShaGho2022,WanCao2024,WanFen2023-bicayley,KalBha2024}
is to pick a group that is nearly abelian, dihedral, dicyclic,
generalized dihedral, $V_{8n}$, assume the connection set is closed under
conjugation so that the eigenvalues can be read off the character table,
and translate the spectral PST criteria into conditions on
character sums.

Lemmas~\ref{lem:abelian-index-two}--\ref{lem:gendihedral} show that this
approach can never produce a genuinely new PST graph: a
conjugacy-closed connection set on a group with an abelian subgroup of
index two forces the graph to be a Cayley graph of an abelian group, and
the eigenvalue conditions for perfect state transfer on abelian Cayley
graphs are already known \cite{TanFenCao2019}.  Accordingly, every
explicit example in \cite{CaoFen2021,AreShaGho2022,WanCao2024,WanFen2023-bicayley,KalBha2024} is an abelian Cayley graph in
a non-abelian presentation (see Example~\ref{ex:wangfeng},
Remark~\ref{rem:some-families}), or even simply $K_{2n}$ minus a perfect
matching.  We acknowledge that there can be value in constructing known
PST graphs in different ways, but it should always be made explicit that
the graphs obtained are new presentations of known examples rather than
new examples, and this has not been done in the cited literature. Both reductions, moreover, were available in the isomorphic-Cayley-graph literature before most of the constructions appeared \cite{morris2021two,MorrisSkelton}; the two literatures do not appear to have been aware of each other.

Escaping the reach of these lemmas takes more than avoiding index two.
Extraspecial $2$-groups have no abelian subgroup of index two, and the
conjugacy-class framework of \cite{sin2020} is the natural setting for
eigenvalue computations over them; yet Lemma~\ref{lem:extraspecial}
shows that every Cayley graph arising in this setting is cubelike.  The
graphs of \cite{sin2020} and \cite{PanSin2026} share a structure: the
state transfer runs from $g$ to $zg$ for the central involution $z$, the
connection set satisfies $zS = S$, so the pairs $\{g, zg\}$ are twins and
the graph is a blow-up of a Cayley graph of $G/\langle z\rangle$
(Lemma~\ref{lem:twin-reduction}).  Whether the construction gives
something new is decided by this quotient.  For an extraspecial group the
quotient is elementary abelian, and nothing new can appear.  For the
linear groups of \cite{PanSin2026} the quotient contains
$\mathrm{PSL}(2,q)$, and Section~\ref{sec:pansin} shows that their graphs
are indeed new.  In particular, conjugacy-closed 
connection sets are not the obstacle: over groups far from abelian they
remain a promising source of new examples.

It is always good practice, when constructing a family of Cayley graphs
  with a particular property, to compute the automorphism groups and the
  regular subgroups of its small members.  This can be feasible even when the automorphism group looks hopelessly large at first glance: for the graphs of Proposition~\ref{prop:pansin-glgu} it has order up to $2^{368}\cdot3^3\cdot5^2$, yet the twin quotient reduces each computation to a group of order at most $172800$.  
 The peak state transfer family of Section~\ref{sec:peak-ref} is an example of
  this practice carried out in full: Theorem~\ref{thm:x14-family}
  determines every regular subgroup of the automorphism group of every
  member.
Each graph therefore comes with the
complete list of its Cayley groups, and no member of the family can
later be rediscovered as a new example in a different presentation.

This is not to say that perfect state transfer in Cayley graphs for abelian groups is settled.  A characterization of the connection sets giving PST builds examples, whereas a characterization in terms of the eigenvalues does not; this is what Ba\v{s}i\'{c} did for circulants \cite{basic2013char} and \'{A}rnad\'{o}ttir and Godsil for abelian groups with cyclic Sylow $2$-subgroups \cite{arnadottir2022pst}, and in general it is open (Problem~\ref{prob:connection-set}).  

The most immediate open question is the one we left in Section~\ref{sec:pansin}.  Theorem~\ref{thm:pansin-sl} settles the  $\SL(2,q)$ family of Pantangi and Sin, but its proof passes to the quotient   $\mathrm{PSL}(2,q)$ and relies on that group being simple; for the  $\GL(2,q)$ and $\GU(2,q)$ families of \cite{PanSin2026} the corresponding  quotients are not simple, and we know only by computation that the four  smallest members beyond the degenerate case $q=3$ have no abelian Cayley  presentation (Proposition~\ref{prop:pansin-glgu}).  Problem~\ref{prob:pansin-stability}  asks whether this persists for all $q\ge7$, that is, whether the  $q=3$ case is the only member of either family that is a Cayley graph  of an abelian group. 
 
 We close with three further  problems suggested by our results.

\begin{prob}\label{prob:connection-set} Characterize, in terms of the connection set rather than the spectrum, the Cayley graphs over abelian groups that admit perfect state transfer.  This is known for circulants \cite{basic2013char} and for abelian groups with cyclic Sylow $2$-subgroups \cite{arnadottir2022pst}.  The same question for conjugacy-closed Cayley graphs over non-abelian groups with no abelian subgroup of index two is open and is not resolved  by Lemma~\ref{lem:nonabel-normal-2}.
\end{prob}

  Graphical regular representations are the extreme case of our theme: such
  a graph is a Cayley graph of one group in one way, so perfect state
  transfer on it cannot be explained away by a hidden abelian presentation.
  Our partial computation found  $38$ such examples of order $32$ with transfer time $\pi/2$. Theorem~\ref{thm:pansin-sl} does not settle the question: its graphs
  have a twin at every vertex, so their automorphism groups have order at
  least $2^{|V|/2}$.
  
  \begin{prob}
      Do there exist infinitely many perfect state transfer graphical
      regular representations of non-abelian groups?  In particular, do
      the $32$-vertex examples of Section~\ref{sec:grrs} lie in an
      infinite family?
  \end{prob}

Every graphical regular representation of an abelian group is
  cubelike: if the group has an element of order greater than two, then
  inversion is a non-identity automorphism fixing the identity in every
  Cayley graph of the group, so an abelian group with a graphical regular
  representation is an elementary abelian $2$-group.  Moreover,
  $\Z_2^d$ admits a graphical regular representation only for $d=1$ and
  $d\ge5$. 
  See \cite{Imr1970, ImrWat1976}. 
  The graph $K_2$, row $X_{2a}$ of the census, is trivially such an
  example with perfect state transfer, and the smallest examples beyond
  it have $32$ vertices. For example, let
  \[
  \begin{aligned}
  S=\{&
  00011,\ 00111,\ 01000,\ 01001,\ 01010,\ 01100,\\
  &01101,\ 01110,\ 01111,\ 10001,\ 10111,\ 11111
  \}\subseteq\{0,1\}^{5},
  \end{aligned}
  \]
  where we write $\Z_2^5$ as $\{0,1\}^{5}$ in the usual way. Then $X=\Cay(\Z_2^5,S)$ 
  is a connected $12$-regular graphical regular representation of
  $\Z_2^5$ admitting perfect state transfer at time $\pi/2$, and, since the scan is complete for valency at most $11$ and produced no cubelike graphical regular representation there, $12$ is the smallest valency of one on $32$ vertices.\footnote{In graph6 format
  \Verb"_DGUHGuiPQGBDavGoAGCxsaDlKpr@mMKR_aHrkIGpeBdEeCiRWaRPKPGdZ@?USKNs_O|GbUw@O_xDM]jKWP?".}

  \begin{prob}\label{prob:cubelike-grr}
  Are there infinitely many cubelike graphical regular representations
  admitting perfect state transfer?  What is the smallest valency of a cubelike graphical regular representation with perfect state transfer on $2^d$ vertices, as a function of $d$? 
  \end{prob}

\subsection*{AI usage declaration}

We used Claude (Anthropic) Fable 5 and Opus models and ChatGPT 5.6 Sol (OpenAI).  They helped develop, test and write all of our SageMath
  and GAP code and produced the table in the appendix, and it let us
  preview results: many statements in this paper were verified
  computationally, in small cases or in full, before we found their
  proofs.   It also drafted portions of the text and the figures, under our supervision.

\subsection*{Acknowledgements}
The authors would like to thank Chris Godsil for helpful discussions and comments.

\appendix

\section{Census of vertex-transitive graphs with perfect state transfer}
\label{app:census}

The connected vertex-transitive graphs on $10$ to $30$ vertices are taken from the census of \cite{RoyleHolt2020VTGraphs}; those on at most $9$ vertices we enumerated directly.
The table lists every connected vertex-transitive graph on at most $30$ vertices with PST. There are no
connected vertex-transitive graphs with PST on an odd number of vertices or on $6$, $10$, $14$, $18$, $22$ or $26$ vertices.  Perfect state transfer at time $\tau$ forces periodicity at
$2\tau$, and a connected regular periodic graph has integral spectrum. In the table, we give the spectra of the graphs; they are written as $\lambda^{(m)}$ with
multiplicities $m$.  All of the PST in the table occurs at time $\pi/2$.
  The last column
lists every group $G$, up to isomorphism and in GAP structure-description
notation, such that the graph is a Cayley graph of $G$. We list the abelian groups after ``ab'' and non-abelian after ``non-ab''. In this
notation, $G{:}H$ denotes a split extension of $G$ by $H$, that is, a
semidirect product $G\rtimes H$ with normal subgroup $G$; since the structure
description does not always determine the group up to isomorphism, each group
is followed by its GAP \textsf{SmallGroup} ID $(o,i)$, meaning the $i$-th
group of order $o$ in the \textsf{SmallGroups} library.
 All $91$ vertex-transitive graphs on up to $30$ vertices admitting PST  are published
alongside this paper as \texttt{pst-census-2-30.csv} (label, graph6 string). The $344$ graphs of the partial census at $32$ vertices
(valencies $k\le 12$, complete only up to $k=11$, labels $X_{32a}$ onward), whose group data is discussed
in Section~\ref{sec:counterexamples}, are also published alongside this paper as \texttt{pst-census-32.csv}. Due to lack of time and computing resources, we could not finish the full computation of the vertex-transitive graphs on $32$ vertices.

\subsection*{Computational methods}

All computations were carried out in SageMath\cite{sage} and GAP\cite{gap}.  There are two stages: determining which of the vertex-transitive graphs which admit PST, and finding their Cayley groups.

\subsubsection*{Detecting perfect state transfer}  Every connected vertex-transitive graph on at most $30$ vertices in the census \cite{RoyleHolt2020VTGraphs} was tested for perfect state transfer, as follows.  By vertex-transitivity, it suffices to test the pairs $(u,v)$ with $u$ a fixed vertex.  A pair admits perfect state transfer if and only if it is strongly cospectral, meaning that $E_\theta e_u = \pm E_\theta e_v$ for every spectral idempotent $E_\theta$, and in addition admits peak state transfer; the latter is decided by the spectral characterization of peak state transfer \cite[Theorem~5.3]{CouGuoSch2025}, which requires the eigenvalues in the eigenvalue support of the pair to be integers, or quadratic integers in a common field $\Q(\sqrt{D})$, satisfying explicit divisibility and parity conditions, and which determines the minimal time $\tau_0$.  The scan handles the eigenvalues exactly, as roots of the irreducible factors of the characteristic polynomial over $\Q$, but computes the idempotent entries in floating point.  Every graph with PST detected was then certified afterwards by an exact computation: each of the $435$ graphs (the $91$ graphs of the table below, published as \texttt{pst-census-2-30.csv}, together with the $344$ graphs on $32$ vertices in \texttt{pst-census-32.csv}) has integral spectrum, so $U(\tau_0)$ is a matrix over a cyclotomic field, and we verified symbolically that $U(\tau_0) = \gamma T$ with $\gamma$ unimodular and $T$ a fixed-point-free symmetric permutation matrix over $\Q(i)$ with $\tau_0 = \pi/2$ for every graph except $X_{32kg}$, where $\tau_0 = \pi/4$ and the computation is over $\Q(\zeta_8)$.

\subsubsection*{Computing the Cayley groups}  For each graph $X$ we compute $\Aut(X)$ in SageMath.  By Sabidussi's theorem, $X$ is a Cayley graph of a group $G$ of order $n$ if and only if $\Aut(X)$ contains a transitive subgroup isomorphic to $G$; such a subgroup has order $n$ and is thus a regular subgroup.  We run over the isomorphism types of groups of order $n$ in the GAP \textsf{SmallGroups} library, search for embeddings with \textsf{IsomorphicSubgroups}, and check the images for transitivity.  Each search is given a time budget. For some graphs whose automorphism groups have order roughly $10^9$, the enumeration of all subgroups did not terminate in the time budget and, instead, we generated all Cayley graphs of all groups of orders $16$ and $24$ and checked for which groups the graph appeared as a Cayley graph.

To give an idea of how we exhaustively generated all regular subgroups, we used some elementary group theory to assist our search. For example,  when $n$ is a power of two, every regular subgroup of $\Aut(X)$ is a $2$-group and hence lies in a Sylow $2$-subgroup $P$; starting from $P$, we repeatedly pass to maximal subgroups and discard the intransitive ones (since every subgroup between a regular subgroup and $P$ is transitive) until reaching order $n$.

\begingroup\footnotesize
\begin{longtable}{@{}l r r >{\raggedright\arraybackslash}p{0.35\textwidth} >{\raggedright\arraybackslash}p{0.40\textwidth}@{}}
\toprule
graph & $n$ & $k$ & spectrum & Cayley groups \\
\midrule
\endfirsthead
\toprule
graph & $n$ & $k$ & spectrum & Cayley groups \\
\midrule
\endhead
\bottomrule
\endlastfoot
  $X_{2a}$ & $2$ & $1$ & $1\,(-1)$ & \emph{ab:} $C_{2}$ {\scriptsize(2,1)} (GRR) \\
  \midrule
  $X_{4a}$ & $4$ & $2$ & $2\,0^{(2)}\,(-2)$ & all 2 groups of order $4$ \\
  \midrule
  $X_{8a}$ & $8$ & $3$ & $3\,1^{(3)}\,(-1)^{(3)}\,(-3)$ & \emph{ab:} $C_{4}{\times}C_{2}$ {\scriptsize(8,2)},
  $C_{2}{\times}C_{2}{\times}C_{2}$ {\scriptsize(8,5)}; \emph{non-ab:} $D_{8}$ {\scriptsize(8,3)} \\
  $X_{8b}$ & $8$ & $4$ & $4\,2\,0^{(3)}\,(-2)^{(3)}$ & \emph{ab:} $C_{4}{\times}C_{2}$ {\scriptsize(8,2)},
  $C_{2}{\times}C_{2}{\times}C_{2}$ {\scriptsize(8,5)}; \emph{non-ab:} $D_{8}$ {\scriptsize(8,3)} \\
  $X_{8c}$ & $8$ & $5$ & $5\,1^{(2)}\,(-1)^{(4)}\,(-3)$ & all 5 groups of order $8$ \\
  $X_{8d}$ & $8$ & $6$ & $6\,0^{(4)}\,(-2)^{(3)}$ & all 5 groups of order $8$ \\
  \midrule
    $X_{12a}$ & $12$ & $9$ & $9\,1^{(3)}\,(-1)^{(6)}\,(-3)^{(2)}$ & all 5 groups of order $12$ \\
$X_{12b}$ & $12$ & $10$ & $10\,0^{(6)}\,(-2)^{(5)}$ & all 5 groups of order $12$ \\
\midrule
$X_{16a}$ & $16$ & $4$ & $4\,2^{(4)}\,0^{(6)}\,(-2)^{(4)}\,(-4)$ & \emph{ab:} $C_{4}{\times}C_{4}$ {\scriptsize(16,2)}, $C_{4}{\times}C_{2}{\times}C_{2}$ {\scriptsize(16,10)}, $C_{2}{\times}C_{2}{\times}C_{2}{\times}C_{2}$ {\scriptsize(16,14)}; \emph{non-ab:} $(C_{4}{\times}C_{2}){:}C_{2}$ {\scriptsize(16,3)}, $C_{4}{:}C_{4}$ {\scriptsize(16,4)}, $C_{8}{:}C_{2}$ {\scriptsize(16,6)}, $QD_{16}$ {\scriptsize(16,8)}, $C_{2}{\times}D_{8}$ {\scriptsize(16,11)} \\
$X_{16b}$ & $16$ & $5$ & $5\,3\,1^{(6)}\,(-1)^{(6)}\,(-3)\,(-5)$ & \emph{ab:} $C_{8}{\times}C_{2}$ {\scriptsize(16,5)}; \emph{non-ab:} $(C_{4}{\times}C_{2}){:}C_{2}$ {\scriptsize(16,3)}, $C_{8}{:}C_{2}$ {\scriptsize(16,6)}, $D_{16}$ {\scriptsize(16,7)}, $QD_{16}$ {\scriptsize(16,8)}, $C_{2}{\times}Q_{8}$ {\scriptsize(16,12)} \\
$X_{16c}$ & $16$ & $5$ & $5\,3^{(2)}\,1^{(4)}\,(-1)^{(6)}\,(-3)^{(3)}$ & \emph{ab:} $C_{4}{\times}C_{4}$ {\scriptsize(16,2)}, $C_{4}{\times}C_{2}{\times}C_{2}$ {\scriptsize(16,10)}; \emph{non-ab:} $(C_{4}{\times}C_{2}){:}C_{2}$ {\scriptsize(16,3)}, $C_{4}{:}C_{4}$ {\scriptsize(16,4)}, $C_{2}{\times}D_{8}$ {\scriptsize(16,11)} \\
$X_{16d}$ & $16$ & $6$ & $6\,2^{(3)}\,0^{(8)}\,(-2)^{(3)}\,(-6)$ & \emph{ab:} $C_{4}{\times}C_{4}$ {\scriptsize(16,2)}, $C_{8}{\times}C_{2}$ {\scriptsize(16,5)}, $C_{4}{\times}C_{2}{\times}C_{2}$ {\scriptsize(16,10)}, $C_{2}{\times}C_{2}{\times}C_{2}{\times}C_{2}$ {\scriptsize(16,14)}; \emph{non-ab:} $(C_{4}{\times}C_{2}){:}C_{2}$ {\scriptsize(16,3)}, $C_{4}{:}C_{4}$ {\scriptsize(16,4)}, $C_{8}{:}C_{2}$ {\scriptsize(16,6)}, $D_{16}$ {\scriptsize(16,7)}, $QD_{16}$ {\scriptsize(16,8)}, $Q_{16}$ {\scriptsize(16,9)}, $C_{2}{\times}D_{8}$ {\scriptsize(16,11)}, $C_{2}{\times}Q_{8}$ {\scriptsize(16,12)}, $(C_{4}{\times}C_{2}){:}C_{2}$ {\scriptsize(16,13)} \\
$X_{16e}$ & $16$ & $6$ & $6\,2^{(4)}\,0^{(6)}\,(-2)^{(3)}\,(-4)^{(2)}$ & \emph{ab:} $C_{4}{\times}C_{4}$ {\scriptsize(16,2)}, $C_{4}{\times}C_{2}{\times}C_{2}$ {\scriptsize(16,10)}; \emph{non-ab:} $(C_{4}{\times}C_{2}){:}C_{2}$ {\scriptsize(16,3)}, $C_{4}{:}C_{4}$ {\scriptsize(16,4)}, $C_{2}{\times}D_{8}$ {\scriptsize(16,11)} \\
$X_{16f}$ & $16$ & $6$ & $6\,4\,2^{(2)}\,0^{(6)}\,(-2)^{(5)}\,(-4)$ & \emph{ab:} $C_{8}{\times}C_{2}$ {\scriptsize(16,5)}; \emph{non-ab:} $(C_{4}{\times}C_{2}){:}C_{2}$ {\scriptsize(16,3)}, $C_{8}{:}C_{2}$ {\scriptsize(16,6)}, $D_{16}$ {\scriptsize(16,7)}, $QD_{16}$ {\scriptsize(16,8)}, $C_{2}{\times}Q_{8}$ {\scriptsize(16,12)} \\
$X_{16g}$ & $16$ & $7$ & $7\,1^{(7)}\,(-1)^{(7)}\,(-7)$ & \emph{ab:} $C_{8}{\times}C_{2}$ {\scriptsize(16,5)}, $C_{4}{\times}C_{2}{\times}C_{2}$ {\scriptsize(16,10)}; \emph{non-ab:} $(C_{4}{\times}C_{2}){:}C_{2}$ {\scriptsize(16,3)}, $C_{8}{:}C_{2}$ {\scriptsize(16,6)}, $D_{16}$ {\scriptsize(16,7)}, $QD_{16}$ {\scriptsize(16,8)}, $C_{2}{\times}D_{8}$ {\scriptsize(16,11)}, $C_{2}{\times}Q_{8}$ {\scriptsize(16,12)}, $(C_{4}{\times}C_{2}){:}C_{2}$ {\scriptsize(16,13)} \\
$X_{16h}$ & $16$ & $7$ & $7\,3\,1^{(6)}\,(-1)^{(5)}\,(-3)^{(2)}\,(-5)$ & \emph{ab:} $C_{4}{\times}C_{4}$ {\scriptsize(16,2)}; \emph{non-ab:} $(C_{4}{\times}C_{2}){:}C_{2}$ {\scriptsize(16,3)}, $C_{4}{:}C_{4}$ {\scriptsize(16,4)}, $C_{2}{\times}Q_{8}$ {\scriptsize(16,12)} \\
$X_{16i}$ & $16$ & $7$ & $7\,5\,1^{(4)}\,(-1)^{(7)}\,(-3)^{(3)}$ & \emph{ab:} $C_{8}{\times}C_{2}$ {\scriptsize(16,5)}, $C_{4}{\times}C_{2}{\times}C_{2}$ {\scriptsize(16,10)}; \emph{non-ab:} $(C_{4}{\times}C_{2}){:}C_{2}$ {\scriptsize(16,3)}, $C_{8}{:}C_{2}$ {\scriptsize(16,6)}, $D_{16}$ {\scriptsize(16,7)}, $QD_{16}$ {\scriptsize(16,8)}, $C_{2}{\times}Q_{8}$ {\scriptsize(16,12)} \\
$X_{16j}$ & $16$ & $7$ & $7\,3^{(2)}\,1^{(4)}\,(-1)^{(5)}\,(-3)^{(4)}$ & \emph{ab:} $C_{4}{\times}C_{4}$ {\scriptsize(16,2)}; \emph{non-ab:} $C_{8}{:}C_{2}$ {\scriptsize(16,6)}, $QD_{16}$ {\scriptsize(16,8)}, $C_{2}{\times}D_{8}$ {\scriptsize(16,11)} \\
$X_{16k}$ & $16$ & $7$ & $7\,3^{(2)}\,1^{(4)}\,(-1)^{(5)}\,(-3)^{(4)}$ & \emph{ab:} $C_{4}{\times}C_{4}$ {\scriptsize(16,2)}, $C_{4}{\times}C_{2}{\times}C_{2}$ {\scriptsize(16,10)}; \emph{non-ab:} $(C_{4}{\times}C_{2}){:}C_{2}$ {\scriptsize(16,3)}, $C_{4}{:}C_{4}$ {\scriptsize(16,4)}, $C_{8}{:}C_{2}$ {\scriptsize(16,6)}, $QD_{16}$ {\scriptsize(16,8)}, $C_{2}{\times}D_{8}$ {\scriptsize(16,11)} \\
$X_{16l}$ & $16$ & $8$ & $8\,2^{(4)}\,0^{(5)}\,(-2)^{(4)}\,(-4)^{(2)}$ & \emph{ab:} $C_{4}{\times}C_{4}$ {\scriptsize(16,2)}; \emph{non-ab:} $C_{8}{:}C_{2}$ {\scriptsize(16,6)}, $QD_{16}$ {\scriptsize(16,8)}, $C_{2}{\times}D_{8}$ {\scriptsize(16,11)} \\
$X_{16m}$ & $16$ & $8$ & $8\,2^{(4)}\,0^{(5)}\,(-2)^{(4)}\,(-4)^{(2)}$ & \emph{ab:} $C_{4}{\times}C_{4}$ {\scriptsize(16,2)}, $C_{4}{\times}C_{2}{\times}C_{2}$ {\scriptsize(16,10)}; \emph{non-ab:} $(C_{4}{\times}C_{2}){:}C_{2}$ {\scriptsize(16,3)}, $C_{4}{:}C_{4}$ {\scriptsize(16,4)}, $C_{8}{:}C_{2}$ {\scriptsize(16,6)}, $QD_{16}$ {\scriptsize(16,8)}, $C_{2}{\times}D_{8}$ {\scriptsize(16,11)} \\
$X_{16n}$ & $16$ & $8$ & $8\,2^{(3)}\,0^{(7)}\,(-2)^{(4)}\,(-6)$ & \emph{ab:} $C_{8}{\times}C_{2}$ {\scriptsize(16,5)}; \emph{non-ab:} $(C_{4}{\times}C_{2}){:}C_{2}$ {\scriptsize(16,3)}, $C_{8}{:}C_{2}$ {\scriptsize(16,6)}, $D_{16}$ {\scriptsize(16,7)}, $QD_{16}$ {\scriptsize(16,8)}, $C_{2}{\times}Q_{8}$ {\scriptsize(16,12)} \\
$X_{16o}$ & $16$ & $8$ & $8\,4\,2^{(2)}\,0^{(5)}\,(-2)^{(6)}\,(-4)$ & \emph{ab:} $C_{4}{\times}C_{4}$ {\scriptsize(16,2)}; \emph{non-ab:} $(C_{4}{\times}C_{2}){:}C_{2}$ {\scriptsize(16,3)}, $C_{4}{:}C_{4}$ {\scriptsize(16,4)}, $C_{2}{\times}Q_{8}$ {\scriptsize(16,12)} \\
$X_{16p}$ & $16$ & $8$ & $8\,6\,0^{(7)}\,(-2)^{(7)}$ & \emph{ab:} $C_{8}{\times}C_{2}$ {\scriptsize(16,5)}, $C_{4}{\times}C_{2}{\times}C_{2}$ {\scriptsize(16,10)}; \emph{non-ab:} $(C_{4}{\times}C_{2}){:}C_{2}$ {\scriptsize(16,3)}, $C_{8}{:}C_{2}$ {\scriptsize(16,6)}, $D_{16}$ {\scriptsize(16,7)}, $QD_{16}$ {\scriptsize(16,8)}, $C_{2}{\times}D_{8}$ {\scriptsize(16,11)}, $(C_{4}{\times}C_{2}){:}C_{2}$ {\scriptsize(16,13)} \\
$X_{16q}$ & $16$ & $9$ & $9\,3\,1^{(5)}\,(-1)^{(6)}\,(-3)^{(2)}\,(-5)$ & \emph{ab:} $C_{8}{\times}C_{2}$ {\scriptsize(16,5)}; \emph{non-ab:} $(C_{4}{\times}C_{2}){:}C_{2}$ {\scriptsize(16,3)}, $C_{8}{:}C_{2}$ {\scriptsize(16,6)}, $D_{16}$ {\scriptsize(16,7)}, $QD_{16}$ {\scriptsize(16,8)}, $C_{2}{\times}Q_{8}$ {\scriptsize(16,12)} \\
$X_{16r}$ & $16$ & $9$ & $9\,1^{(6)}\,(-1)^{(8)}\,(-7)$ & \emph{ab:} $C_{16}$ {\scriptsize(16,1)}, $C_{4}{\times}C_{4}$ {\scriptsize(16,2)}, $C_{8}{\times}C_{2}$ {\scriptsize(16,5)}, $C_{4}{\times}C_{2}{\times}C_{2}$ {\scriptsize(16,10)}, $C_{2}{\times}C_{2}{\times}C_{2}{\times}C_{2}$ {\scriptsize(16,14)}; \emph{non-ab:} $(C_{4}{\times}C_{2}){:}C_{2}$ {\scriptsize(16,3)}, $C_{4}{:}C_{4}$ {\scriptsize(16,4)}, $C_{8}{:}C_{2}$ {\scriptsize(16,6)}, $D_{16}$ {\scriptsize(16,7)}, $QD_{16}$ {\scriptsize(16,8)}, $Q_{16}$ {\scriptsize(16,9)}, $C_{2}{\times}D_{8}$ {\scriptsize(16,11)}, $C_{2}{\times}Q_{8}$ {\scriptsize(16,12)}, $(C_{4}{\times}C_{2}){:}C_{2}$ {\scriptsize(16,13)} \\
$X_{16s}$ & $16$ & $9$ & $9\,3^{(2)}\,1^{(3)}\,(-1)^{(6)}\,(-3)^{(4)}$ & \emph{ab:} $C_{4}{\times}C_{4}$ {\scriptsize(16,2)}, $C_{4}{\times}C_{2}{\times}C_{2}$ {\scriptsize(16,10)}; \emph{non-ab:} $(C_{4}{\times}C_{2}){:}C_{2}$ {\scriptsize(16,3)}, $C_{4}{:}C_{4}$ {\scriptsize(16,4)}, $C_{2}{\times}D_{8}$ {\scriptsize(16,11)} \\
$X_{16t}$ & $16$ & $9$ & $9\,5\,1^{(3)}\,(-1)^{(8)}\,(-3)^{(3)}$ & \emph{ab:} $C_{4}{\times}C_{4}$ {\scriptsize(16,2)}, $C_{8}{\times}C_{2}$ {\scriptsize(16,5)}, $C_{4}{\times}C_{2}{\times}C_{2}$ {\scriptsize(16,10)}, $C_{2}{\times}C_{2}{\times}C_{2}{\times}C_{2}$ {\scriptsize(16,14)}; \emph{non-ab:} $(C_{4}{\times}C_{2}){:}C_{2}$ {\scriptsize(16,3)}, $C_{4}{:}C_{4}$ {\scriptsize(16,4)}, $C_{8}{:}C_{2}$ {\scriptsize(16,6)}, $D_{16}$ {\scriptsize(16,7)}, $QD_{16}$ {\scriptsize(16,8)}, $Q_{16}$ {\scriptsize(16,9)}, $C_{2}{\times}D_{8}$ {\scriptsize(16,11)}, $C_{2}{\times}Q_{8}$ {\scriptsize(16,12)}, $(C_{4}{\times}C_{2}){:}C_{2}$ {\scriptsize(16,13)} \\
$X_{16u}$ & $16$ & $10$ & $10\,2^{(2)}\,0^{(8)}\,(-2)^{(4)}\,(-6)$ & \emph{ab:} $C_{16}$ {\scriptsize(16,1)}, $C_{4}{\times}C_{4}$ {\scriptsize(16,2)}, $C_{8}{\times}C_{2}$ {\scriptsize(16,5)}, $C_{4}{\times}C_{2}{\times}C_{2}$ {\scriptsize(16,10)}, $C_{2}{\times}C_{2}{\times}C_{2}{\times}C_{2}$ {\scriptsize(16,14)}; \emph{non-ab:} $(C_{4}{\times}C_{2}){:}C_{2}$ {\scriptsize(16,3)}, $C_{4}{:}C_{4}$ {\scriptsize(16,4)}, $C_{8}{:}C_{2}$ {\scriptsize(16,6)}, $D_{16}$ {\scriptsize(16,7)}, $QD_{16}$ {\scriptsize(16,8)}, $Q_{16}$ {\scriptsize(16,9)}, $C_{2}{\times}D_{8}$ {\scriptsize(16,11)}, $C_{2}{\times}Q_{8}$ {\scriptsize(16,12)}, $(C_{4}{\times}C_{2}){:}C_{2}$ {\scriptsize(16,13)} \\
$X_{16v}$ & $16$ & $10$ & $10\,2^{(3)}\,0^{(6)}\,(-2)^{(4)}\,(-4)^{(2)}$ & \emph{ab:} $C_{4}{\times}C_{4}$ {\scriptsize(16,2)}, $C_{4}{\times}C_{2}{\times}C_{2}$ {\scriptsize(16,10)}; \emph{non-ab:} $(C_{4}{\times}C_{2}){:}C_{2}$ {\scriptsize(16,3)}, $C_{4}{:}C_{4}$ {\scriptsize(16,4)}, $C_{2}{\times}D_{8}$ {\scriptsize(16,11)} \\
$X_{16w}$ & $16$ & $10$ & $10\,4\,2\,0^{(6)}\,(-2)^{(6)}\,(-4)$ & \emph{ab:} $C_{8}{\times}C_{2}$ {\scriptsize(16,5)}; \emph{non-ab:} $(C_{4}{\times}C_{2}){:}C_{2}$ {\scriptsize(16,3)}, $C_{8}{:}C_{2}$ {\scriptsize(16,6)}, $D_{16}$ {\scriptsize(16,7)}, $QD_{16}$ {\scriptsize(16,8)}, $C_{2}{\times}Q_{8}$ {\scriptsize(16,12)} \\
$X_{16x}$ & $16$ & $11$ & $11\,1^{(6)}\,(-1)^{(6)}\,(-3)^{(2)}\,(-5)$ & \emph{ab:} $C_{4}{\times}C_{4}$ {\scriptsize(16,2)}, $C_{8}{\times}C_{2}$ {\scriptsize(16,5)}; \emph{non-ab:} $(C_{4}{\times}C_{2}){:}C_{2}$ {\scriptsize(16,3)}, $C_{4}{:}C_{4}$ {\scriptsize(16,4)}, $C_{8}{:}C_{2}$ {\scriptsize(16,6)}, $D_{16}$ {\scriptsize(16,7)}, $QD_{16}$ {\scriptsize(16,8)} \\
$X_{16y}$ & $16$ & $11$ & $11\,3\,1^{(4)}\,(-1)^{(6)}\,(-3)^{(4)}$ & \emph{ab:} $C_{4}{\times}C_{4}$ {\scriptsize(16,2)}, $C_{4}{\times}C_{2}{\times}C_{2}$ {\scriptsize(16,10)}; \emph{non-ab:} $(C_{4}{\times}C_{2}){:}C_{2}$ {\scriptsize(16,3)}, $C_{4}{:}C_{4}$ {\scriptsize(16,4)}, $C_{8}{:}C_{2}$ {\scriptsize(16,6)}, $QD_{16}$ {\scriptsize(16,8)}, $C_{2}{\times}D_{8}$ {\scriptsize(16,11)} \\
$X_{16z}$ & $16$ & $12$ & $12\,2^{(2)}\,0^{(6)}\,(-2)^{(6)}\,(-4)$ & \emph{ab:} $C_{4}{\times}C_{4}$ {\scriptsize(16,2)}, $C_{8}{\times}C_{2}$ {\scriptsize(16,5)}; \emph{non-ab:} $(C_{4}{\times}C_{2}){:}C_{2}$ {\scriptsize(16,3)}, $C_{4}{:}C_{4}$ {\scriptsize(16,4)}, $C_{8}{:}C_{2}$ {\scriptsize(16,6)}, $D_{16}$ {\scriptsize(16,7)}, $QD_{16}$ {\scriptsize(16,8)} \\
  $X_{16aa}$ & $16$ & $13$ & $13\,1^{(4)}\,(-1)^{(8)}\,(-3)^{(3)}$ & all 14 groups of order $16$ \\
$X_{16ab}$ & $16$ & $14$ & $14\,0^{(8)}\,(-2)^{(7)}$ & all 14 groups of order $16$ \\
\midrule
 $X_{20a}$ & $20$ & $17$ & $17\,1^{(5)}\,(-1)^{(10)}\,(-3)^{(4)}$ & all 5 groups of order $20$ \\
$X_{20b}$ & $20$ & $18$ & $18\,0^{(10)}\,(-2)^{(9)}$ & all 5 groups of order $20$ \\
\midrule
$X_{24a}$ & $24$ & $9$ & $9\,3^{(2)}\,1^{(9)}\,(-1)^{(9)}\,(-3)^{(2)}\,(-9)$ & \emph{ab:} $C_{12}{\times}C_{2}$ {\scriptsize(24,9)}, $C_{6}{\times}C_{2}{\times}C_{2}$ {\scriptsize(24,15)}; \emph{non-ab:} $C_{4}{\times}S_{3}$ {\scriptsize(24,5)}, $D_{24}$ {\scriptsize(24,6)}, $C_{2}{\times}(C_{3}{:}C_{4})$ {\scriptsize(24,7)}, $(C_{6}{\times}C_{2}){:}C_{2}$ {\scriptsize(24,8)}, $C_{3}{\times}D_{8}$ {\scriptsize(24,10)}, $S_{4}$ {\scriptsize(24,12)}, $C_{2}{\times}A_{4}$ {\scriptsize(24,13)}, $C_{2}{\times}C_{2}{\times}S_{3}$ {\scriptsize(24,14)} \\
$X_{24b}$ & $24$ & $9$ & $9\,3^{(3)}\,1^{(9)}\,(-1)^{(6)}\,(-3)^{(2)}\,(-5)^{(3)}$ & \emph{non-ab:} $S_{4}$ {\scriptsize(24,12)}, $C_{2}{\times}A_{4}$ {\scriptsize(24,13)} \\
$X_{24c}$ & $24$ & $9$ & $9\,7\,1^{(9)}\,(-1)^{(9)}\,(-3)^{(2)}\,(-5)^{(2)}$ & \emph{ab:} $C_{12}{\times}C_{2}$ {\scriptsize(24,9)}, $C_{6}{\times}C_{2}{\times}C_{2}$ {\scriptsize(24,15)}; \emph{non-ab:} $C_{4}{\times}S_{3}$ {\scriptsize(24,5)}, $D_{24}$ {\scriptsize(24,6)}, $C_{2}{\times}(C_{3}{:}C_{4})$ {\scriptsize(24,7)}, $(C_{6}{\times}C_{2}){:}C_{2}$ {\scriptsize(24,8)}, $C_{3}{\times}D_{8}$ {\scriptsize(24,10)}, $S_{4}$ {\scriptsize(24,12)}, $C_{2}{\times}A_{4}$ {\scriptsize(24,13)}, $C_{2}{\times}C_{2}{\times}S_{3}$ {\scriptsize(24,14)} \\
$X_{24d}$ & $24$ & $9$ & $9\,5^{(2)}\,1^{(6)}\,(-1)^{(12)}\,(-3)^{(2)}\,(-7)$ & \emph{ab:} $C_{24}$ {\scriptsize(24,2)}, $C_{12}{\times}C_{2}$ {\scriptsize(24,9)}, $C_{6}{\times}C_{2}{\times}C_{2}$ {\scriptsize(24,15)}; \emph{non-ab:} $C_{3}{:}C_{8}$ {\scriptsize(24,1)}, $C_{3}{:}Q_{8}$ {\scriptsize(24,4)}, $C_{4}{\times}S_{3}$ {\scriptsize(24,5)}, $D_{24}$ {\scriptsize(24,6)}, $C_{2}{\times}(C_{3}{:}C_{4})$ {\scriptsize(24,7)}, $(C_{6}{\times}C_{2}){:}C_{2}$ {\scriptsize(24,8)}, $C_{3}{\times}D_{8}$ {\scriptsize(24,10)}, $C_{3}{\times}Q_{8}$ {\scriptsize(24,11)}, $S_{4}$ {\scriptsize(24,12)}, $C_{2}{\times}A_{4}$ {\scriptsize(24,13)}, $C_{2}{\times}C_{2}{\times}S_{3}$ {\scriptsize(24,14)} \\
$X_{24e}$ & $24$ & $9$ & $9\,5^{(3)}\,1^{(3)}\,(-1)^{(12)}\,(-3)^{(5)}$ & \emph{non-ab:} $SL(2,3)$ {\scriptsize(24,3)}, $S_{4}$ {\scriptsize(24,12)} \\
$X_{24f}$ & $24$ & $10$ & $10\,2^{(5)}\,0^{(12)}\,(-2)^{(5)}\,(-10)$ & \emph{ab:} $C_{12}{\times}C_{2}$ {\scriptsize(24,9)}, $C_{6}{\times}C_{2}{\times}C_{2}$ {\scriptsize(24,15)}; \emph{non-ab:} $C_{3}{:}Q_{8}$ {\scriptsize(24,4)}, $C_{4}{\times}S_{3}$ {\scriptsize(24,5)}, $D_{24}$ {\scriptsize(24,6)}, $C_{2}{\times}(C_{3}{:}C_{4})$ {\scriptsize(24,7)}, $(C_{6}{\times}C_{2}){:}C_{2}$ {\scriptsize(24,8)}, $C_{3}{\times}D_{8}$ {\scriptsize(24,10)}, $C_{3}{\times}Q_{8}$ {\scriptsize(24,11)}, $S_{4}$ {\scriptsize(24,12)}, $C_{2}{\times}A_{4}$ {\scriptsize(24,13)}, $C_{2}{\times}C_{2}{\times}S_{3}$ {\scriptsize(24,14)} \\
$X_{24g}$ & $24$ & $10$ & $10\,2^{(6)}\,0^{(12)}\,(-2)^{(2)}\,(-6)^{(3)}$ & \emph{non-ab:} $SL(2,3)$ {\scriptsize(24,3)}, $S_{4}$ {\scriptsize(24,12)} \\
$X_{24h}$ & $24$ & $10$ & $10\,6\,2^{(3)}\,0^{(12)}\,(-2)^{(5)}\,(-6)^{(2)}$ & \emph{ab:} $C_{12}{\times}C_{2}$ {\scriptsize(24,9)}, $C_{6}{\times}C_{2}{\times}C_{2}$ {\scriptsize(24,15)}; \emph{non-ab:} $C_{3}{:}Q_{8}$ {\scriptsize(24,4)}, $C_{4}{\times}S_{3}$ {\scriptsize(24,5)}, $D_{24}$ {\scriptsize(24,6)}, $C_{2}{\times}(C_{3}{:}C_{4})$ {\scriptsize(24,7)}, $(C_{6}{\times}C_{2}){:}C_{2}$ {\scriptsize(24,8)}, $C_{3}{\times}D_{8}$ {\scriptsize(24,10)}, $C_{3}{\times}Q_{8}$ {\scriptsize(24,11)}, $S_{4}$ {\scriptsize(24,12)}, $C_{2}{\times}A_{4}$ {\scriptsize(24,13)}, $C_{2}{\times}C_{2}{\times}S_{3}$ {\scriptsize(24,14)} \\
$X_{24i}$ & $24$ & $10$ & $10\,4^{(2)}\,2^{(3)}\,0^{(9)}\,(-2)^{(8)}\,(-8)$ & \emph{ab:} $C_{12}{\times}C_{2}$ {\scriptsize(24,9)}; \emph{non-ab:} $C_{4}{\times}S_{3}$ {\scriptsize(24,5)}, $D_{24}$ {\scriptsize(24,6)}, $C_{2}{\times}(C_{3}{:}C_{4})$ {\scriptsize(24,7)}, $(C_{6}{\times}C_{2}){:}C_{2}$ {\scriptsize(24,8)}, $C_{3}{\times}D_{8}$ {\scriptsize(24,10)}, $S_{4}$ {\scriptsize(24,12)}, $C_{2}{\times}A_{4}$ {\scriptsize(24,13)} \\
$X_{24j}$ & $24$ & $10$ & $10\,4^{(3)}\,2^{(3)}\,0^{(6)}\,(-2)^{(8)}\,(-4)^{(3)}$ & \emph{non-ab:} $S_{4}$ {\scriptsize(24,12)}, $C_{2}{\times}A_{4}$ {\scriptsize(24,13)} \\
$X_{24k}$ & $24$ & $10$ & $10\,8\,2^{(3)}\,0^{(9)}\,(-2)^{(8)}\,(-4)^{(2)}$ & \emph{ab:} $C_{12}{\times}C_{2}$ {\scriptsize(24,9)}; \emph{non-ab:} $C_{4}{\times}S_{3}$ {\scriptsize(24,5)}, $D_{24}$ {\scriptsize(24,6)}, $C_{2}{\times}(C_{3}{:}C_{4})$ {\scriptsize(24,7)}, $(C_{6}{\times}C_{2}){:}C_{2}$ {\scriptsize(24,8)}, $C_{3}{\times}D_{8}$ {\scriptsize(24,10)}, $S_{4}$ {\scriptsize(24,12)}, $C_{2}{\times}A_{4}$ {\scriptsize(24,13)} \\
$X_{24l}$ & $24$ & $10$ & $10\,6^{(2)}\,0^{(12)}\,(-2)^{(8)}\,(-6)$ & \emph{ab:} $C_{24}$ {\scriptsize(24,2)}, $C_{12}{\times}C_{2}$ {\scriptsize(24,9)}, $C_{6}{\times}C_{2}{\times}C_{2}$ {\scriptsize(24,15)}; \emph{non-ab:} $C_{3}{:}C_{8}$ {\scriptsize(24,1)}, $C_{3}{:}Q_{8}$ {\scriptsize(24,4)}, $C_{4}{\times}S_{3}$ {\scriptsize(24,5)}, $D_{24}$ {\scriptsize(24,6)}, $C_{2}{\times}(C_{3}{:}C_{4})$ {\scriptsize(24,7)}, $(C_{6}{\times}C_{2}){:}C_{2}$ {\scriptsize(24,8)}, $C_{3}{\times}D_{8}$ {\scriptsize(24,10)}, $C_{3}{\times}Q_{8}$ {\scriptsize(24,11)}, $S_{4}$ {\scriptsize(24,12)}, $C_{2}{\times}A_{4}$ {\scriptsize(24,13)}, $C_{2}{\times}C_{2}{\times}S_{3}$ {\scriptsize(24,14)} \\
$X_{24m}$ & $24$ & $10$ & $10\,4^{(3)}\,2^{(3)}\,0^{(6)}\,(-2)^{(8)}\,(-4)^{(3)}$ & \emph{non-ab:} $S_{4}$ {\scriptsize(24,12)}, $C_{2}{\times}A_{4}$ {\scriptsize(24,13)} \\
  $X_{24n}$ & $24$ & $11$ & $11\,1^{(11)}\,(-1)^{(11)}\,(-11)$ & \emph{ab:} $C_{12}{\times}C_{2}$ {\scriptsize(24,9)},
  $C_{6}{\times}C_{2}{\times}C_{2}$ {\scriptsize(24,15)}; \emph{non-ab:} $C_{4}{\times}S_{3}$ {\scriptsize(24,5)},
  $D_{24}$ {\scriptsize(24,6)}, $C_{2}{\times}(C_{3}{:}C_{4})$ {\scriptsize(24,7)}, $(C_{6}{\times}C_{2}){:}C_{2}$
  {\scriptsize(24,8)}, $C_{3}{\times}D_{8}$ {\scriptsize(24,10)}, $S_{4}$ {\scriptsize(24,12)}, $C_{2}{\times}A_{4}$
  {\scriptsize(24,13)}, $C_{2}{\times}C_{2}{\times}S_{3}$ {\scriptsize(24,14)} \\
$X_{24o}$ & $24$ & $11$ & $11\,5\,1^{(9)}\,(-1)^{(11)}\,(-7)^{(2)}$ & \emph{ab:} $C_{12}{\times}C_{2}$ {\scriptsize(24,9)}, $C_{6}{\times}C_{2}{\times}C_{2}$ {\scriptsize(24,15)}; \emph{non-ab:} $C_{4}{\times}S_{3}$ {\scriptsize(24,5)}, $D_{24}$ {\scriptsize(24,6)}, $C_{2}{\times}(C_{3}{:}C_{4})$ {\scriptsize(24,7)}, $(C_{6}{\times}C_{2}){:}C_{2}$ {\scriptsize(24,8)}, $C_{3}{\times}D_{8}$ {\scriptsize(24,10)}, $S_{4}$ {\scriptsize(24,12)}, $C_{2}{\times}A_{4}$ {\scriptsize(24,13)}, $C_{2}{\times}C_{2}{\times}S_{3}$ {\scriptsize(24,14)} \\
$X_{24p}$ & $24$ & $11$ & $11\,3^{(2)}\,1^{(9)}\,(-1)^{(8)}\,(-3)^{(3)}\,(-9)$ & \emph{ab:} $C_{12}{\times}C_{2}$ {\scriptsize(24,9)}; \emph{non-ab:} $C_{4}{\times}S_{3}$ {\scriptsize(24,5)}, $D_{24}$ {\scriptsize(24,6)}, $C_{2}{\times}(C_{3}{:}C_{4})$ {\scriptsize(24,7)}, $(C_{6}{\times}C_{2}){:}C_{2}$ {\scriptsize(24,8)}, $C_{3}{\times}D_{8}$ {\scriptsize(24,10)}, $S_{4}$ {\scriptsize(24,12)}, $C_{2}{\times}A_{4}$ {\scriptsize(24,13)} \\
$X_{24q}$ & $24$ & $11$ & $11\,7\,1^{(9)}\,(-1)^{(8)}\,(-3)^{(3)}\,(-5)^{(2)}$ & \emph{ab:} $C_{12}{\times}C_{2}$ {\scriptsize(24,9)}; \emph{non-ab:} $C_{4}{\times}S_{3}$ {\scriptsize(24,5)}, $D_{24}$ {\scriptsize(24,6)}, $C_{2}{\times}(C_{3}{:}C_{4})$ {\scriptsize(24,7)}, $(C_{6}{\times}C_{2}){:}C_{2}$ {\scriptsize(24,8)}, $C_{3}{\times}D_{8}$ {\scriptsize(24,10)}, $S_{4}$ {\scriptsize(24,12)}, $C_{2}{\times}A_{4}$ {\scriptsize(24,13)} \\
$X_{24r}$ & $24$ & $11$ & $11\,5^{(2)}\,1^{(6)}\,(-1)^{(11)}\,(-3)^{(3)}\,(-7)$ & \emph{ab:} $C_{12}{\times}C_{2}$ {\scriptsize(24,9)}; \emph{non-ab:} $C_{4}{\times}S_{3}$ {\scriptsize(24,5)}, $D_{24}$ {\scriptsize(24,6)}, $C_{2}{\times}(C_{3}{:}C_{4})$ {\scriptsize(24,7)}, $(C_{6}{\times}C_{2}){:}C_{2}$ {\scriptsize(24,8)}, $C_{3}{\times}D_{8}$ {\scriptsize(24,10)}, $S_{4}$ {\scriptsize(24,12)}, $C_{2}{\times}A_{4}$ {\scriptsize(24,13)} \\
$X_{24s}$ & $24$ & $11$ & $11\,3^{(3)}\,1^{(9)}\,(-1)^{(5)}\,(-3)^{(3)}\,(-5)^{(3)}$ & \emph{non-ab:} $S_{4}$ {\scriptsize(24,12)}, $C_{2}{\times}A_{4}$ {\scriptsize(24,13)} \\
$X_{24t}$ & $24$ & $11$ & $11\,9\,1^{(6)}\,(-1)^{(11)}\,(-3)^{(5)}$ & \emph{ab:} $C_{12}{\times}C_{2}$ {\scriptsize(24,9)}; \emph{non-ab:} $C_{4}{\times}S_{3}$ {\scriptsize(24,5)}, $D_{24}$ {\scriptsize(24,6)}, $C_{2}{\times}(C_{3}{:}C_{4})$ {\scriptsize(24,7)}, $(C_{6}{\times}C_{2}){:}C_{2}$ {\scriptsize(24,8)}, $C_{3}{\times}D_{8}$ {\scriptsize(24,10)}, $S_{4}$ {\scriptsize(24,12)}, $C_{2}{\times}A_{4}$ {\scriptsize(24,13)} \\
$X_{24u}$ & $24$ & $11$ & $11\,5^{(3)}\,1^{(3)}\,(-1)^{(11)}\,(-3)^{(6)}$ & \emph{non-ab:} $S_{4}$ {\scriptsize(24,12)}, $C_{2}{\times}A_{4}$ {\scriptsize(24,13)} \\
$X_{24v}$ & $24$ & $12$ & $12\,4^{(2)}\,2^{(3)}\,0^{(8)}\,(-2)^{(9)}\,(-8)$ & \emph{ab:} $C_{12}{\times}C_{2}$ {\scriptsize(24,9)}; \emph{non-ab:} $C_{4}{\times}S_{3}$ {\scriptsize(24,5)}, $D_{24}$ {\scriptsize(24,6)}, $C_{2}{\times}(C_{3}{:}C_{4})$ {\scriptsize(24,7)}, $(C_{6}{\times}C_{2}){:}C_{2}$ {\scriptsize(24,8)}, $C_{3}{\times}D_{8}$ {\scriptsize(24,10)}, $S_{4}$ {\scriptsize(24,12)}, $C_{2}{\times}A_{4}$ {\scriptsize(24,13)} \\
$X_{24w}$ & $24$ & $12$ & $12\,2^{(6)}\,0^{(11)}\,(-2)^{(3)}\,(-6)^{(3)}$ & \emph{non-ab:} $S_{4}$ {\scriptsize(24,12)}, $C_{2}{\times}A_{4}$ {\scriptsize(24,13)} \\
$X_{24x}$ & $24$ & $12$ & $12\,2^{(5)}\,0^{(11)}\,(-2)^{(6)}\,(-10)$ & \emph{ab:} $C_{12}{\times}C_{2}$ {\scriptsize(24,9)}; \emph{non-ab:} $C_{4}{\times}S_{3}$ {\scriptsize(24,5)}, $D_{24}$ {\scriptsize(24,6)}, $C_{2}{\times}(C_{3}{:}C_{4})$ {\scriptsize(24,7)}, $(C_{6}{\times}C_{2}){:}C_{2}$ {\scriptsize(24,8)}, $C_{3}{\times}D_{8}$ {\scriptsize(24,10)}, $S_{4}$ {\scriptsize(24,12)}, $C_{2}{\times}A_{4}$ {\scriptsize(24,13)} \\
$X_{24y}$ & $24$ & $12$ & $12\,4^{(3)}\,2^{(3)}\,0^{(5)}\,(-2)^{(9)}\,(-4)^{(3)}$ & \emph{non-ab:} $S_{4}$ {\scriptsize(24,12)}, $C_{2}{\times}A_{4}$ {\scriptsize(24,13)} \\
$X_{24z}$ & $24$ & $12$ & $12\,8\,2^{(3)}\,0^{(8)}\,(-2)^{(9)}\,(-4)^{(2)}$ & \emph{ab:} $C_{12}{\times}C_{2}$ {\scriptsize(24,9)}; \emph{non-ab:} $C_{4}{\times}S_{3}$ {\scriptsize(24,5)}, $D_{24}$ {\scriptsize(24,6)}, $C_{2}{\times}(C_{3}{:}C_{4})$ {\scriptsize(24,7)}, $(C_{6}{\times}C_{2}){:}C_{2}$ {\scriptsize(24,8)}, $C_{3}{\times}D_{8}$ {\scriptsize(24,10)}, $S_{4}$ {\scriptsize(24,12)}, $C_{2}{\times}A_{4}$ {\scriptsize(24,13)} \\
$X_{24aa}$ & $24$ & $12$ & $12\,6\,2^{(3)}\,0^{(11)}\,(-2)^{(6)}\,(-6)^{(2)}$ & \emph{ab:} $C_{12}{\times}C_{2}$ {\scriptsize(24,9)}; \emph{non-ab:} $C_{4}{\times}S_{3}$ {\scriptsize(24,5)}, $D_{24}$ {\scriptsize(24,6)}, $C_{2}{\times}(C_{3}{:}C_{4})$ {\scriptsize(24,7)}, $(C_{6}{\times}C_{2}){:}C_{2}$ {\scriptsize(24,8)}, $C_{3}{\times}D_{8}$ {\scriptsize(24,10)}, $S_{4}$ {\scriptsize(24,12)}, $C_{2}{\times}A_{4}$ {\scriptsize(24,13)} \\
$X_{24ab}$ & $24$ & $12$ & $12\,6^{(2)}\,0^{(11)}\,(-2)^{(9)}\,(-6)$ & \emph{ab:} $C_{12}{\times}C_{2}$ {\scriptsize(24,9)}, $C_{6}{\times}C_{2}{\times}C_{2}$ {\scriptsize(24,15)}; \emph{non-ab:} $C_{4}{\times}S_{3}$ {\scriptsize(24,5)}, $D_{24}$ {\scriptsize(24,6)}, $C_{2}{\times}(C_{3}{:}C_{4})$ {\scriptsize(24,7)}, $(C_{6}{\times}C_{2}){:}C_{2}$ {\scriptsize(24,8)}, $C_{3}{\times}D_{8}$ {\scriptsize(24,10)}, $S_{4}$ {\scriptsize(24,12)}, $C_{2}{\times}A_{4}$ {\scriptsize(24,13)}, $C_{2}{\times}C_{2}{\times}S_{3}$ {\scriptsize(24,14)} \\
  $X_{24ac}$ & $24$ & $12$ & $12\,10\,0^{(11)}\,(-2)^{(11)}$ & \emph{ab:} $C_{12}{\times}C_{2}$ {\scriptsize(24,9)},
  $C_{6}{\times}C_{2}{\times}C_{2}$ {\scriptsize(24,15)}; \emph{non-ab:} $C_{4}{\times}S_{3}$ {\scriptsize(24,5)},
  $D_{24}$ {\scriptsize(24,6)}, $C_{2}{\times}(C_{3}{:}C_{4})$ {\scriptsize(24,7)}, $(C_{6}{\times}C_{2}){:}C_{2}$
  {\scriptsize(24,8)}, $C_{3}{\times}D_{8}$ {\scriptsize(24,10)}, $S_{4}$ {\scriptsize(24,12)}, $C_{2}{\times}A_{4}$
  {\scriptsize(24,13)}, $C_{2}{\times}C_{2}{\times}S_{3}$ {\scriptsize(24,14)} \\
$X_{24ad}$ & $24$ & $13$ & $13\,3^{(2)}\,1^{(8)}\,(-1)^{(9)}\,(-3)^{(3)}\,(-9)$ & \emph{ab:} $C_{12}{\times}C_{2}$ {\scriptsize(24,9)}; \emph{non-ab:} $C_{4}{\times}S_{3}$ {\scriptsize(24,5)}, $D_{24}$ {\scriptsize(24,6)}, $C_{2}{\times}(C_{3}{:}C_{4})$ {\scriptsize(24,7)}, $(C_{6}{\times}C_{2}){:}C_{2}$ {\scriptsize(24,8)}, $C_{3}{\times}D_{8}$ {\scriptsize(24,10)}, $S_{4}$ {\scriptsize(24,12)}, $C_{2}{\times}A_{4}$ {\scriptsize(24,13)} \\
$X_{24ae}$ & $24$ & $13$ & $13\,7\,1^{(8)}\,(-1)^{(9)}\,(-3)^{(3)}\,(-5)^{(2)}$ & \emph{ab:} $C_{12}{\times}C_{2}$ {\scriptsize(24,9)}; \emph{non-ab:} $C_{4}{\times}S_{3}$ {\scriptsize(24,5)}, $D_{24}$ {\scriptsize(24,6)}, $C_{2}{\times}(C_{3}{:}C_{4})$ {\scriptsize(24,7)}, $(C_{6}{\times}C_{2}){:}C_{2}$ {\scriptsize(24,8)}, $C_{3}{\times}D_{8}$ {\scriptsize(24,10)}, $S_{4}$ {\scriptsize(24,12)}, $C_{2}{\times}A_{4}$ {\scriptsize(24,13)} \\
$X_{24af}$ & $24$ & $13$ & $13\,3^{(3)}\,1^{(8)}\,(-1)^{(6)}\,(-3)^{(3)}\,(-5)^{(3)}$ & \emph{non-ab:} $S_{4}$ {\scriptsize(24,12)}, $C_{2}{\times}A_{4}$ {\scriptsize(24,13)} \\
    $X_{24ag}$ & $24$ & $13$ & $13\,1^{(10)}\,(-1)^{(12)}\,(-11)$ & \emph{ab:} $C_{24}$ {\scriptsize(24,2)},
    $C_{12}{\times}C_{2}$ {\scriptsize(24,9)}, $C_{6}{\times}C_{2}{\times}C_{2}$ {\scriptsize(24,15)}; \emph{non-ab:}
    $C_{3}{:}C_{8}$ {\scriptsize(24,1)}, $C_{3}{:}Q_{8}$ {\scriptsize(24,4)}, $C_{4}{\times}S_{3}$
  {\scriptsize(24,5)},
    $D_{24}$ {\scriptsize(24,6)}, $C_{2}{\times}(C_{3}{:}C_{4})$ {\scriptsize(24,7)}, $(C_{6}{\times}C_{2}){:}C_{2}$
    {\scriptsize(24,8)}, $C_{3}{\times}D_{8}$ {\scriptsize(24,10)}, $C_{3}{\times}Q_{8}$ {\scriptsize(24,11)}, $S_{4}$
    {\scriptsize(24,12)}, $C_{2}{\times}A_{4}$ {\scriptsize(24,13)}, $C_{2}{\times}C_{2}{\times}S_{3}$
  {\scriptsize(24,14)} \\
$X_{24ah}$ & $24$ & $13$ & $13\,3^{(3)}\,1^{(8)}\,(-1)^{(6)}\,(-3)^{(3)}\,(-5)^{(3)}$ & \emph{non-ab:} $S_{4}$ {\scriptsize(24,12)}, $C_{2}{\times}A_{4}$ {\scriptsize(24,13)} \\
$X_{24ai}$ & $24$ & $13$ & $13\,5\,1^{(8)}\,(-1)^{(12)}\,(-7)^{(2)}$ & \emph{ab:} $C_{24}$ {\scriptsize(24,2)}, $C_{12}{\times}C_{2}$ {\scriptsize(24,9)}, $C_{6}{\times}C_{2}{\times}C_{2}$ {\scriptsize(24,15)}; \emph{non-ab:} $C_{3}{:}C_{8}$ {\scriptsize(24,1)}, $C_{3}{:}Q_{8}$ {\scriptsize(24,4)}, $C_{4}{\times}S_{3}$ {\scriptsize(24,5)}, $D_{24}$ {\scriptsize(24,6)}, $C_{2}{\times}(C_{3}{:}C_{4})$ {\scriptsize(24,7)}, $(C_{6}{\times}C_{2}){:}C_{2}$ {\scriptsize(24,8)}, $C_{3}{\times}D_{8}$ {\scriptsize(24,10)}, $C_{3}{\times}Q_{8}$ {\scriptsize(24,11)}, $S_{4}$ {\scriptsize(24,12)}, $C_{2}{\times}A_{4}$ {\scriptsize(24,13)}, $C_{2}{\times}C_{2}{\times}S_{3}$ {\scriptsize(24,14)} \\
$X_{24aj}$ & $24$ & $13$ & $13\,5^{(3)}\,1^{(2)}\,(-1)^{(12)}\,(-3)^{(6)}$ & \emph{non-ab:} $SL(2,3)$ {\scriptsize(24,3)}, $S_{4}$ {\scriptsize(24,12)} \\
$X_{24ak}$ & $24$ & $13$ & $13\,5^{(2)}\,1^{(5)}\,(-1)^{(12)}\,(-3)^{(3)}\,(-7)$ & \emph{ab:} $C_{12}{\times}C_{2}$ {\scriptsize(24,9)}, $C_{6}{\times}C_{2}{\times}C_{2}$ {\scriptsize(24,15)}; \emph{non-ab:} $C_{3}{:}Q_{8}$ {\scriptsize(24,4)}, $C_{4}{\times}S_{3}$ {\scriptsize(24,5)}, $D_{24}$ {\scriptsize(24,6)}, $C_{2}{\times}(C_{3}{:}C_{4})$ {\scriptsize(24,7)}, $(C_{6}{\times}C_{2}){:}C_{2}$ {\scriptsize(24,8)}, $C_{3}{\times}D_{8}$ {\scriptsize(24,10)}, $C_{3}{\times}Q_{8}$ {\scriptsize(24,11)}, $S_{4}$ {\scriptsize(24,12)}, $C_{2}{\times}A_{4}$ {\scriptsize(24,13)}, $C_{2}{\times}C_{2}{\times}S_{3}$ {\scriptsize(24,14)} \\
$X_{24al}$ & $24$ & $13$ & $13\,9\,1^{(5)}\,(-1)^{(12)}\,(-3)^{(5)}$ & \emph{ab:} $C_{12}{\times}C_{2}$ {\scriptsize(24,9)}, $C_{6}{\times}C_{2}{\times}C_{2}$ {\scriptsize(24,15)}; \emph{non-ab:} $C_{3}{:}Q_{8}$ {\scriptsize(24,4)}, $C_{4}{\times}S_{3}$ {\scriptsize(24,5)}, $D_{24}$ {\scriptsize(24,6)}, $C_{2}{\times}(C_{3}{:}C_{4})$ {\scriptsize(24,7)}, $(C_{6}{\times}C_{2}){:}C_{2}$ {\scriptsize(24,8)}, $C_{3}{\times}D_{8}$ {\scriptsize(24,10)}, $C_{3}{\times}Q_{8}$ {\scriptsize(24,11)}, $S_{4}$ {\scriptsize(24,12)}, $C_{2}{\times}A_{4}$ {\scriptsize(24,13)}, $C_{2}{\times}C_{2}{\times}S_{3}$ {\scriptsize(24,14)} \\
  $X_{24am}$ & $24$ & $14$ & $14\,2^{(4)}\,0^{(12)}\,(-2)^{(6)}\,(-10)$ & \emph{ab:} $C_{24}$ {\scriptsize(24,2)},
  $C_{12}{\times}C_{2}$ {\scriptsize(24,9)}, $C_{6}{\times}C_{2}{\times}C_{2}$ {\scriptsize(24,15)}; \emph{non-ab:}
  $C_{3}{:}C_{8}$ {\scriptsize(24,1)}, $C_{3}{:}Q_{8}$ {\scriptsize(24,4)}, $C_{4}{\times}S_{3}$ {\scriptsize(24,5)},
  $D_{24}$ {\scriptsize(24,6)}, $C_{2}{\times}(C_{3}{:}C_{4})$ {\scriptsize(24,7)}, $(C_{6}{\times}C_{2}){:}C_{2}$
  {\scriptsize(24,8)}, $C_{3}{\times}D_{8}$ {\scriptsize(24,10)}, $C_{3}{\times}Q_{8}$ {\scriptsize(24,11)}, $S_{4}$
  {\scriptsize(24,12)}, $C_{2}{\times}A_{4}$ {\scriptsize(24,13)}, $C_{2}{\times}C_{2}{\times}S_{3}$ {\scriptsize(24,14)}\\
$X_{24an}$ & $24$ & $14$ & $14\,6\,2^{(2)}\,0^{(12)}\,(-2)^{(6)}\,(-6)^{(2)}$ & \emph{ab:} $C_{24}$ {\scriptsize(24,2)}, $C_{12}{\times}C_{2}$ {\scriptsize(24,9)}, $C_{6}{\times}C_{2}{\times}C_{2}$ {\scriptsize(24,15)}; \emph{non-ab:} $C_{3}{:}C_{8}$ {\scriptsize(24,1)}, $C_{3}{:}Q_{8}$ {\scriptsize(24,4)}, $C_{4}{\times}S_{3}$ {\scriptsize(24,5)}, $D_{24}$ {\scriptsize(24,6)}, $C_{2}{\times}(C_{3}{:}C_{4})$ {\scriptsize(24,7)}, $(C_{6}{\times}C_{2}){:}C_{2}$ {\scriptsize(24,8)}, $C_{3}{\times}D_{8}$ {\scriptsize(24,10)}, $C_{3}{\times}Q_{8}$ {\scriptsize(24,11)}, $S_{4}$ {\scriptsize(24,12)}, $C_{2}{\times}A_{4}$ {\scriptsize(24,13)}, $C_{2}{\times}C_{2}{\times}S_{3}$ {\scriptsize(24,14)} \\
$X_{24ao}$ & $24$ & $14$ & $14\,2^{(5)}\,0^{(12)}\,(-2)^{(3)}\,(-6)^{(3)}$ & \emph{non-ab:} $SL(2,3)$ {\scriptsize(24,3)}, $S_{4}$ {\scriptsize(24,12)} \\
$X_{24ap}$ & $24$ & $14$ & $14\,4^{(3)}\,2^{(2)}\,0^{(6)}\,(-2)^{(9)}\,(-4)^{(3)}$ & \emph{non-ab:} $S_{4}$ {\scriptsize(24,12)}, $C_{2}{\times}A_{4}$ {\scriptsize(24,13)} \\
$X_{24aq}$ & $24$ & $14$ & $14\,4^{(2)}\,2^{(2)}\,0^{(9)}\,(-2)^{(9)}\,(-8)$ & \emph{ab:} $C_{12}{\times}C_{2}$ {\scriptsize(24,9)}, $C_{6}{\times}C_{2}{\times}C_{2}$ {\scriptsize(24,15)}; \emph{non-ab:} $C_{4}{\times}S_{3}$ {\scriptsize(24,5)}, $D_{24}$ {\scriptsize(24,6)}, $C_{2}{\times}(C_{3}{:}C_{4})$ {\scriptsize(24,7)}, $(C_{6}{\times}C_{2}){:}C_{2}$ {\scriptsize(24,8)}, $C_{3}{\times}D_{8}$ {\scriptsize(24,10)}, $S_{4}$ {\scriptsize(24,12)}, $C_{2}{\times}A_{4}$ {\scriptsize(24,13)}, $C_{2}{\times}C_{2}{\times}S_{3}$ {\scriptsize(24,14)} \\
$X_{24ar}$ & $24$ & $14$ & $14\,8\,2^{(2)}\,0^{(9)}\,(-2)^{(9)}\,(-4)^{(2)}$ & \emph{ab:} $C_{12}{\times}C_{2}$ {\scriptsize(24,9)}, $C_{6}{\times}C_{2}{\times}C_{2}$ {\scriptsize(24,15)}; \emph{non-ab:} $C_{4}{\times}S_{3}$ {\scriptsize(24,5)}, $D_{24}$ {\scriptsize(24,6)}, $C_{2}{\times}(C_{3}{:}C_{4})$ {\scriptsize(24,7)}, $(C_{6}{\times}C_{2}){:}C_{2}$ {\scriptsize(24,8)}, $C_{3}{\times}D_{8}$ {\scriptsize(24,10)}, $S_{4}$ {\scriptsize(24,12)}, $C_{2}{\times}A_{4}$ {\scriptsize(24,13)}, $C_{2}{\times}C_{2}{\times}S_{3}$ {\scriptsize(24,14)} \\
  $X_{24as}$ & $24$ & $17$ & $17\,1^{(9)}\,(-1)^{(12)}\,(-7)^{(2)}$ & \emph{ab:} $C_{24}$ {\scriptsize(24,2)},
  $C_{12}{\times}C_{2}$ {\scriptsize(24,9)}, $C_{6}{\times}C_{2}{\times}C_{2}$ {\scriptsize(24,15)}; \emph{non-ab:}
  $C_{3}{:}C_{8}$ {\scriptsize(24,1)}, $SL(2,3)$ {\scriptsize(24,3)}, $C_{3}{:}Q_{8}$ {\scriptsize(24,4)},
  $C_{4}{\times}S_{3}$ {\scriptsize(24,5)}, $D_{24}$ {\scriptsize(24,6)}, $C_{2}{\times}(C_{3}{:}C_{4})$
  {\scriptsize(24,7)}, $(C_{6}{\times}C_{2}){:}C_{2}$ {\scriptsize(24,8)}, $C_{3}{\times}D_{8}$ {\scriptsize(24,10)},
  $C_{3}{\times}Q_{8}$ {\scriptsize(24,11)}, $S_{4}$ {\scriptsize(24,12)}, $C_{2}{\times}A_{4}$ {\scriptsize(24,13)},
  $C_{2}{\times}C_{2}{\times}S_{3}$ {\scriptsize(24,14)} \\
$X_{24at}$ & $24$ & $18$ & $18\,2^{(3)}\,0^{(12)}\,(-2)^{(6)}\,(-6)^{(2)}$ & \emph{ab:} $C_{24}$ {\scriptsize(24,2)}, $C_{12}{\times}C_{2}$ {\scriptsize(24,9)}, $C_{6}{\times}C_{2}{\times}C_{2}$ {\scriptsize(24,15)}; \emph{non-ab:} $C_{3}{:}C_{8}$ {\scriptsize(24,1)}, $SL(2,3)$ {\scriptsize(24,3)}, $C_{3}{:}Q_{8}$ {\scriptsize(24,4)}, $C_{4}{\times}S_{3}$ {\scriptsize(24,5)}, $D_{24}$ {\scriptsize(24,6)}, $C_{2}{\times}(C_{3}{:}C_{4})$ {\scriptsize(24,7)}, $(C_{6}{\times}C_{2}){:}C_{2}$ {\scriptsize(24,8)}, $C_{3}{\times}D_{8}$ {\scriptsize(24,10)}, $C_{3}{\times}Q_{8}$ {\scriptsize(24,11)}, $S_{4}$ {\scriptsize(24,12)}, $C_{2}{\times}A_{4}$ {\scriptsize(24,13)}, $C_{2}{\times}C_{2}{\times}S_{3}$ {\scriptsize(24,14)} \\
$X_{24au}$ & $24$ & $19$ & $19\,1^{(9)}\,(-1)^{(9)}\,(-3)^{(3)}\,(-5)^{(2)}$ & \emph{ab:} $C_{12}{\times}C_{2}$ {\scriptsize(24,9)}, $C_{6}{\times}C_{2}{\times}C_{2}$ {\scriptsize(24,15)}; \emph{non-ab:} $C_{4}{\times}S_{3}$ {\scriptsize(24,5)}, $D_{24}$ {\scriptsize(24,6)}, $C_{2}{\times}(C_{3}{:}C_{4})$ {\scriptsize(24,7)}, $(C_{6}{\times}C_{2}){:}C_{2}$ {\scriptsize(24,8)}, $C_{3}{\times}D_{8}$ {\scriptsize(24,10)}, $S_{4}$ {\scriptsize(24,12)}, $C_{2}{\times}A_{4}$ {\scriptsize(24,13)}, $C_{2}{\times}C_{2}{\times}S_{3}$ {\scriptsize(24,14)} \\
$X_{24av}$ & $24$ & $20$ & $20\,2^{(3)}\,0^{(9)}\,(-2)^{(9)}\,(-4)^{(2)}$ & \emph{ab:} $C_{12}{\times}C_{2}$ {\scriptsize(24,9)}, $C_{6}{\times}C_{2}{\times}C_{2}$ {\scriptsize(24,15)}; \emph{non-ab:} $C_{4}{\times}S_{3}$ {\scriptsize(24,5)}, $D_{24}$ {\scriptsize(24,6)}, $C_{2}{\times}(C_{3}{:}C_{4})$ {\scriptsize(24,7)}, $(C_{6}{\times}C_{2}){:}C_{2}$ {\scriptsize(24,8)}, $C_{3}{\times}D_{8}$ {\scriptsize(24,10)}, $S_{4}$ {\scriptsize(24,12)}, $C_{2}{\times}A_{4}$ {\scriptsize(24,13)}, $C_{2}{\times}C_{2}{\times}S_{3}$ {\scriptsize(24,14)} \\
  $X_{24aw}$ & $24$ & $21$ & $21\,1^{(6)}\,(-1)^{(12)}\,(-3)^{(5)}$ & all 15 groups of order $24$ \\
$X_{24ax}$ & $24$ & $22$ & $22\,0^{(12)}\,(-2)^{(11)}$ & all 15 groups of order $24$ \\
\midrule
  $X_{28a}$ & $28$ & $25$ & $25\,1^{(7)}\,(-1)^{(14)}\,(-3)^{(6)}$ & all 4 groups of order $28$ \\
$X_{28b}$ & $28$ & $26$ & $26\,0^{(14)}\,(-2)^{(13)}$ & all 4 groups of order $28$ \\
\midrule
$X_{30a}$ & $30$ & $17$ & $17\,5^{(5)}\,(-1)^{(15)}\,(-3)^{(9)}$ & none (not a Cayley graph) \\
\end{longtable}
\endgroup

\section{The \texorpdfstring{$\GL$}{GL} and \texorpdfstring{$\GU$}{GU} graphs of Pantangi and Sin}
\label{app:pansin-glgu}

This appendix describes the computation behind
Proposition~\ref{prop:pansin-glgu}.  Let $G$ be $\GL(2,q)$
or $\GU(2,q)$ with $q$ odd, let $t=-I$, and let $X=\Cay(G,S)$
be one of the four graphs of the proposition; so $S$ is the connection
set of \cite[Theorem~4.8 or~4.15]{PanSin2026}, respectively the set
$\mathfrak T$ for the graph $\Gamma'$.  In every case $S$ is a union of
conjugacy classes with $St=S$ and $t\notin S$.  Hence the pairs
$\{g,tg\}$ are non-adjacent twins, and, exactly as in the last
paragraph of the proof of Lemma~\ref{lem:pansin-structure}, whether $h$
is adjacent to $g$ depends only on the images of $g$ and $h$ in
$G/\langle t\rangle$, so distinct pairs are joined by all four possible
edges or by none.  Thus $X=Y[2K_1]$, where
\[
    Y=\Cay(G/\langle t\rangle,\pi(S))
\]
and $\pi$ is the quotient map.

Lemma~\ref{lem:twin-reduction} requires in addition that every twin
class of $X$ has size two.  We have not established this for the
$\GL$ and $\GU$ families in general; for the four
graphs of the proposition we computed the twin classes directly, by
grouping the vertices by their neighbourhoods, and they are exactly the
pairs $\{g,tg\}$.  The lemma therefore applies: each of the four graphs
is a Cayley graph of an abelian group if and only if $\Aut(Y)$ contains
an abelian regular subgroup.  We remark that the computed spectra of
the four quotients consist of odd integers, so the perfect state
transfer in these graphs is the odd case of
Proposition~\ref{prop:doubling}, complementing the even case seen
for $\mathrm{SL}(2,q)$ in Remark~\ref{rem:pansin-parity}.

The four quotient graphs $Y$ have $24$, $48$, $240$ and $360$ vertices,
and their automorphism groups have orders $1152$, $73728$, $115200$ and
$172800$, respectively.  For $\Gamma'$ the quotient is
\[
  \Cay(S_4,\{\text{all $17$ elements of order $2$ or $3$}\}),
\]
with spectrum $\{17,5,1^9,(-1)^4,(-3)^9\}$; for the $\GU(2,3)$
graph it is a Cayley graph of $\mathrm{SmallGroup}(48,30)$ with
spectrum $\{31,7,3^9,1^{18},(-3)^6,(-5)^{13}\}$.

In each of the four cases an exhaustive search in SageMath \cite{sage},
using GAP for the group-theoretic steps, shows that $\Aut(Y)$ contains
no abelian regular subgroup.  The search is organized as follows.
Suppose that $A\le\Aut(Y)$ is abelian and regular, and fix a prime $p$
dividing $|V(Y)|$.  By Cauchy's theorem, $A$ contains an element $g$ of
order $p$, and $g$ is fixed-point-free, because a regular group is
semiregular.  Since $A$ is abelian, $A\le C:=C_{\Aut(Y)}(g)$, and $C$
is transitive, because it contains $A$.  After conjugating $A$, we may
take $g$ to be one of the conjugacy class representatives of the
fixed-point-free elements of order $p$ in $\Aut(Y)$ whose centralizer
is transitive.  For each such representative we search its centralizer
$C$ for abelian regular subgroups, either by enumerating the subgroups
of $C$ of order $|V(Y)|$ directly, or by running
\textsf{IsomorphicSubgroups} over the isomorphism types of abelian
groups of order $|V(Y)|$, and we test the abelian subgroups so found
for transitivity.  No transitive abelian subgroup occurs in any of the
four cases, which proves Proposition~\ref{prop:pansin-glgu}.

The proof of Theorem~\ref{thm:pansin-sl} does not extend to the
$\GL$ and $\GU$ families: the quotient
$G/\langle t\rangle$ is not simple, its two-sided translation group is
then not primitive, and the socle argument breaks at the first step.
This is the obstacle behind Problem~\ref{prob:pansin-stability}.

\end{document}